\documentclass[12pt]{book}
\usepackage[a4paper,twoside,hmargin=2.5cm,vmargin=3.35cm]{geometry}
\usepackage{amsthm,amsmath,amsfonts,makeidx,url,epsfig}

\newtheorem{theorem}{Theorem}
\newtheorem{proposition}{Proposition}

\newtheorem{remark}{Remark}
\newtheorem*{fact}{Fact}

\newcommand*{\aut}{\mathrm{Aut}(q,n)}
\newcommand*{\autbin}{\mathrm{Aut}(2,n)}
\newcommand*{\autj}{\mathrm{Aut}(n,w)}
\newcommand*{\stabj}{\mathrm{Aut}_{W}(n,w)}
\newcommand*{\stab}{\mathrm{Aut}_{\mathbf{0}}(q,n)}
\newcommand*{\stabbin}{\mathrm{Aut}_{\mathbf{0}}(2,n)}
\newcommand*{\transp}{\mathsf{T}}
\newcommand*{\q}{\mathbf{q}}
\newcommand*{\Real}{\mathbb{R}}
\newcommand*{\Complex}{\mathbb{C}}
\newcommand*{\dd}{\overline{d}}
\newcommand*{\E}{\mathbb{E}}
\newcommand*{\roundup}[1]{\left\lceil #1\right\rceil}
\newcommand*{\rounddown}[1]{\left\lfloor #1\right\rfloor}
\newcommand*{\sumentries}[1]{\one^{\transp}#1\one}
\newcommand*{\x}{\mathbf{x}}
\newcommand*{\y}{\mathbf{y}}

\newcommand*{\w}{\mathbf{w}}
\newcommand*{\zero}{\mathbf{0}}
\newcommand*{\one}{\mathbf{1}}
\newcommand*{\diag}{\mathrm{diag}}
\newcommand*{\Diag}{\mathrm{Diag}}
\newcommand*{\trace}{\mathrm{tr}}
\newcommand*{\psd}{\succeq 0}
\renewcommand*{\P}{\mathcal{P}}
\newcommand*{\notion}[1]{\index{#1}\emph{#1}}
\providecommand{\noopsort}[1]{}

\makeindex

\begin{document}

\begin{titlepage}
\thispagestyle{empty}
\vspace*{\fill}
\vspace*{\fill}
\begin{center}
\large \textbf{Matrix Algebras and Semidefinite Programming Techniques for Codes}
\end{center}
\vfill
\begin{center}
\large Dion Gijswijt
\end{center}
\vspace*{\fill}
\vspace*{\fill}
\pagebreak
\thispagestyle{empty}
\cleardoublepage
\thispagestyle{empty}
\begin{center}\Large\textbf{Matrix Algebras and Semidefinite Programming Techniques for Codes}\end{center}
\vfill
\begin{center}
\large \textsc{Academisch Proefschrift}
\end{center}
\vfill
\begin{center}
ter verkrijging van de graad van doctor\\ aan de Universiteit van Amsterdam\\ op gezag van de Rector Magnificus prof. mr. P.F. van der Heijden\\ten overstaan van een door het college voor promoties ingestelde\\commissie, in het openbaar te verdedigen in de Aula der Universiteit\\ \ \\op donderdag 22 september 2005, te 12.00 uur
\vfill
\end{center}
\begin{center}
door
\end{center}
\smallskip
\begin{center}
\textbf{\large Dion Camilo Gijswijt}
\end{center}
\smallskip
\begin{center}
geboren te Bunschoten.
\end{center}
\vfill
\end{titlepage}
\thispagestyle{empty}
\begin{flushleft}
\begin{tabbing}
\textbf{Promotiecommissie}\\
\\
Overige leden:\hspace{1cm}\=\kill
Promotor:\>Prof. dr. A. Schrijver\\
\\
Overige leden:\>Prof. dr. A.E. Brouwer\\
\>Prof. dr. G. van der Geer\\
\>Prof. dr. T.H. Koornwinder\\
\>Prof. dr.ir. H.C.A. van Tilborg\\
\\
Faculteit der Natuurwetenschappen, Wiskunde en Informatica
\end{tabbing}
\end{flushleft}
\vspace*{10cm}
This research was supported by the Netherlands Organisation for Scientific Research (NWO) under project number 613.000.101.
\vspace*{1cm}
\begin{flushright}
\epsfxsize=12cm
\end{flushright}
\vspace*{-6cm}
\begin{flushleft}
\epsfxsize=12cm
\epsfbox{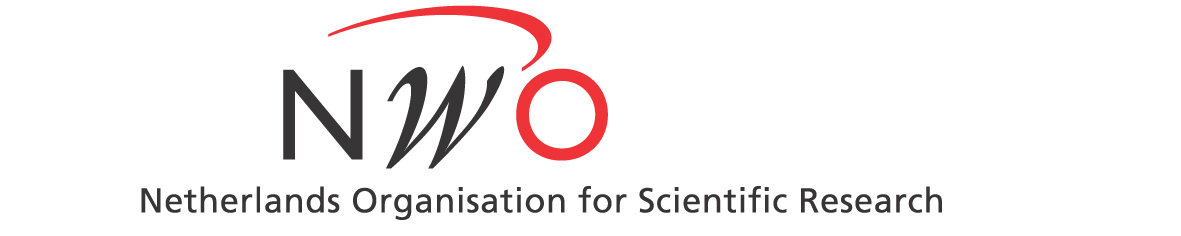}
\end{flushleft}
\newpage
\thispagestyle{empty}
\vspace*{\fill}
\begin{center}
\large \textsl{voor Violeta}
\end{center}
\vspace*{\fill}
\newpage
\thispagestyle{empty}
\cleardoublepage
\pagenumbering{roman}
\tableofcontents
\nocite{*}

\chapter{Introduction}
\pagenumbering{arabic}
In this thesis, we consider codes of length $n$ over an alphabet of $q$ symbols. We give new upper bounds on the maximum size $A_q(n,d)$ of codes with minimum distance $d$, and new lower bounds on the minimum size $K_q(n,r)$ of codes of covering radius $r$. The bounds are based on semidefinite programming and on an explicit block diagonalisation of the (non-commutative) Terwilliger algebra of the Hamming scheme. Our methods can be seen as a refinement of Delsarte's linear programming approach and are related to the theory of matrix cuts. They build upon the recent work of Schrijver \cite{Lexcodes} for binary codes.
 
\section{Codes}
A \notion{code} is a collection of \emph{words} of some fixed length, for example the collection of all six letter words in a dictionary. However, the words need not have any meaning. They are merely concatenations of symbols chosen from a fixed set called the \notion{alphabet}. Other examples of codes are: all possible strings of eight binary digits, a set of bets in a football pool, or a collection of DNA sequences. Here the alphabets are $\{0,1\}$, $\{\mathrm{Win},\mathrm{Lose},\mathrm{Tie}\}$ and $\{\mathrm{A},\mathrm{C},\mathrm{T},\mathrm{G}\}$ respectively.

It is often important to know how similar two words are. This can be measured by their \emph{Hamming distance}. By definition, this is the number of positions in which the two words differ. Suppose for example that we want to transmit information over a noisy communication channel. The letters in a transmitted word each have some small chance of being changed into a different letter. At the receiving end, we would like to be able to recover the original message (if not too many letters are erroneous). This can be achieved by using a code in which any two words have distance at least $d=2e+1$ for some integer $e$. If we only transmit words belonging to this code, it is always possible to recover the sent code word if at most $e$ errors are introduced during transmission. The received word is interpreted as the code word that is the closest match. If we aim for a highest possible information rate, we should maximize the number of words in the code, under the condition that any two code words have distance at least $d$. Geometrically, this means that we want to pack a maximum number of spheres (`balls' would be more accurate) of radius $r$ inside the \emph{Hamming space} consisting of all $q^n$ words of length $n$. Here the chosen code words correspond to the centers of the spheres. This leads to the following central question in coding theory.
\begin{center}
\emph{
What is the maximum cardinality of a code of word length $n$, in which\\
any two words have distance at least $d$?}
\end{center}
When the alphabet consists of $q$ symbols, this maximum is denoted by $A_q(n,d)$. The number $A_q(n,d)$ can also be seen as the \emph{stability number} of a graph. Let $G$ be the graph with the set of all $q^n$ words as vertices, and two words are joined by an edge if their distance is less than $d$. Then the maximum size of a set of vertices, no two of which are joined by an edge, equals $A_q(n,d)$.

The problem of determining $A_q(n,d)$ is hard in general and we will have to be satisfied with lower and upper bounds. One major field of research is to find explicit examples of (families of) good codes. In this thesis we will address the converse problem and give upper bounds on the numbers $A_q(n,d)$. In the case $d=2e+1$ the geometric idea of packing spheres already gives an upper bound. Since the spheres are disjoint, their `volumes' should add up to a number that is at most $q^n$. This gives an upper bound on $A_q(n,d)$ called the \emph{sphere packing bound}. 

\section{The Delsarte bound}
Currently, many of the best bounds known, are based on Delsarte's linear programming approach \cite{Delsarte}. When viewed from the right perspective, the work in this thesis can be seen as a refinement of this method. Let us give a very rough sketch of Delsarte's method. 

As is often the case in mathematics, we first seem to make the problem harder. Instead of optimizing the cardinality of a code directly, we associate to each code $C$ a symmetric matrix of which the rows and columns correspond to all $q^n$ possible code words. The matrix is constructed by putting a $1$ in those positions where both the row and the column of the matrix belong to $C$, and a $0$ in all other positions. The size of the code can be recovered from the matrix by dividing the total number of ones by the number of ones on the diagonal. Although we have no good grasp of the set of matrices that arise this way, they share some important and elegant abstract properties:
\begin{itemize}
\item the matrix has zeros in positions indexed by a row and column that are at distance $1,2,\ldots, d-1$,
\item the matrix is \emph{positive semidefinite}: it has no negative \emph{eigenvalues}. \end{itemize}
We enlarge our set of matrices from those associated to codes, to include \emph{all} symmetric matrices sharing the two given properties. The resulting \emph{relaxation} is much `smoother' and has a very clear description which allows more efficient optimization. Of course the magical part is, that optimizing over this larger set gives a good approximation of the original problem! This bound was given in the more general setting of bounding the stability number of a graph by Lov\'asz \cite{lovasztheta}. It can be calculated using semidefinite programming in time bounded by a polynomial in the number of vertices of the graph. 

In the coding setting, this will not suffice since the size of the matrices is prohibitively large. Even for codes of length $n=20$, we have to deal with matrices of more then a million rows and columns. However, the problem admits a very large symmetry group. It turns out that we can use these symmetries to our advantage to ---dramatically--- reduce the complexity of the problem. We may restrict ourselves to only those matrices, that are invariant under the full group of symmetries. These matrices live in a low-dimensional commutative subalgebra called the \emph{Bose-Mesner algebra} of the \emph{Hamming scheme}. Diagonalising this algebra reduces the huge optimization problem to a simple \emph{linear} program of only $n$ variables! The resulting linear programming bound (adding the constraint that the matrix is nonnegative) is due to Delsarte.

\section{Overview of the thesis}
In this thesis, we give tighter bounds for codes by essentially isolating more properties satisfied by the zero-one matrices associated to a code. More accurately, we associate to each code $C$ \emph{two} matrices. Both matrices are obtained by summing zero-one matrices corresponding to certain permutations of the code $C$. This allows to include constraints that come from \emph{triples} of code words, instead of pairs. This method was initiated recently by Schrijver \cite{Lexcodes} to obtain bounds for binary error correcting codes, resulting in a large number of improved upper bounds.

The main result in this thesis is to generalize the methods to include non-binary codes. The primary issue that we need to deal with, is how to exploit the remaining symmetries to obtain a semidefinite program of a size that is polynomially bounded by the word length $n$. This is the most technical part of the thesis and requires an explicit \emph{block diagonalisation} of the \emph{Terwilliger algebra} of the nonbinary Hamming scheme. Such a block diagonalisation is described in Chapter \ref{CH:hammingschemes}. It uses the block diagonalisation of the Terwilliger algebra of the binary Hamming scheme found by Schrijver, which we will describe as well. 

In Chapter \ref{CH:errorcodes} we apply our methods to obtain a semidefinite programming bound for nonbinary codes. Computationally, we have found a large number of improved upper bounds for $q=3,4,5$, which we have tabulated in the final section.

In Chapter \ref{CH:coveringcodes} we discuss covering codes. The problem here is to cover the Hamming space with as few spheres as possible. When the spheres have radius $r$, this minimum number of required spheres is denoted by $K_{q}(n,r)$. We give new linear and semidefinite programming bounds on $K_q(n,r)$. For $q=4,5$ we obtain several improved lower bounds on $K_q(n,r)$. 

In Chapter \ref{CH:matrixcuts} we relate our coding bounds to the general theory of matrix cuts for obtaining improved relaxations of $0$--$1$ polytopes. It is shown that the bound for error correcting codes is stronger than the bound obtained from a single iteration of the $N_+$ operator applied to the modified theta body of the graph on all words in which two words are joined by an edge if they are at distance smaller then $d$.  
  
\chapter{Preliminaries}\label{CH:Preliminaries}
This thesis is largely self-contained and most results are derived from explicit constructions. However, some theory is desirable for putting them into the right perspective and relating them to the body of mathematics to which they connect. In this chapter we give some definitions and basic facts. 
After giving some general notation in the first section, we introduce matrix $\ast$-algebras in the second section, which are an important tool throughout the thesis. The main (classical) theorem says (roughly) that any matrix $\ast$-algebra is isomorphic to a direct sum of full matrix $\ast$-algebras. In the third section we describe semidefinite programming. The bounds we derive for codes, are defined as the optimum of certain semidefinite programs and can be computed efficiently. Finally, we recall the basics of association schemes. In particular we describe the Delsarte bound on the maximum size of cliques in association schemes.  

\section{Notation}
For positive integers $n,m$ and a set $R$ (usually $R=\Complex,\Real$), we denote by $R^{n\times m}$ the set of $n$ by $m$ matrices with entries in $R$ and by $R^n$ the set of (column) vectors of length $n$. When $R$ is a ring, we define matrix addition and multiplication of matrices (with compatible dimensions) as usual. Frequently, the rows and columns correspond to the elements of some given finite sets $X$ and $Y$. When we want to explicitly index the rows and columns of the matrix using these sets, we will write $R^{X\times Y}$ for the set of matrices with rows indexed by $X$ and columns indexed by $Y$. The $i$-th row of a matrix $A$ is denoted by $A_i$ and the entry in row $i$ and column $j$ by $A_{i,j}$. The \notion{transpose} of an $X\times Y$ matrix is the $Y\times X$ matrix $A^{\transp}$, where $A^{\transp}_{i,j}=A_{j,i}$ for $i\in Y$, $j\in X$. When $|Y|=1$, we often identify the matrices in $R^{X\times Y}$ and the vectors in $R^X$.

For finite sets $X,Y$, the all-one vector in $R^X$ is denoted by $\one$. The $X\times Y$ all-one matrix is denoted by $J$, the all-zero matrix by $0$ and the $X\times X$ identity matrix by $I$. The sets $X$ and $Y$ will be clear from the context.

Given a matrix $A\in R^{X\times X}$, we define $\diag(A)$ to be the vector $a\in R^X$ of diagonal elements of $A$, that is $a_i:=A_{i,i}$ for $i\in X$. The \notion{trace} of $A$ is the sum of the diagonal elements of $A$ and is denoted $\trace A$. So $\trace A=\one^{\transp}\diag(A)$. We mention the useful fact that for matrices $A\in R^{k\times l}$ and $B\in R^{l\times k}$ the following identity holds:
\begin{equation}
\trace (AB)=\trace (BA).
\end{equation}
Given a vector $a\in R^X$, we denote by $\Diag (a)$ the diagonal matrix $A\in R^{X\times X}$ with $\diag(A)=a$.

For a subset $S\subseteq X$ we denote by $\chi^S$ the vector in $R^X$ definied by
\begin{equation}
(\chi^S)_i:=\begin{cases}1&\text{if $i\in S$}\\0&\text{otherwise}.\end{cases}
\end{equation}
For a vector $x$ and a set $S$, we define 
\begin{equation}
x(S):=\sum_{i\in S} x_i.
\end{equation}  

For a matrix $A\in \Complex^{X\times Y}$, the \notion{conjugate transpose} of $A$ is denoted by $A^{\ast}$. That is $A^{\ast}_{i,j}=\overline{A_{j,i}}$ for $i\in X$ and $j\in Y$, where $\overline{z}$ is the complex conjugate of a complex number $z$. A square matrix $A$ is called \notion{normal} if $A^{\ast}A=AA^{\ast}$, \notion{hermitian} if $A^{\ast}=A$ and \notion{unitary} if $A^{\ast}A=AA^{\ast}=I$. 

For $A,B\in \Complex^{X\times Y}$, we define 
\begin{equation}
\left<A,B\right>:=\trace (AB^{\ast})=\sum_{i\in X,j\in Y} A_{i,j}\overline{B_{i,j}}.
\end{equation}
This is the standard complex \notion{inner product} on $\Complex^{X\times Y}$. Observe that
\begin{equation}
\left<A,J\right>=\one^{\transp}A\one.
\end{equation}

For matrices $A\in \Complex^{X_1\times Y_1}$ and $B\in \Complex^{X_2\times Y_2}$, we denote by $A\otimes B$ the \notion{tensor product} of $A$ and $B$ defined as the $(X_1\times X_2)\times (Y_1\times Y_2)$ matrix given by
\begin{equation}
(A\otimes B)_{(i,i'),(j,j')}:=A_{i,j}B_{i',j'}.
\end{equation}   

\section{Matrix $\ast$-algebras}
In this section we consider algebras of matrices. For general background on linear algebra we refer the reader to \cite{HornJohnson,Langlinalg}.

A \notion{matrix $\ast$-algebra} is a nonempty set of matrices $\mathcal{A}\subseteq \Complex^{n\times n}$ that is closed under addition, scalar multiplication, matrix multiplication and under taking the conjugate transpose. A matrix $\ast$-algebra is a special case of a finite dimensional $C^{\ast}$-algebra. Trivial examples are the full matrix algebra $\Complex^{n\times n}$ and the \notion{zero algebra} $\{0\}$.

Most of the matrix $\ast$-algebras that we will encounter in this thesis are of a special type. They are the set of matrices that commute with a given set of \notion{permutation matrices}. More precisely, we have the following. 

Let $G\subseteq S_n$ be a subgroup of the symmetric group on $n$ elements. To every element $\sigma\in G$ we associate the permutation matrix $M_\sigma\in\Complex^{n\times n}$ given by
\begin{equation}
(M_\sigma)_{i,j}:=\begin{cases}1&\text{if $\sigma(j)=i$},\\0&\text{otherwise.}\end{cases}
\end{equation}
Observe that 
\begin{equation}
M_{\sigma}^{\ast}=M_{\sigma}^{\transp}=M_{\sigma^{-1}}.
\end{equation}
The map $\sigma\mapsto M_{\sigma}$ defines a \emph{representation} of $G$. This means that for all $\sigma,\tau\in G$ we have
\begin{equation}
M_{\tau\sigma}=M_{\tau}M_{\sigma}\quad\text{and}\quad M_{\sigma^{-1}}=M_{\sigma}^{-1}.
\end{equation}
Here $\tau\sigma$ denotes the permutation $(\tau\sigma) (i):=\tau(\sigma(i))$. We define the \notion{centralizer algebra} (see \cite{BannaiIto}) of $G$ to be the set $\mathcal{A}$ of matrices that are invariant under permuting the rows and columns by elements of $G$. That is,
\begin{equation}\label{centralizeralgebra}
\mathcal{A}:=\{A\in\Complex^{n\times n}\mid M_{\sigma}^{-1}AM_{\sigma}=A\quad\text{for all $\sigma\in G$}\}.
\end{equation}
If we denote by $\mathcal{B}$ the matrix $\ast$-algebra spanned by the set of permutation matrices $\{M_{\sigma},\ \sigma\in G\}$, then $\mathcal{A}$ is also called the \notion{commutant algebra} of $\mathcal{B}$: the algebra of matrices that commute with all the elements of $\mathcal{B}$. To see that the set $\mathcal{A}$ is indeed a matrix $\ast$-algebra, we first observe that it is closed under addition and scalar multiplication. That $\mathcal{A}$ is closed under matrix multiplication and taking the conjugate transpose follows from
\begin{eqnarray}
M_{\sigma}^{-1}ABM_{\sigma}&=&M_{\sigma}^{-1}AM_{\sigma}M_{\sigma}^{-1}BM_{\sigma}=AB,\\
M_{\sigma}^{-1}A^{\ast}M_{\sigma}&=&(M_{\sigma}^{-1}AM_{\sigma})^{\ast}=A^{\ast}\nonumber
\end{eqnarray}
for any $A,B\in \mathcal{A}$.

One of the special features of $\mathcal{A}$ is that it contains the identity and is spanned by a set of zero-one matrices whose supports partition $\{1,\ldots,n\}\times \{1,\ldots,n\}$ (that is, it is the algebra belonging to a \notion{coherent configuration}, see \cite{coherent}). These matrices have a combinatorial interpretation. Indeed, from (\ref{centralizeralgebra}) it follows that
\begin{equation}
\text{$A\in\mathcal{A}$\  if and only if $A_{i,j}=A_{\sigma (i),\sigma (j)}$ for all $i,j\in X$.}
\end{equation}
Hence $\mathcal{A}$ is spanned by zero-one matrices $A_1,\ldots, A_t$, where the supports of the $A_i$ are the orbits of $\{1,\ldots,n\}\times \{1,\ldots,n\}$ under the action of $G$, called the \notion{orbitals} of $G$.

The following structure theorem is one of the main motivations for this thesis. It allows to give a matrix $\ast$-algebra a simple appearance by performing a unitary transformation (a block diagonalisation). We will not use this theorem, but rather give explicit block diagonalisations for the matrix $\ast$-algebras under consideration.
\begin{theorem}\label{BD}
Let $\mathcal{A}\subseteq \Complex^{n\times n}$ be a matrix $\ast$-algebra containing the identity matrix $I$. Then there exists a unitary $n\times n$ matrix $U$ and positive integers $p_1,\ldots,p_m$ and $q_1,\ldots, q_m$ such that $U^{\ast}\mathcal{A}U$ consists of \emph{all} block diagonal matrices 
\begin{equation}
\begin{pmatrix}
C_1&0&\cdots&0\\
0&C_2&\cdots&0\\
\vdots&\vdots&\ddots&0\\
0&0&\cdots&C_m
\end{pmatrix}
\end{equation}
where each $C_k$ is a block diagonal matrix 
\begin{equation}
\begin{pmatrix}
B_k&0&\cdots&0\\
0&B_k&\cdots&0\\
\vdots&\vdots&\ddots&0\\
0&0&\cdots&B_k
\end{pmatrix}
\end{equation}
with $q_k$ identical blocks $B_k\in \Complex^{p_k\times p_k}$ on the diagonal. 
\end{theorem}
Observe that the numbers $p_1,\ldots,p_m$ and $q_1,\ldots,q_m$ satisfy 
\begin{eqnarray}
q_1p_1+q_2p_2+\cdots+q_mp_m&=&n,\\
p_1^2+p_2^2+\cdots+p_m^2&=&\dim \mathcal{A}.\nonumber
\end{eqnarray}
We call the algebra $U^{\ast}\mathcal{A}U$ a \notion{block diagonalisation} of $\mathcal{A}$. This theorem was proved in \cite{bdref} by using (a special case of) the Wedderburn-Artin theorem (see also \cite{wedderburn}, \cite{Roerdam}). However, we will present a self-contained proof here.

A well-known instance is when $\mathcal{A}\subseteq \Complex^{n\times n}$ is commutative. This occurs for example when $\mathcal{A}$ is the Bose--Mesner algebra of an association scheme. In the commutative case we must have $p_1=\ldots=p_m=1$, since for any $p\geq 2$ the algebra $\Complex ^{p\times p}$ is non-commutative. The theorem then says that the matrices in $\mathcal{A}$ can be simultaneously diagonalised: 
\begin{equation}
U^{\ast}\mathcal{A}U=\{x_1I_1+x_2I_2+\cdots+x_mI_m\mid x\in \Complex^m\},
\end{equation}
where for each $k$ the matrix $I_k\in\Complex^n$ is a zero-one diagonal matrix with $q_k$ ones, and $I_1+\cdots+I_m=I$. The rest of this section is devoted to proving Theorem \ref{BD}.

We first introduce some more notation. For two square matrices $A\in \Complex^{n\times n}$ and $B\in\Complex^{m\times m}$, we define their \notion{direct sum} $A\oplus B\in \Complex^{(n+m)\times (n+m)}$ by
\begin{equation}
A\oplus B:=\begin{pmatrix}A&0\\0&B\end{pmatrix}.
\end{equation}
For two matrix $\ast$-algebras $\mathcal{A}$ and $\mathcal{B}$, we define their direct sum by
\begin{equation}
\mathcal{A}\oplus \mathcal{B}:=\{A\oplus B\mid A\in\mathcal{A},B\in \mathcal{B}\}.
\end{equation}
This is again a matrix $\ast$-algebra. For a positive integer $t$, we define
\begin{equation}
t\odot \mathcal{A}:=\{t\odot A\mid A\in \mathcal{A}\},
\end{equation}
where $t\odot A$ denotes the iterated direct sum $\oplus_{i=1}^t A$.

We call two square matrices $A,B\in\Complex^{n\times n}$ \emph{equivalent} if there exists a unitary matrix $U$ such that $B=U^{\ast}AU$. We extend this to matrix $\ast$-algebras and call two matrix $\ast$-algebras $\mathcal{A}$ and $\mathcal{B}$ equivalent if $\mathcal{B}=U^{\ast}\mathcal{A}U$ for some unitary matrix $U$.

Theorem \ref{BD} can thus be expressed by saying that every matrix $\ast$-algebra $\mathcal{A}$ containing the identity matrix is equivalent to a matrix $\ast$-algebra of the form
\begin{equation}
\bigoplus_{i=1}^m (q_i\odot \Complex^{p_i\times p_i}).
\end{equation}

We start by considering the commutative case. We first introduce some more notions. Let $\mathcal{A}\subseteq \Complex^{n\times n}$ and let $V\subseteq \Complex^n$ be a linear subspace. We say that $V$ is \emph{$\mathcal{A}$-invariant} when $Av\in V$ for every $A\in \mathcal{A}$ and every $v\in V$. Observe that if $\mathcal{A}$ is closed under taking the conjugate transpose, also the orthoplement 
\begin{equation}
V^{\perp}:=\{v\in \Complex^n\mid \left<v,u\right>=0\text{ for all $u\in V$}\}
\end{equation}
is $\mathcal{A}$-invariant. Indeed, for every $u\in V$, $v\in V^{\perp}$ and $A\in\mathcal{A}$ we have
\begin{equation}
\left<Av,u \right>=\left<v,A^{\ast}u\right>=0,
\end{equation}
since $A^{\ast}u\in V$.

A nonzero vector $v\in \Complex^n$ is called a \notion{common eigenvector} for $\mathcal{A}$ when $\Complex v$ is $\mathcal{A}$-invariant. We recall the following basic fact from linear algebra.
\begin{fact}
Let $V$ be a complex linear space of finite, nonzero dimension, and let $A:V\longrightarrow V$ be a linear map. Then there exist $\lambda\in \Complex$ and $v\in V\setminus \{0\}$ such that $Av=\lambda v$.  
\end{fact} 
In particular, this implies that when $A\in\Complex^{n\times n}$ and $V\subseteq \Complex^n$ is $\{A\}$-invariant, there exists an eigenvector of $A$ that belongs to $V$. We are now ready to prove the following theorem.
\begin{theorem}
Let $\mathcal{A}\subseteq \Complex^{n\times n}$ be a \emph{commutative} matrix $\ast$-algebra and let $V\subseteq \Complex^n$ be an $\mathcal{A}$-invariant subspace. Then $V$ has an orthonormal basis of common eigenvectors for $\mathcal{A}$. 
\end{theorem}
\begin{proof}
The proof is by induction on $\dim V$. If all vectors in $V$ are common eigenvectors for $\mathcal{A}$, then we are done since we can take any orthonormal basis of $V$. Therefore we may assume that there exists an $A\in\mathcal{A}$ such that not every $v\in V$ is an eigenvector for $A$. Since $V$ is $\{A\}$-invariant, $A$ has some eigenvector $v\in V$ of eigenvalue $\lambda\in \Complex$. Denote by 
\begin{equation}
E_{\lambda}:=\{x\in \Complex^n\mid Ax=\lambda x\}
\end{equation}
the eigenspace of $A$ for eigenvalue $\lambda$. As $\mathcal{A}$ is commutative, the space $E_{\lambda}$ is $\mathcal{A}$-invariant. This follows since for any $B\in \mathcal{A}$ and any $v\in E_{\lambda}$ we have
\begin{equation}
A(Bv)=B(Av)=\lambda Bv,
\end{equation}   
and hence $Bv\in E_{\lambda}$. It follows that also $V':=V\cap E_{\lambda}$ and $V'':=V\cap E_{\lambda}^{\perp}$ are $\mathcal{A}$-invariant. By assumption on $A$, $V''$ has positive dimension, yielding a nontrivial orthogonal decomposition $V=V'\oplus V''$. By induction both $V'$ and $V''$ have an orthonormal basis of common eigenvectors of $\mathcal{A}$. The union of these two bases gives an orthonormal basis of $V$ consisting of common eigenvectors of $\mathcal{A}$.
\end{proof}

Let us consider the special case $V:=\Complex ^n$. Let $\{U_1,\ldots, U_n\}$ be an orthonormal basis of common eigenvectors for $\mathcal{A}$ and denote by $U$ the square matrix with these vectors as columns (in some order). Then $U$ is a unitary matrix that diagonalises $\mathcal{A}$. That is, all matrices in $U^{\ast}\mathcal{A}U$ are diagonal matrices. 
\begin{proposition}
Let $\mathcal{A}\subseteq \Complex^{n\times n}$ be an algebra consisting of diagonal matrices. Then there exist zero-one diagonal matrices $I_1,\ldots I_m$ with disjoint support such that
\begin{equation}
\mathcal{A}=\Complex I_1+\cdots+\Complex I_m.
\end{equation}
\end{proposition}   
\begin{proof}
Let $S:=\{i\in\{1,\ldots,n\}\mid A_{i,i}\not=0\text{ for some $A\in \mathcal{A}$}\}$ be the union of the supports on the diagonal, of the matrices in $\mathcal{A}$. Define an equivalence relation on $S$ by calling $i$ and $j$ equivalent when $A_{i,i}=A_{j,j}$ for every $A\in \mathcal{A}$, and let $S_1,\ldots,S_m$ be the equivalence classes. Denote for $k=1,\ldots,m$ by $I_k:=\Diag(\chi^{S_k})$ the zero-one diagonal matrix with support $S_k$. The inclusion 
\begin{equation}
\mathcal{A}\subseteq \Complex I_1+\cdots+\Complex I_m
\end{equation}
is clear.

To finish the proof, we show that $I_1,\ldots,I_m\in \mathcal{A}$. 
It is not hard to see that there is a matrix $A\in \mathcal{A}$ with 
\begin{equation}
A=c_1I_1+c_2I_2+\cdots+c_mI_m,
\end{equation}
for pairwise different nonzero numbers $c_1,\ldots,c_m$. It then follows that for $k=1,\ldots, m$ we have
\begin{equation}
I_k=\frac{A\prod_{i\not=k}(A-c_iI)}{c_k\prod_{i\not=k}(c_k-c_i)}.
\end{equation} 
Since the right-hand side is a polynomial in $A$ with constant term equal to zero, we obtain $I_k\in \mathcal{A}$.
\end{proof}

When $\mathcal{A}$ is a commutative matrix $\ast$-algebra containing the identity, and $U$ is a unitary matrix diagonalising the algebra, say $U^{\ast}\mathcal{A}U=\Complex I_1+\cdots+\Complex I_m$, then the matrices
\begin{equation}
E_k:=UI_kU^{\ast}\in \mathcal{A}
\end{equation}
form a basis of \notion{orthogonal idempotents} of $\mathcal{A}$. They satisfy
\begin{eqnarray}
E_1+\cdots+E_m&=&I,\\
E_iE_j&=&\delta_{i,j},\nonumber\\
E_i&=&E_i^{\ast},\nonumber
\end{eqnarray} 
for $i,j\in \{1,\ldots, m\}$.
Geometrically, we have an orthogonal decomposition 
\begin{equation}
\Complex^n=V_1\oplus\cdots\oplus V_m
\end{equation}
and $E_k$ is the orthogonal projection of $\Complex^n$ onto $V_k$.

We will now consider the case that the matrix $\ast$-algebra $\mathcal{A}$ is not necessarily commutative. We first introduce some terminology. Let $\mathcal{A}\subseteq \Complex^{n\times n}$ be a matrix $\ast$-algebra. An element $E\in \mathcal{A}$ is called a \notion{unit} of $\mathcal{A}$ when $EA=AE=A$ for every $A\in \mathcal{A}$. Every matrix $\ast$-algebra has a unit, see Proposition \ref{unit} below. A \notion{sub $\ast$-algebra} of $\mathcal{A}$ is a subset of $\mathcal{A}$ that is a matrix $\ast$-algebra. An important example is 
\begin{equation}
C_{\mathcal{A}}:=\{A\in \mathcal{A}\mid AB=BA \text{ for all $B\in \mathcal{A}$}\}.
\end{equation}
An \notion{ideal} of $\mathcal{A}$ is a sub $\ast$-algebra that is closed under both left and right multiplication by elements of $\mathcal{A}$. We observe that if $\mathcal{I}$ is an ideal of $\mathcal{A}$ and $E$ is the unit of $\mathcal{I}$, then $E\in C_{\mathcal{A}}$. This follows since for any $A\in \mathcal{A}$ both $EA$ and $AE$ belong to $\mathcal{I}$ and hence $EA=EAE=AE$.

We say that a commutative sub $\ast$-algebra of $\mathcal{A}$ is \emph{maximal} if it is not strictly contained in a commutative sub $\ast$-algebra of $\mathcal{A}$. We have the following useful property.
\begin{proposition}
Let $\mathcal{B}$ be a maximal commutative sub $\ast$-algebra of the matrix $\ast$-algebra $\mathcal{A}$ and let 
\begin{equation}
\mathcal{B}':=\{A\in\mathcal{A}\mid AB=BA\text{ for all $B\in \mathcal{B}$}\}.
\end{equation}
Then $\mathcal{B'}=\mathcal{B}$.   
\end{proposition}
\begin{proof}
Clearly $\mathcal{B}\subseteq \mathcal{B}'$. We show the converse inclusion. First observe that for any $A\in\mathcal{B}'$ also $A^{\ast}\in \mathcal{B}'$. This follows since for any $B\in \mathcal{B}$ we have 
\begin{equation}
A^{\ast}B=(B^{\ast}A)^{\ast}=(AB^{\ast})^{\ast}=BA^{\ast}.
\end{equation}
Next, we show that $\mathcal{B}$ contains every $A\in\mathcal{B}'$ that is normal (that is $AA^{\ast}=AA^{\ast}$). This follows since for any normal $A\in \mathcal{B}'\setminus \mathcal{B}$ the commutative matrix $\ast$-algebra generated by $A$, $A^{\ast}$ and $\mathcal{B}$,
strictly contains $\mathcal{B}$.

Finally, let $A\in\mathcal{B}'$ be arbitrary. The matrix $A+A^{\ast}$ is normal, and hence belongs to $\mathcal{B}$. It follows that $A(A+A^{\ast})=(A+A^{\ast})A$, or $AA^{\ast}=A^{\ast}A$, and hence $A$ itself is normal and therefore belongs to $\mathcal{B}$.
\end{proof} 

When a matrix $\ast$-algebra does not contain the identity, the following proposition is useful. 
\begin{proposition}\label{unit}
Every nonzero matrix $\ast$-algebra $\mathcal{A}$ is equivalent to a direct sum of a matrix $\ast$-algebra containing the identity and (possibly) a zero algebra. In particular, $\mathcal{A}$ has a unit.
\end{proposition}
\begin{proof}
Let $\mathcal{B}$ be a maximal commutative sub $\ast$-algebra of $\mathcal{A}$. By diagonalising $\mathcal{B}$ we may assume that
\begin{equation}
\mathcal{B}=\Complex I_1+\cdots+\Complex I_m
\end{equation} 
for diagonal zero-one matrices $I_0,I_1,\ldots,I_m$ with $I=I_0+I_1+\cdots+I_m$. If $I_0=0$, we are done. So we may assume that $I_0\not=0$.  
To prove the proposition, it suffices to show that  
\begin{equation}
I_0\mathcal{A}=\mathcal{A}I_0=\{0\},
\end{equation}
since this implies that $\mathcal{A}$ is the direct sum of the algebra obtained by restricting $\mathcal{A}$ to the support of $I_1+\cdots+I_m$ and the zero algebra on the support of $I_0$.

First observe that 
\begin{equation}
I_0A=A-(I_1+\cdots+I_m)A\in \mathcal{A}\text{ for every $A\in\mathcal{A}$}.
\end{equation}
Let $A\in \mathcal{A}$ be arbitrary and let 
\begin{equation}
A':=(I_0A)(I_0A)^{\ast}\in \mathcal{A}.
\end{equation}
Then for $k=1,\ldots,m$ we have
\begin{equation}
I_kA'=I_kI_0AA^{\ast}I_0=0=I_0AA^{\ast}I_0I_k=A'I_k.
\end{equation} 
It follows that $A'$ commutes with $I_1,\ldots, I_m$ and hence is a linear combination of $I_1,\ldots,I_m$ by the maximality of $\mathcal{B}$. On the other hand $A'I_k=0$ for $k=1,\ldots,m$, and hence $A'=0$. It follows that also $I_0A=0$.

Similarly, by considering $A^{\ast}$ we obtain $I_0A^{\ast}=0$ and hence $AI_0=0$. 
\end{proof}

We call a nonzero matrix $\ast$-algebra $\mathcal{A}$ \notion{simple} if $C_{\mathcal{A}}=\Complex E$, where $E$ is the unit of $\mathcal{A}$. Since the unit of any ideal of $\mathcal{A}$ belongs to $C_{\mathcal{A}}$, it follows that if $\mathcal{A}$ is simple, it has only the two trivial ideals $\{0\}$ and $\mathcal{A}$. The reverse implication also holds (see Proposition \ref{BD1}).

\begin{proposition}\label{BD1}
Every matrix $\ast$-algebra $\mathcal{A}$ containing the identity is equivalent to a direct sum of simple matrix $\ast$-algebras.  
\end{proposition}
\begin{proof}
Since $C_{\mathcal{A}}$ is commutative, we may assume it is diagonalised by replacing $\mathcal{A}$ by $U^{\ast}\mathcal{A}U$ for a unitary matrix $U$ diagonalising $C_{\mathcal{A}}$. Then 
\begin{equation}
C_{\mathcal{A}}=\Complex I_1+\cdots+\Complex I_m
\end{equation}
where $I_1,\ldots,I_m$ are zero-one diagonal matrices with $I_1+\cdots+I_m=I$. For every $i,j\in \{1,\ldots,m\}$ with $i\not=j$ we have
\begin{equation}
I_i\mathcal{A}I_j=I_iI_j\mathcal{A}=\{0\}.
\end{equation} 
It follows that $\mathcal{A}$ is the direct sum
\begin{equation}
\mathcal{A}=\mathcal{A}_1\oplus\cdots\oplus\mathcal{A}_m,
\end{equation}
where for $i=1,\ldots,m$ the matrix $\ast$-algebra $\mathcal{A}_i$ is obtained from $I_i\mathcal{A}I_i$ by restricting to the rows and columns in which $I_i$ has a $1$.
\end{proof}

Finally, we show that every simple matrix $\ast$-algebra can be brought into block diagonal form.
\begin{proposition}\label{BD2}
Every simple matrix $\ast$-algebra $\mathcal{A}$ containing the identity is equivalent to a matrix $\ast$-algebra of the form $t\odot \Complex^{m\times m}$ for some $t,m$.
\end{proposition}
\begin{proof}
Let $\mathcal{A}\subseteq \Complex^{n\times n}$ be a simple matrix $\ast$-algebra containing the identity, and let $\mathcal{B}$ be a maximal commutative sub $\ast$-algebra of $\mathcal{A}$. We may assume that $\mathcal{B}$ consists of diagonal matrices, say 
\begin{equation}
\mathcal{B}=\Complex I_1+\cdots+\Complex I_m
\end{equation}
where $I_i=\chi^{S_i}$ for $i=1,\ldots,m$ and $S_1\cup \cdots\cup S_m$ is a partition of $\{1,\ldots,n\}$. For every $i$ and every $A\in \mathcal{A}$ the matrix $I_iAI_i$ commutes with $I_1,\ldots,I_m$ and hence, by the maximality of $\mathcal{B}$, the matrix $I_iAI_i$ is a linear combination of $I_1,\ldots,I_m$. It follows that 
\begin{equation}
I_i\mathcal{A}I_i=\Complex I_i \text{\  for $i=1,\ldots,m$}.
\end{equation}
For any $i$, the set $\mathcal{I}:=\mathcal{A}I_i\mathcal{A}$ is a nonzero ideal of $\mathcal{A}$. Hence the unit of $\mathcal{I}$ belongs to $C_{\mathcal{A}}=\Complex I$. It follows that $I\in \mathcal{I}$ and hence 
\begin{equation}\label{nonzeroblocks}
I_i\mathcal{A}I_j\not=\{0\} \text{ for every $i,j=1,\ldots,m$.} 
\end{equation}

For any $A\in\Complex^{n\times n}$, and $i,j\in \{1,\ldots,m\}$, we denote by $A_{i,j}\in \Complex^{|S_i|\times |S_j|}$ the matrix obtained from $A$ by restricting the rows to $S_i$ and the columns to $S_j$ (and renumbering the rows and columns). By (\ref{nonzeroblocks}) we can fix an $A\in\mathcal{A}$ with $A_{i,j}\not=0$ for every $i,j\in \{1,\ldots,m\}$. In fact we can arrange that
\begin{equation}\label{scaled}
\trace ((A_{i,j})^{\ast}A_{i,j})=|S_i|.
\end{equation}
Let $i$ be arbitrary and let $A':=I_1AI_i$. Then 
\begin{eqnarray}
A'(A')^{\ast} \text{ is a nonzero matrix in $\Complex I_1$},\\
(A')^{\ast}A' \text{ is a nonzero matrix in $\Complex I_i$}.\nonumber
\end{eqnarray}
This shows that $I_1$ and $I_i$ have the same rank $t$, namely the rank of $A'$. In other words: $|S_1|=|S_i|=t$. Moreover by (\ref{scaled}), the matrices $A_{1,i}$ are unitary since 
\begin{equation}
(A_{1,i})^{\ast}A_{1,i}=A_{1,i}(A_{1,i})^{\ast}=I.
\end{equation}
Let $U:=A_{1,1}^{\ast}\oplus\cdots\oplus A_{1,m}^{\ast}\in \Complex^{n\times n}$ be the unitary matrix with blocks $A_{1,i}^{\ast}$ on the diagonal. By replacing $\mathcal{A}$ by $U^{\ast}\mathcal{A}U$ we may assume that $A_{1,i}=I$ for $i=1,\ldots, m$.

This implies that for any $i,j\in \{1,\ldots,m\}$
\begin{equation}
B_{i,1}=A_{1,i}B_{i,1}=(I_1AI_iBI_1)_{1,1}\in \Complex I\text{\  for any $B\in \mathcal{A}$},
\end{equation}
and hence
\begin{equation}
B_{i,j}=B_{i,j}(A^{\ast})_{j,1}=(I_iBI_jA^{\ast}I_1)_{i,1}\in \Complex I\text{\  for any $B\in \mathcal{A}$}.
\end{equation}
Summarizing, we have 
\begin{equation}
\mathcal{A}=\{A\in\Complex^{n\times n}\mid A_{i,j}\in \Complex I\text{\  for all $i,j\in\{1,\ldots,m\}$}\}.
\end{equation}
By reordering the rows and columns, we obtain the proposition.
\end{proof}

Proposition \ref{BD1} and \ref{BD2} together imply Theorem \ref{BD}.

\section{Semidefinite programming}
In this section we introduce semidefinite programming. For an overview of semidefinite programming and further references, we refer the reader to \cite{Todd}.
 
Recall that a complex matrix $A$ is called hermitian if $A^{\ast}=A$. It follows that all eigenvalues of $A$ are real. An hermitian matrix $A\in \Complex^{n\times n}$ is called \notion{positive semidefinite}, in notation $A\psd$, when it has only nonnegative eigenvalues.  
\begin{proposition}
For an hermitian matrix $A\in \Complex^{n\times n}$ the following are equivalent:
\begin{eqnarray}
\mathrm{(i)}&&A\psd,\\ 
\mathrm{(ii)}&&x^{\ast}Ax\geq 0\quad\text{for all $x\in \Complex^n$},\nonumber\\
\mathrm{(iii)}&&A=B^{\ast}B\quad \text{for some $B\in \Complex^{n\times n}$}.\nonumber
\end{eqnarray}
In the case that $A$ is real, we may restrict to real vectors $x$ in \textrm{(ii)} and take $B$ real in \textrm{(iii)}.
\end{proposition}  
It follows that for positive semidefinite matrices $A,B\in \Complex^{n\times n}$ the inner product is nonnegative:
\begin{equation}
\left<A,B\right>=\trace (C^{\ast}CDD^{\ast})=\trace (CDD^{\ast}C^{\ast})=\left<CD,CD\right>\geq 0,
\end{equation}
when $A=C^{\ast}C$ and $B=D^{\ast}D$. Another useful observation is that when $A$ is positive semidefinite, every principal submatrix is positive semidefinite as well. In particular, the diagonal of $A$ consists of nonnegative real numbers. Also
\begin{equation}
\text{if $U$ is nonsingular, then $A\psd$ if and only if $U^{\ast}AU\psd$}. 
\end{equation}

In the remainder of this section, all matrices will be real. A \notion{semidefinite program} is a an optimization problem of the following form, where $A_1,\ldots, A_n,B$ are given symmetric matrices in $\Real^{n\times n}$ and $c\in\Real^n$ is a given vector:
\begin{eqnarray}\label{generalSDP}
\text{minimize}&& c^{\transp}x\\
\text{subject to}&& x_1A_1+\cdots+x_nA_n-B\psd\nonumber.
\end{eqnarray}  
When $A_1,\ldots, A_n,B$ are diagonal matrices, the program reduces to a linear program. In particular, linear constraints $Ax\leq b$ can be incorporated into the program (\ref{generalSDP}) by setting
\begin{equation}
\widetilde{A}_i:=\begin{pmatrix}A_i&0\\0&-\Diag(a_i)\end{pmatrix}
\end{equation}
and
\begin{equation}
\widetilde{B}:=\begin{pmatrix}B&0\\0&-\Diag(b)\end{pmatrix},
\end{equation}
where $a_i$ is the $i$-th column of $A$.
Semidefinite programs can be approximated in polynomial time within any specified accuracy by the ellipsoid algorithm (\cite{GLS}) or by practically efficient interior point methods (\cite{nesterov}).

For any symmetric matrix $A\in\Real^{n\times n}$, the matrix $R(A)$ is defined by:
\begin{equation}
R(A):=\begin{pmatrix}1&a^{\transp}\\a&A\end{pmatrix},
\end{equation}
where $a:=\diag (A)$ is the vector of diagonal elements of $A$. We will index the extra row and column of $R(A)$ by $0$. 

The following propositions are helpful when dealing with semidefinite programs that involve matrices of the form $R(A)$. 
\begin{proposition}\label{R(A)PSD}
Let $A\in\Real^{n\times n}$ be a symmetric matrix such that $\diag (A)=c\cdot A\one$ for some $c\in\Real$. Then the following are equivalent:
\begin{eqnarray}
\mathrm{(i)}&&R(A)\text{ is positive semidefinite,}\\
\mathrm{(ii)}&&A \text{ is positive semidefinite and $\one^{\transp}A\one\geq (\trace A)^2$.}\nonumber
\end{eqnarray}
\end{proposition}
\begin{proof}
First assume that (i) holds. Let $R(A)=U^{\transp}U$, where $U\in\Real^{(n+1)\times (n+1)}$. Using $U_0^{\transp}U_0=1$, we obtain
\begin{eqnarray}
\sumentries{A}=\sum_{i,j=1}^n U_i^{\transp}U_j&=&(\sum_{i=1}^n U_i)^{\transp}(\sum_{i=1}^n U_i)\cdot U_0^{\transp}U_0\\
&\geq&((\sum_{i=1}^n U_i)^{\transp}U_0)^2=(\trace A)^2.\nonumber
\end{eqnarray}
Here the inequality follows using Cauchy-Schwarz, and in the last equality we use $U_i^{\transp}U_0=A_{i,i}$.
Next assume that (ii) holds. We may assume that $\trace A>0$, since otherwise $A=0$ and hence $R(A)$ is positive semidefinite. Let $A=U^{\transp}U$ where $U\in\Real^{n\times n}$. Let $a:=\diag(A)$. For any $x\in \Real^n$ the following holds:
\begin{eqnarray}
x^{\transp}Ax&\geq& (\one^{\transp}A\one)^{-1}(x^{\transp}A\one)^2\\
&\geq&\left(\frac{\trace A}{\one^{\transp}A\one}x^{\transp}A\one\right)^2\nonumber\\
&=&c\frac{\one^{\transp}a}{\one^{\transp}a}c^{-1}x^{\transp}a\nonumber\\
&=&(x^{\transp}a)^2.\nonumber
\end{eqnarray}
Here the first inequality follows by applying Cauchy-Schwartz on the inner product of $Ux$ and $U\one$, and the second inequality follows from the assumption $\one^{\transp} A\one\geq (\trace A)^2$. It follows that for any vector ${\alpha\choose x}$ with $x\in \Real^{n}$ and $\alpha\in \Real$, we have
\begin{eqnarray}
(\alpha,x^{\transp})R(A){\alpha\choose x}&=&\alpha^2+2\alpha a^{\transp}x+x^{\transp}Ax\nonumber\\
&\geq&\alpha^2+2\alpha a^{\transp}x+(a^{\transp}x)^2\nonumber\\
&=&(\alpha+a^{\transp}x)^2\geq 0.\nonumber
\end{eqnarray}
\end{proof}

This implies the folowing useful equivalence of semidefinite programs.
\begin{proposition}\label{twooptima}
Let $C\subseteq \Real^{n\times n}$ be a cone, and assume that the following two maxima exist:
\begin{eqnarray}
O_1&:=&\max\{\one^{\transp}A\one\mid \trace A=1, A\psd, A\in C\},\\
O_2&:=&\max\{\trace A\mid R(A)\psd, A\in C\}.\nonumber
\end{eqnarray}
Further assume that the maximum in the first program is attained by a matrix $A$ with $\diag (A)=c\cdot A\one$ for some $c\in\Real$. Then $O_1=O_2$.
\end{proposition}
\begin{proof}
Let $A$ be an optimal solution to the first program with $\diag(A)=c\cdot A\one$ for some $c\in\mathbb{R}$, and define $A':=(\one^{\transp}A\one)A$. Then 
\begin{equation}
\one^{\transp}A'\one=(\one^{\transp}A\one)^2=(\trace A')^2.
\end{equation}
Hence $A'$ is feasible for the second program by Proposition \ref{R(A)PSD}. Since $\trace A'=\one^{\transp}A\one$ we obtain $O_2\geq O_1$.

Let $A$ be an optimal solution to the second program. If $\trace A=0$ we have $O_1\geq O_2$ and we are done. Hence we may assume that $\trace A>0$. Observe that $(\trace A)^2=\sumentries{A}$, since otherwise we would have $\sumentries{A}=\lambda (\trace A)^2$ for some $\lambda>1$ by Proposition \ref{R(A)PSD}. This would imply that $\lambda A$ is also feasible, contradicting the optimality of $A$. Define $A':=\frac{1}{\trace A}A$. Then $A'$ is feasible for the first program and 
\begin{equation}
\sumentries{A'}=\frac{1}{\trace A}\sumentries{A}=\trace A
\end{equation}
This implies that $O_1\geq O_2$.
\end{proof}
An important special case is when all feasible matrices have constant diagonal and constant row sum. this occurs for example in semidefinite programs where the feasible matrices belong to the Bose-Mesner algebra of an association scheme. Another case is when the cone $C$ is closed under scaling rows and columns by nonnegative numbers.

\begin{proposition}\label{symoptimum}
Let $C\subseteq \Real^{n\times n}$ be a cone of symmetric matrices, such that for any nonnegative $x\in \Real^n$ and any $A\in C$ also $\Diag(x)A\Diag(x)$ belongs to $C$. Then any optimal solution $A$ to the program 
\begin{equation}\label{program}
\max\{\one^{\transp}A\one\mid \trace A=1, A\psd, A\in C\}
\end{equation} 
satisfies $\diag (A)=c\cdot A\one$ for some $c\in\Real$.
\end{proposition}
\begin{proof}
Let $A$ be an optimal solution. If $A_{i,i}$=0 for some $i$, we have $A_i=0$ and the claim follows by induction on $n$. Therefore we may assume that $a_i:=\sqrt{A_{i,i}}>0$ for $i=1,\ldots, n$. The matrix $A':=(\Diag(a))^{-1}A(\Diag(a))^{-1}$ is scaled to have only ones on the diagonal. Now for every nonnegative $x\in \Real^n$ with $\|x\|=1$, the matrix $A(x):=\Diag(x)A'\Diag(x)$ is a feasible solution to (\ref{program}) and has value $x^{\transp}A'x$. By the optimality of $A$, the vector $a$ maximizes $x^{\transp}A'x$ over all nonnegative vectors $x$ with $\|x\|=1$. In fact, since $a>0$, it maximizes $x^{\transp}A'x$ over all $x$ with $\|x\|=1$. As $\Real^n$ has an orthonormal basis of eigenvectors for $A'$, it follows that $a$ is an eigenvector of $A'$ belonging to the maximal eigenvalue $\lambda$. This implies that
\begin{eqnarray}
A\one=\Diag(a)A'\Diag(a)\one&=&\Diag(a)A'a\\
&=&\lambda \Diag(a)a\nonumber\\
&=&\lambda (a_1^2,\ldots,a_n^2)^{\transp}.\nonumber
\end{eqnarray}
This finishes the proof since 
\begin{equation}
\diag (A)=(a_1^2,\ldots, a_n^2)^{\transp}.
\end{equation}
\end{proof}

\section{Association schemes}
In this section, we give some basic facts and notions related to association schemes, including Delsarte's linear programming approach for bounding the size of cliques in an association scheme. This is by no means a complete introduction to the theory of association schemes. For further reading, we recommend \cite{Bose, Bailey, Delsarte} on association schemes and \cite{Distreggraph} on the related topic of distance regular graphs.
  
Roughly speaking, an association scheme is a very regular colouring of the edges of a complete graph. The colouring is such, that the number of walks from a vertex $a$ to a vertex $b$ traversing colours in a prescribed order, does not depend on the two vertices $a$ and $b$, but merely on the colour of the edge $ab$. The following formal definition is due to Bose and Shimamoto \cite{Bose}. 
A \emph{$t$-class} \notion{association scheme} $S=(X,\{R_0,R_1,\ldots,R_t\})$ is a finite set $X$ together with $t+1$ relations $R_0,\ldots,R_t$ on $X$ that satisfy the following axioms
\begin{eqnarray}\label{axioms}
\mathrm{(i)}&&\text{$\{R_0,R_1,\ldots,R_t\}$ is a partition of $X\times X$,}\nonumber\\
\mathrm{(ii)}&&\text{$R_0=\{(x,x)\mid x\in X\}$,}\nonumber\\
\mathrm{(iii)}&&\text{$(x,y)\in R_i$ if and only if $(y,x)\in R_i$ for all $x,y\in X, i\in\{0,\ldots,t$\},}\nonumber\\
\mathrm{(iv)}&&\text{for any $i,j,k\in \{0,\ldots,t\}$ there is an integer $p_{i,j}^k$ such that}\nonumber\\
&&\text{$|\{z\in X\mid (x,z)\in R_i, (z,y)\in R_j\}|=p_{i,j}^k$ whenever $(x,y)\in R_k$}.\nonumber 
\end{eqnarray}
The set $X$ is called the set of \emph{points} of the association scheme and two points $x,y\in X$ are said to be $i$-related when $(x,y)\in R_i$. An association scheme defined as above, is sometimes called a \emph{symmetric} association scheme since all relations are symmetric by (iii). Some authors prefer to allow for `non-symmetric association schemes' by replacing condition (iii) by
\begin{eqnarray}
\mathrm{(iii')}&&\text{for each $i\in \{0,\ldots, t\}$ there is an $i^{\ast}\in \{0,\ldots,t\}$ such that}\nonumber\\
&&\text{$(x,y)\in R_i$ implies $(y,x)\in R_{i^{\ast}}$ for all $x,y\in X$},\nonumber\\
\mathrm{(iii'')}&&p_{i,j}^k=p_{j,i}^k\quad \text{for all $i,j,k\in\{0,\ldots,t\}$}.\nonumber\\ 
\end{eqnarray}
In this thesis we will only use symmetric association schemes.
 
The numbers $p_{i,j}^k$ are called the \notion{intersection numbers} of the association scheme. The intersection numbers are not free of relations. We mention some obvious relations:
\begin{eqnarray}
p_{i,j}^k&=&p_{j,i}^k,\\
p_{i,j}^0&=&\text{$0$ when $i\not=j$}.\nonumber
\end{eqnarray}
The numbers $n_i:=p_{i,i}^0$ are called the \notion{degrees} of the scheme and give the number of points that are $i$-related to a given point (each relation $R_i$ induces an $n_i$-regular graph on $X$). 

To each relation $R_i$, we associate the $X\times X$ matrix $A_i$ in the obvious way:
\begin{equation}
(A_i)_{x,y}:=\begin{cases}1&\text{if $(x,y)\in R_i$}\\0&\text{otherwise.}\end{cases}
\end{equation}
The matrices $A_0,\ldots, A_t$ are called the \notion{adjacency matrices} of the association scheme and allow to study the association scheme using algebraic (spectral) tools. In terms of the adjacency matrices, the axioms in (\ref{axioms}) become
\begin{eqnarray}\label{axioms2}
\mathrm{(i)}&&A_0+A_1 +\cdots + A_t=J,\nonumber\\
\mathrm{(ii)}&&A_0=I,\nonumber\\
\mathrm{(iii)}&&A_i=A_i^{\transp}\quad \text{for all $i\in\{0,\ldots,t$\},}\nonumber\\
\mathrm{(iv)}&&A_iA_j=\sum_{k=0}^t p_{i,j}^k A_k\quad\text{for any $i,j\in \{0,\ldots,t\}$.}\nonumber 
\end{eqnarray} 
Let 
\begin{equation}
\mathcal{A}:=\{x_0A_0+x_1A_1+\cdots +x_tA_t\mid x_0,\ldots,x_t\in \Complex\ \}
\end{equation}
be the linear space spanned by the adjacency matrices. Axiom (iv) says that $\mathcal{A}$ is closed under matrix multiplication. Since all matrices in $\mathcal{A}$ are symmetric, it follows that $\mathcal{A}$ is a commutative matrix $\ast$-algebra, which is called the \notion{Bose--Mesner algebra} of the association scheme. Since the adjacency matrices are nonzero and have disjoint support, they are linearly independent. This implies that the dimension of $\mathcal{A}$ equals $t+1$.

Since the algebra $\mathcal{A}$ is commutative, it has a basis $E_0,E_1,\ldots, E_t$ of matrices satisfying
\begin{eqnarray}\label{idempotent}
\mathrm{(i)}&&E_iE_j=\delta_{i,j} E_i,\\
\mathrm{(ii)}&&E_0+\ldots +E_t=I,\nonumber\\
\mathrm{(iii)}&&E_i^{\ast}=E_i,\nonumber  
\end{eqnarray}
for every $i,j\in\{0,\ldots,t\}$. The matrices $E_i$ are called the \notion{minimal idempotents} of the algebra and are uniquely determined by $\mathcal{A}$. Geometrically, this means that there is an orthogonal decomposition
\begin{equation}
\Complex^X=V_0\oplus V_1\oplus\cdots\oplus V_t,
\end{equation}
where $E_i$ is the orthogonal projection onto $V_i$ for $i=0,\ldots, t$. For each $i$ the dimension 
\begin{equation}
m_i:=\dim V_i
\end{equation}
equals the rank of $E_i$. The numbers $m_0,\ldots,m_t$ are called the \notion{multiplicities} of the association scheme.  

In general, there is no natural way to order the $E_i$. However, there is one exception. The matrix $|X|^{-1}J$ is always a minimal idempotent, hence it is customary to take $E_0:=|X|^{-1}J$ (and $V_0=\Complex\one$, $m_0=1$). Since all matrices in $\mathcal{A}$ are symmetric, the idempotents $E_i$ are real by (\ref{idempotent})(iii).  

Since both $\{E_0,\ldots, E_k\}$ and $\{A_0,\ldots, A_k\}$ are bases for $\mathcal{A}$, we can express every matrix in one base as a linear combination of matrices in the other base. The $(t+1)\times (t+1)$ real matrices $P,Q$ are defined as follows:
\begin{eqnarray}\label{basetransform}
A_j=\sum_{i=0}^t P_{i,j}E_i,\\
|X|\cdot E_j=\sum_{i=0}^t Q_{i,j}A_i,\nonumber
\end{eqnarray}
for $j=0,\ldots, t$. The matrices $P$ and $Q$ are called the \emph{first} and \emph{second} \notion{eigenmatrix} of the scheme respectively. Indeed, since 
\begin{equation}
\sum_{i=0}^t c_i E_i
\end{equation}
has eigenvalue $c_i$ with multiplicity $m_i$ (if the $c_i$ are different), the $i$-th column of $P$ gives the eigenvalues of $A_i$. Clearly 
\begin{equation}\label{inverses}
PQ=QP=|X|\cdot I,
\end{equation}
but additionally, the matrices $P$ and $Q$ satisfy the following relation
\begin{equation}\label{orthorelation}
m_jP_{j,i}=n_iQ_{i,j}, \quad\text{for all $i,j\in \{0,\ldots,t\}$}.
\end{equation} 
In matrix form:
\begin{equation}
P^{\transp}\Diag(m_0,\ldots,m_t)=\Diag(n_0,\ldots,n_t)Q. 
\end{equation}
This is a consequence of the fact that both bases $\{A_0,\ldots,A_t\}$ and $\{E_0,\ldots,E_t\}$ are orthogonal. Indeed, this implies by (\ref{basetransform}) that both the left-hand side and the right-hand side in equation (\ref{orthorelation}) are equal to $\left<A_i,E_j\right>$.
  
Given a subset $Y\subseteq X$ of the point set, the \notion{distribution vector} of $Y$ is the $(t+1)$-tuple $(a_0,a_1,\ldots, a_t)$ of nonnegative numbers defined by
\begin{equation}
a_i:=|Y|^{-1}\cdot |(Y\times Y)\cap R_i|,\quad i=0,\ldots,t.
\end{equation}
The numbers $a_i$ give the average number of elements in $Y$ that are $i$-related to a given element in $Y$. In particular $a_0=1$ and $a_0+\cdots+a_t=|Y|$. Delsarte \cite{Delsarte} showed that the distribution vector satisfies the following system of inequalities:
\begin{equation}\label{Delsarte}
\sum_{i=0}^t Q_{i,j}a_i\geq 0\quad\text{for $j=0,\ldots,t$}.
\end{equation} 
Let $K\subseteq \{1,\ldots, t\}$. A subset $S\subseteq X$ of the point set is called a \emph{$K$-}\notion{clique} if any two different elements $x,y\in K$ are $i$-related for some $i\in K$. The inequalities (\ref{Delsarte}) yield an upper bound on the maximum size of a $K$-clique called the \notion{Delsarte bound}.
\begin{theorem}
Let $(X,\{R_0,\ldots,R_t\})$ be an association scheme and let $K\subseteq \{1,\ldots,t\}$. Then the maximum size of a $K$-clique is upper bounded by   
\begin{eqnarray}
\max\ \{a_1+\cdots +a_t\mid &&\text{$a_0=1$, $a_i=0$ for $i\in \{1,2,\ldots,t\}\setminus K$}\\
&&\text{$a_i\geq 0$ for all $i\in\{1,2,\ldots,t\}$}\nonumber\\
&&\text{$a_0,\ldots,a_t$ satisfy the inequalities (\ref{Delsarte})}\}.\nonumber
\end{eqnarray} 
\end{theorem}
The Delsarte bound can be efficiently calculated using linear programming and often gives a remarkably good upper bound.

One source of association schemes are (permutation) groups. Let $G$ be a group acting on a finite set $X$. Then $G$ has a natural action on $X\times X$ given by $g(x,y):=(gx,gy)$. The orbits 
\begin{equation}
\{(gx,gy)\mid g\in G\}
\end{equation}
of $X\times X$ are called \emph{orbitals}. The group $G$ is said to act \notion{generously transitive} when for every pair $(x,y)\in X\times X$ there is a group element $g\in G$ that exchanges $x$ and $y$, that is $gx=gy$ and $gy=gx$. When $G$ acts generously transitive, the orbitals form the relations of an association scheme.

Indeed, the orbitals partition $X\times X$, for any $x\in X$ the orbital $\{(gx,gx)\mid g\in G\}$
is the identity relation (as $G$ acts transitively on $X$) and the orbitals are symmetric (since $G$ acts generously transitive). Finally, let $R_i,R_j,R_k$ be orbitals and let for $(x,y)\in R_k$
\begin{equation}
Z_{x,y}:=\{z\in X\mid (x,z)\in R_i, (z,y)\in R_j\}.
\end{equation}
We have to show that the cardinality $p_{i,j}^k$ of $Z_{x,y}$ only depends on the relations $i,j,k$ and not on the particular choice of $x$ and $y$. This follows since 
\begin{equation}
Z_{gx,gy}\supseteq \{gz\mid z\in Z_{x,y}\}
\end{equation}
for any $g\in G$. In this case, the Bose--Mesner algebra is the centralizer algebra of $G$.

Given an association scheme $S=(X,\mathcal{R})$ with adjacency matrices $A_0, A_1,\ldots, A_t\in \Complex^{X\times X}$, and a point $x\in X$, the \notion{Terwilliger algebra} \emph{of $S$ with respect to $x$} is the complex algebra generated by $A_0,\ldots, A_t$ and the diagonal matrices $E_0',\ldots, E_t'$ defined by
\begin{equation}
(E_i')_{y,y}:=\begin{cases}1&\text{if $(x,y)\in R_i$}\\0&\text{otherwise.}\end{cases}
\end{equation}
Observe that $E_0'+\cdots+E_t'=I$. These algebras were introduced by Terwilliger in \cite{Terwilliger} under the name \notion{subconstituent algebra} as a tool for studying association schemes. In this thesis we will use the Terwilliger algebra of the Hamming scheme to obtain bounds for codes, improving the Delsarte bound.
 
\chapter[The Terwilliger algebra of $H(n,q)$]{The Terwilliger algebra of the Hamming scheme}\label{CH:hammingschemes}
A particular association scheme that plays an important role in the theory of error correcting codes is the Hamming scheme. In this chapter we will consider this scheme together with matrix algebras associated to it. In particular we construct a block diagonalisation of the Terwilliger algebra of the binary and the nonbinary Hamming scheme.

\section[The Hamming scheme]{The Hamming scheme and its Terwilliger algebra}
Fix integers $n\geq 1$ and $q\geq 2$, and fix an alphabet $\q=\{0,1,\ldots,q-1\}$. We will consider the \notion{Hamming space} $\E=\q^n$ consisting of words of length $n$ equipped with the \notion{Hamming distance} given by
\begin{equation}
d(\mathbf{u},\mathbf{v}):=|\{i\mid \mathbf{u}_i\not= \mathbf{v}_i\}|.
\end{equation}
For a word $\mathbf{u}\in \E$, we denote the \emph{support} of $\mathbf{u}$ by $S(\mathbf{u}):=\{i\mid \mathbf{u}_i\not=0\}$. Note that $|S(\mathbf{u})|=d(\mathbf{u},\zero)$, where $\zero$ is the all-zero word. This number is called the \notion{weight} of $\mathbf{u}$.

Denote by $\aut$ the set of permutations of $\E$ that preserve the Hamming distance. It is not hard to see that $\aut$ consists of the permutations of $\E$ obtained by permuting the $n$ coordinates followed by independently permuting the alphabet $\q$ at each of the $n$ coordinates. If we consider the action of $\aut$ on the set $\E\times \E$, the orbits form an association scheme known as the \notion{Hamming scheme} $H(n,q)$, with adjacency matrices $A_0,A_1,\ldots,A_n$ defined by
\begin{equation}
(A_i)_{\mathbf{u},\mathbf{v}}:=\begin{cases}
1&\text{if $d(\mathbf{u},\mathbf{v})=i$,}\\
0&\text{otherwise,}
\end{cases}`
\end{equation}
for $i=0,1,\ldots,n$. The adjacency matrices span a commutative algebra over the complex numbers called the Bose--Mesner algebra of the scheme. 

We will now consider the action of $\aut$ on ordered triples of words, leading to a noncommutative algebra $\mathcal{A}_{q,n}$ containing the Bose--Mesner algebra. To each ordered triple $(\mathbf{u},\mathbf{v},\mathbf{w})\in\ \E\times\E\times\E$ we associate the four-tuple
\begin{eqnarray}
d(\mathbf{u},\mathbf{v},\mathbf{w})&:=&(i,j,t,p),\text{\ where}\\
&&i:=d(\mathbf{u},\mathbf{v}),\nonumber\\
&&j:=d(\mathbf{u},\mathbf{w}),\nonumber\\
&&t:=|\{i\mid \mathbf{u}_i\not=\mathbf{v}_i\text{\ and\ }\mathbf{u}_i\not=\mathbf{w}_i\}|,\nonumber\\
&&p:=|\{i\mid \mathbf{u}_i\not=\mathbf{v}_i=\mathbf{w}_i\}|\nonumber.
\end{eqnarray}
We remark that the case $q=2$ is special since in that case we always have $p=t$. 
Note that $d(\mathbf{v},\mathbf{w})=i+j-t-p$ and $|\{i\mid \mathbf{u}_i\not=\mathbf{v}_i\not=\mathbf{w}_i\not=\mathbf{u}_i\}|=t-p$. The set of four-tuples $(i,j,t,p)$ that occur as $d(\mathbf{u},\mathbf{v},\mathbf{w})$ for some $\mathbf{u},\mathbf{v},\mathbf{w}\in\E$ is given by
\begin{equation}
\mathcal{I}(2,n):=\{(i,j,t,p)\mid 0\leq p=t\leq i,j\text{\ and\ } i+j\leq n+t\},
\end{equation}
and 
\begin{equation}
\mathcal{I}(q,n):=\{(i,j,t,p)\mid 0\leq p\leq t\leq i,j\text{\ and\ } i+j\leq n+t\},
\end{equation}
for $q\geq 3$. The sets $\mathcal{I}(q,n)$ will be used to index various objects defined below.
\begin{proposition}\label{dimension}
Let $n\geq 1$. We have 
\begin{equation}
|\mathcal{I}(q,n)|\begin{cases}{n+3 \choose 3}&\text{for $q=2$,}\\ \\{n+4 \choose 4}&\text{for $q\geq 3$}.\end{cases}
\end{equation}
\end{proposition}
\begin{proof}
Substitute $p':=p$, $t':=t-p$, $i':=i-t$ and $j':=j-t$. Then the integer solutions of 
\begin{equation}
0\leq p\leq t\leq i,j,\quad i+j\leq n+t
\end{equation}
are in bijection with the integer solutions of 
\begin{equation}
0\leq p',t',i',j',\quad p'+t'+i'+j'\leq n.
\end{equation}
The proposition now follows since 
\begin{equation}
|\{(n_1,n_2,\ldots,n_k)\in\mathbb{Z}_{\geq 0}\mid n_1+\cdots+n_k=n\}|={n+k\choose k}.
\end{equation}
\end{proof}
The integers $i,j,t,p$ parametrize the ordered triples of words up to symmetry. We define 
\begin{equation}
X_{i,j,t,p}:=\{(\mathbf{u},\mathbf{v},\mathbf{w})\in\E\times\E\times\E\mid d(\mathbf{u},\mathbf{v},\mathbf{w})=(i,j,t,p)\},
\end{equation}
for $(i,j,t,p)\in \mathcal{I}(q,n)$. The meaning of the sets $X_{i,j,t,p}$ is given by the following proposition.
\begin{proposition}\label{orbits}
The sets $X_{i,j,t,p}$, $(i,j,t,p)\in\mathcal{I}(q,n)$ are the orbits of $\E\times\E\times\E$ under the action of $\aut$.
\end{proposition}
\begin{proof} Let $\mathbf{u},\mathbf{v},\mathbf{w}\in\E$ and let $(i,j,t,p)=d(\mathbf{u},\mathbf{v},\mathbf{w})$. Since the Hamming distances $i,j,i+j-t-p$ and the number $t-p=|\{i\mid \mathbf{u}_i\not=\mathbf{v}_i\not=\mathbf{w}_i\not=\mathbf{u}_i\}|$ are unchanged when permuting the coordinates or permuting the elements of $\q$ at any coordinate, we have $d(\mathbf{u},\mathbf{v},\mathbf{w})=d(\pi \mathbf{u},\pi \mathbf{v},\pi \mathbf{w})$ for any $\pi\in\aut$.

Hence it suffices to show that there is an automorphism $\pi$ such that $(\pi \mathbf{u},\pi \mathbf{v},\pi \mathbf{w})$ only depends upon $i,j,t\text{\ and\ }p$. By permuting $\q$ at the coordinates in the support of $\mathbf{u}$, we may assume that $\mathbf{u}=\zero$. Let $A:=\{i\mid \mathbf{v}_i\not=0, \mathbf{w}_i=0\}$, $B:=\{i\mid \mathbf{v}_i=0, \mathbf{w}_i\not=0\}$, $ C:=\{i \mid \mathbf{v}_i\not=0,\mathbf{w}_i\not=0, \mathbf{v}_i\not=\mathbf{w}_i\}$ and $D:=\{i\mid \mathbf{v}_i=\mathbf{w}_i\not=0\}$. Note that $|A|=i-t$, $|B|=j-t$, $|C|=t-p$ and $|D|=p$. By permuting coordinates, we may assume that $A=\{1,2,\ldots,i-t\}$, $B=\{i-t+1,\ldots,i+j-2t\}$, $C=\{i+j-2t+1,\ldots,i+j-t-p\}$ and $D=\{i+j-t-p+1,\ldots, i+j-t\}$. Now by permuting $\q$ at each of the points in $A\cup B\cup C\cup D$, we can accomplish that $\mathbf{v}_i=1$ for $i\in A\cup C\cup D$ and $\mathbf{w}_i=2$ for $i\in B\cup C$ and $\mathbf{w}_i=1$ for $i\in D$.
\end{proof}

Denote the stabilizer of $\zero$ in $\aut$ by $\stab$. For $(i,j,t,p)\in \mathcal{I}(q,n)$, let $M_{i,j}^{t,p}$ be the $\E\times\E$ matrix defined by:
\begin{equation}
(M_{i,j}^{t,p})_{\mathbf{u},\mathbf{v}}:=\begin{cases}
1&\text{if $|S(\mathbf{u})|=i$, $|S(\mathbf{v})|=j$, $|S(\mathbf{u})\cap S(\mathbf{v})|=t$,}\\
&\text{$|\{i\mid \mathbf{v}_i=\mathbf{u}_i\not=0\}|=p$,}\\ 0&\text{otherwise.}
\end{cases}
\end{equation}
Let $\mathcal{A}_{q,n}$ be the set of matrices
\begin{equation}
\sum_{(i,j,t,p)\in\mathcal{I}(q,n)} x_{i,j}^{t,p} M_{i,j}^{t,p},
\end{equation}
where $x_{i,j}^{t,p}\in\mathbb{C}$. In the binary case, we will usually drop the superfluous $p$ from the notation and write $x_{i,j}^t$ and $M_{i,j}^t$.

From Proposition \ref{orbits} it follows that $\mathcal{A}_{q,n}$ is the set of matrices that are stable under permutations $\pi\in\stab$ of the rows and columns. Hence $\mathcal{A}_{q,n}$ is the centralizer algebra  of $\stab$. The $M_{i,j}^{t,p}$ constitute a basis for $\mathcal{A}_{q,n}$ and hence 
\begin{equation}
\dim \mathcal{A}_{q,n}=\begin{cases}{n+3\choose 3}&\text{if $q=2$},\\ \\{n+4 \choose 4}&\text{if $q\geq 3$,}\end{cases}
\end{equation}
by Proposition \ref{dimension}. Note that the algebra $\mathcal{A}_{q,n}$ contains the Bose--Mesner algebra since
\begin{equation}\label{subpartition}
A_k=\sum_{\substack{(i,j,t,p)\in\mathcal{I}(q,n)\\ i+j-t-p=k}}M_{i,j}^{t,p}.
\end{equation}

We would like to point out here, that $\mathcal{A}_{q,n}$ coincides with the Terwilliger algebra (see \cite{Terwilliger}) of the Hamming scheme $H(n,q)$ (with respect to $\zero$). Recall that the Terwilliger algebra $\mathcal{T}_{q,n}$ is the complex matrix algebra generated by the adjacency matrices $A_0,A_1,\ldots,A_n$ of the Hamming scheme and the diagonal matrices $E'_0,E'_1,\ldots,E'_n$ defined by 
\begin{equation} (E'_i)_{\mathbf{u},\mathbf{u}}:=\begin{cases}
1&\text{if $|S(\mathbf{u})|=i$,}\\
0&\text{otherwise,}\\
\end{cases}
\end{equation}
for $i=0,1,\ldots,n$.
\begin{proposition}
The algebras $\mathcal{A}_{q,n}$ and $\mathcal{T}_{q,n}$ coincide.
\end{proposition}
\begin{proof}
Since $\mathcal{A}_{q,n}$ contains the matrices $A_k$ and the matrices $E'_k=M_{k,k}^{k,k}$ for $k=0,1,\ldots,n$, it follows that $\mathcal{T}_{q,n}$ is a subalgebra of $\mathcal{A}_{q,n}$. We show the converse inclusion. In the case $q=2$ this follows since 
\begin{equation}
M_{i,j}^{t,t}=E'_iA_{i+j-2t}E'_j,
\end{equation}
as is readily verified. We concentrate on the case $q\geq 3$. Define the zero-one matrices $B_i, C_i, D_i\in \mathcal{T}_{q,n}$ by
\begin{eqnarray}
B_i&:=&E'_iA_1E'_i,\\
C_i&:=&E'_iA_1E'_{i+1},\nonumber\\
D_i&:=&E'_iA_1E'_{i-1}.\nonumber
\end{eqnarray}
Observe that:
\begin{eqnarray}\label{stepmatrix}
(B_i)_{\mathbf{u},\mathbf{v}}&=1& \text{if and only if}\\
&&|S(\mathbf{u})|=i, d(\mathbf{u},\mathbf{v})=1, |S(\mathbf{v})|=i, S(\mathbf{u})=S(\mathbf{v}),\nonumber\\
(C_i)_{\mathbf{u},\mathbf{v}}&=1& \text{if and only if }\nonumber\\
&&|S(\mathbf{u})|=i, d(\mathbf{u},\mathbf{v})=1, |S(\mathbf{v})|=i+1, |S(\mathbf{u})\Delta S(\mathbf{v})|=1,\nonumber\\
(D_i)_{\mathbf{u},\mathbf{v}}&=1& \text{if and only if}\nonumber\\
&& |S(\mathbf{u})|=i, d(\mathbf{u},\mathbf{v})=1, |S(\mathbf{v})|=i-1, |S(\mathbf{u})\Delta S(\mathbf{v})|=1.\nonumber
\end{eqnarray}
For given $(i,j,t,p)\in\mathcal{I}(q,n)$, let $A_{i,j}^{t,p}\in\mathcal{T}_{q,n}$ be given by 
\begin{equation}
A_{i,j}^{t,p}:=(D_iD_{i-1}\cdots D_{t+1})(C_tC_{t+1}\cdots C_{j-1})(B_j)^{t-p}.
\end{equation}
Then for words $\mathbf{u},\mathbf{v}\in \E$, the entry $(A_{i,j}^{t,p})_{\mathbf{u},\mathbf{v}}$ counts the number of ($i+j-t-p+3$)-tuples 
\begin{equation}
\mathbf{u}=\mathbf{d}_i,\mathbf{d}_{i-1},\ldots,\mathbf{d}_t=\mathbf{c}_t,\mathbf{c}_{t+1},\ldots,\mathbf{c}_j=\mathbf{b}_0,\ldots,\mathbf{b}_{t-p}=\mathbf{v}
\end{equation}
where any two consecutive words have Hamming distance $1$, the $\mathbf{b}_k$ have equal support of cardinality $j$, and $|S(\mathbf{d}_k)|=k$, $|S(\mathbf{c}_k)|=k$ for all $k$.
Hence for $\mathbf{u},\mathbf{v}\in\E$ the following holds.
\begin{eqnarray}\label{step2a}
(A_{i,j}^{t,p})_{\mathbf{u},\mathbf{v}}&=0&\text{if $d(\mathbf{u},\mathbf{v})>i+j-t-p$\quad or}\nonumber\\&&\text{$|S(\mathbf{u})\Delta S(\mathbf{v})|>i+j-2t$}\\
\nonumber\\
&&\text{and}\nonumber\\
\nonumber\\
(A_{i,j}^{t,p})_{\mathbf{u},\mathbf{v}}&>0&\text{if $|S(\mathbf{u})|=i$, $|S(\mathbf{v})|=j$,}\label{step2b}\\
&&\text{$d(\mathbf{u},\mathbf{v})=i+j-t-p$\quad and}\nonumber\\ && \text{$|S(\mathbf{u})\Delta S(\mathbf{v})|=i+j-2t$}.\nonumber
\end{eqnarray}
Equation (\ref{step2a}) follows from the triangle inequality for $d$ and $d'(\mathbf{x},\mathbf{y}):=|S(\mathbf{x})\cap S(\mathbf{y})|$. To see (\ref{step2b}) one may take for $\mathbf{d}_k$ the zero-one word with support $\{i+1-k,\ldots,i\}$, for $\mathbf{c}_k$ the zero-one word with support $\{i+1-t,\ldots,i+k-t\}$ and for $\mathbf{b}_k$ the word with support $\{i+1-t,\ldots,i+j-t\}$ where the first $k$ nonzero entries are $2$ and the other nonzero entries are $1$.

Now suppose that $\mathcal{A}_{q,n}$ is not contained in $\mathcal{T}_{q,n}$, and let $M_{i,j}^{t,p}$ be a matrix not in $\mathcal{T}_{q,n}$ with $t$ maximal and (secondly) $p$ maximal. If we write
\begin{equation}
A_{i,j}^{t,p}=\sum_{t',p'} x_{i,j}^{t',p'} M_{i,j}^{t',p'},
\end{equation}
then by (\ref{step2a}) $x_{i,j}^{t',p'}=0$ if $t'+p'<t+p$ or $t'<t$ implying that $A_{i,j}^{t,p}-x_{i,j}^{t,p}M_{i,j}^{t,p}\in \mathcal{T}_{q,n}$ by the maximality assumption. Therefore since $x_{i,j}^{t,p}>0$ by (\ref{step2b}), also $M_{i,j}^{t,p}$ belongs to $\mathcal{T}_{q,n}$, a contradiction.
\end{proof}

\section{Block diagonalisation of $\mathcal{A}_n$}\label{Sec:blockdiag}
A block diagonalisation of $\mathcal{A}_n:=\mathcal{A}_{2,n}$ was first given by Schrijver in \cite{Lexcodes}. In this section we will describe this block diagonalisation. In the next section we will use it to describe a block diagonalisation of $\mathcal{A}_{q,n}$ for general $q$.
   
Let $n$ be a fixed positive integer and let $\mathcal{P}=\mathcal{P}_n$ denote the collection of subsets of $\{1,\ldots,n\}$. It will be convenient to identify binary words with their supports (as elements of $\mathcal{P}$). We will use capital letters to denote sets. For convenience, we use the notation
\begin{equation}
C_i:=M_{i-1,i}^{i-1},
\end{equation}
that is
\begin{equation}
(C_i)_{X,Y}=\begin{cases}1&\text{if $|X|=i-1, |Y|=i, X\subseteq Y$},\\0&\text{otherwise}\end{cases},
\end{equation}
for $i=0,\ldots n$. In particular observe that $C_0$ is the zero matrix. The matrices $C_1,\ldots,C_n$ and their transposes $C_1^{\transp},\ldots,C_n^{\transp}$ play a prominent role in the block diagonalisation of $\mathcal{A}_n$. They generate the algebra $\mathcal{A}_n$ as can be easily seen from the identities
\begin{equation}\label{bracket}
C_{i+1}C_{i+1}^{\transp}-C_i^{\transp}C_i=(n-2i)E'_i
\end{equation}
and
\begin{equation}
\sum_{i=0}^{n} (C_i+C_i^{\transp})=M_1.
\end{equation}
Indeed, the adjacency matrix $M_1$ of the Hamming cube generates the Bose--Mesner algebra of the Hamming scheme. Since $I=\sum_{i=1}^n E'_i$ is in the Bose--Mesner algebra it follows by (\ref{bracket}) that also the diagonal matrices $E'_1,\ldots,E'_n$ are in the algebra generated by $C_1,\ldots,C_n,C_1^{\transp},\ldots,C_n^{\transp}$.

For $k=0,\ldots,\rounddown{\frac{n}{2}}$, define the linear space $L_k$ to be the intersection of  the space of vectors with support contained in the collection of sets of cardinality $k$, and the kernel of $C_k$:
\begin{equation}
L_k:=\{b\in\Real^{\P}\mid C_kb=0, b_X=0\text{\ if\ }|X|\not=k\}.
\end{equation}
\begin{proposition}\label{dimlk}
For each $k\leq \rounddown{\frac{n}{2}}$ the dimension of $L_k$ is given by
\begin{equation}
\dim L_k={n\choose k}-{n\choose k-1}.
\end{equation}
\end{proposition}
\begin{proof}
It suffices to prove that $C_k$ has rank ${n\choose k-1}$. This follows since for any nonzero $x\in\Real^{\P}$ with $x_I=0$ when $|I|\not=k-1$, we have $C_kx\not=0$. Indeed
\begin{equation}
x^{\transp}C_kC_k^{\transp}x=x^{\transp}C_{k-1}^{\transp}C_{k-1}x+(n-2k+2)x^{\transp}x>0
\end{equation} 
by (\ref{bracket}).
\end{proof}

Before giving an explicit block diagonalisation, we will first sketch the basic idea.
Let $b\in L_k$ be nonzero and consider the vectors $b_k,b_{k+1},b_{k+2},\ldots$, where $b_k:=b$ and 
\begin{equation}
b_{i+1}:=C_{i+1}^{\transp}\cdots C_{k+2}^{\transp}C_{k+1}^{\transp}b
\end{equation}
for $i\geq k$. It can be shown (see Proposition \ref{innerprod} below) that 
\begin{equation}
\| b_i\| =\|b\|\cdot {n-2k\choose i-k}^{\frac{1}{2}}(i-k)!.
\end{equation}
It follows that $b_i$ is zero for $i>n-k$ and nonzero for $i=k,\ldots,n-k$. Since the $b_i$ have disjoint support, $b_k,\ldots,b_{n-k}$ are an orthogonal basis for the linear space $V_b$ they span. From (\ref{bracket}) it follows that
\begin{equation}
C_{i+1}b_{i+1}=C_{i+1}C_{i+1}^{\transp}b_i=(n-2i)b_i+C_i^{\transp}(C_ib_i)
\end{equation}
and hence, since $C_kb_k=0$, that
\begin{equation}
C_{i+i}b_{i+1}=b_i\cdot\sum_{s=k}^i (n-2s).
\end{equation} 
The space $V_b$ is thus mapped to itself by each of the $C_i$ and $C_i^{\transp}$ and hence by every $M\in \mathcal{A}_n$. The action of $\mathcal{A}_n$ restricted to $V_b$ is determined by
\begin{eqnarray}
C_{i+1}(\sum_{j=k}^{n-k}x_jb_j)&=&x_{i+1}(\sum_{s=k}^i (n-2s))b_i\\
C_{i+1}^{\transp}(\sum_{j=k}^{n-k}x_jb_j)&=&x_i b_{i+1}\nonumber
\end{eqnarray}
and does not depend on the particular choice of $b\in L_k$, but only on $k$. If we take for each $k$ an orthonormal basis of $L_k$ and let $b$ range over the union of these bases, we will obtain a decomposition of $\Real^{\P}$ as a direct sum of orthogonal subspaces $V_b$. This yields a block diagonalisation of $\mathcal{A}_n$, where for each $k$ there is a block of multiplicity $\dim L_k$. In order to obtain a formula for the image of $M_{i,j}^t$ in each of the blocks, we need to express $M_{i,j}^t$ (as a polynomial) in the matrices $C_l$ and $C_l^{\transp}$. 
 
We will now give a detailed proof, see also \cite{Lexcodes}. We begin by giving a convenient way to express the matrices $M_{i,j}^t$ in terms of $C_1,\ldots, C_n,C_1^{\transp},\ldots, C_n^{\transp}$. A first observation is that 
\begin{equation}
(k-i)M_{i,k}^i=M_{i,k-1}^iC_k\quad\text{for all $i<k$}.
\end{equation}
An important consequence is that 
\begin{equation}
M_{i,k}^i b=0\quad\text{for all $i<k$ and $b\in L_k$}. 
\end{equation}
Secondly, we have the following identity. 
\begin{proposition}
For all $l,k,p\in\{0,\ldots,n\}$:
\begin{equation}\label{eq1}
M_{l,k}^p=\sum_{s=0}^n (-1)^{s-p}{s\choose p}M_{l,s}^sM_{s,k}^s.
\end{equation}
\end{proposition}
\begin{proof}
The entry of 
\begin{equation}
\sum_{s=0}^n (-1)^{s-p}{s\choose p}M_{l,s}^sM_{s,k}^s
\end{equation}
in position $(X,Y)$, with $|X|=k$, $|Y|=l$ and $|X\cap Y|=t$, equals
\begin{eqnarray}
\sum_{s=0}^n (-1)^{s-p}{s\choose p}{t\choose s}&=& \sum_{s=p}^{t}(-1)^{s-p}{t\choose p}{t-p\choose s-p}\\
&=&{t\choose p}\sum_{s'=0}^{t-p} (-1)^{s'}{t-p\choose s'}.\nonumber
\end{eqnarray}
This last sum equals zero if $t\not=p$ and equals $1$ if $t=p$.
\end{proof}

The following proposition gives the inner products between vectors of the form $M_{j,k}^kb$, where $b\in L_k$. These will be used to construct an orthonormal basis with respect to which the algebra is in block diagonal form.
\begin{proposition}\label{innerprod}
For $i,j,k,l\in\{0,\ldots,n\}$ with $k,l\leq \rounddown{\frac{n}{2}}$, and for $c\in L_l,b\in L_k$:
\begin{equation}
c^{\transp}M_{l,i}^l M_{j,k}^k b=\begin{cases}{n-2k\choose i-k} c^{\transp}b&\text{if $l=k, i=j$}\\0&\text{otherwise.}\end{cases}
\end{equation}
\end{proposition}
\begin{proof}
Clearly $M_{l,i}^l M_{j,k}^k=0$ if $i\not=j$, hence we may assume $i=j$ in the remainder of the proof. By (\ref{eq1}) we have for $0\leq p\leq k,l$:
\begin{equation}
c^{\transp}M_{l,k}^p b=\begin{cases}(-1)^{k-p}{k\choose p} c^{\transp}b&\text{if $k=l$}\\0&\text{otherwise},\end{cases}
\end{equation}
since $c^{\transp}M_{l,s}^s=0$ for $s\not=l$ and $M_{s,k}^sb=0$ for $s\not=k$. Hence we obtain
\begin{eqnarray}
c^{\transp}M_{l,i}^l M_{i,k}^k b&=&\sum_{p=0}^{n} {n+p-l-k\choose n-i} c^{\transp}M_{l,k}^p b\\
&=&\delta_{k,l}\cdot \sum_{p=0}^{k} {n+p-2k\choose n-i}(-1)^{k-p}{k\choose p} c^{\transp}b\nonumber\\
&=&\delta_{k,l}\cdot {n-2k\choose i-k}c^{\transp}b.\nonumber
\end{eqnarray}
\end{proof}

Define for $i,j,k,t\in\{0,\ldots,n\}$ the number
\begin{equation}
\beta_{i,j,k}^t:={n-2k\choose i-k}\sum_{p=0}^n (-1)^{k-p}{k\choose p}{i-p\choose t-p}{n+p-i-k\choose n+t-i-j}.
\end{equation}
These numbers will be used to describe the block diagonalisation.
\begin{proposition}\label{actionmijt}
For $i,j,k,t\in\{0,\ldots,n\}$ with $k\leq \rounddown{\frac{n}{2}}$, and for $b\in L_k$:
\begin{equation}
{n-2k\choose i-k}M_{i,j}^t M_{j,k}^k b=\beta_{i,j,k}^t M_{i,k}^k b.
\end{equation}
\end{proposition}
\begin{proof}
By (\ref{eq1}), it follows that for $0\leq p\leq n$:
\begin{equation}
M_{i,k}^p b=(-1)^{k-p}{k\choose p}M_{i,k}^k b.
\end{equation}
This implies that 
\begin{eqnarray}
M_{i,j}^t M_{j,k}^k b&=&\sum_{p=0}^{n} {i-p\choose t-p}{n+p-i-k\choose n+t-i-j}M_{i,k}^p b\\
&=&\sum_{p=0}^n {i-p\choose t-p} {n+p-i-k\choose n+t-i-j}(-1)^{k-p}{k\choose p} M_{i,k}^k b.\nonumber
\end{eqnarray}
This proves the proposition.
\end{proof}
We will now describe the block diagonalisation. For each $k=0,\ldots, \rounddown{\frac{n}{2}}$, choose an orthonormal basis $B_k$ of the linear space $L_k$. By Proposition \ref{dimlk} we know that $|B_k|={n\choose k}-{n\choose k-1}$. Let 
\begin{equation}
\mathcal{V}:=\{(k,b,i)\mid k\in\{0,\ldots,\rounddown{\frac{n}{2}}\}, b\in B_k, i\in \{k,k+1,\ldots, n-k\}\}.
\end{equation}
Then 
\begin{eqnarray}\label{count}
|\mathcal{V}|&=&\sum_{i=0}^n \sum_{k=0}^{\min \{i,n-i\}}\left( {n\choose k}-{n\choose k-1}\right) \\
&=&\sum_{i=0}^n {n\choose \min \{i,n-i\}}=\sum_{i=0}^n {n\choose i}=2^n.\nonumber
\end{eqnarray}
Define for each $(k,b,i)\in \mathcal{V}$ the vector $u_{k,b,i}\in\Real^{\P}$ by
\begin{equation}
u_{k,b,i}:={n-2k\choose i-k}^{-\frac{1}{2}} M_{i,k}^k b.
\end{equation}
If follows from Proposition \ref{innerprod} and $|\mathcal{V}|=2^n$ that the vectors $u_{k,b,i}$ form an orthonormal base of $\Real^{\P}$. Let $U$ be the $\P\times \mathcal{V}$ matrix with $u_{k,b,i}$ as the $(k,b,i)$-th column. We will show that for each triple $i,j,t$ the matrix 
\begin{equation}
\widetilde{M}_{i,j}^t:=U^{\transp}M_{i,j}^t U
\end{equation}
is in block diagonal form. Indeed we have
\begin{proposition}\label{imageinblok}
For $(l,c,i'), (k,b,j')\in \mathcal{V}$ and $i,j,t\in \{0,\ldots,n\}$:
\begin{equation}
(\widetilde{M}_{i,j}^t)_{(l,c,i'),(k,b,j')}=\begin{cases}{n-2k\choose i-k}^{-\frac{1}{2}}{n-2k\choose j-k}^{-\frac{1}{2}}\beta_{i,j,k}^t&\text{if $l=k,i=i',j=j',b=c$,}\\ \\0&\text{otherwise.}\end{cases}
\end{equation}
\end{proposition}

\begin{proof}
We have 
\begin{eqnarray}
M_{i,j}^t u_{k,b,j'}&=&{n-2k\choose j'-k}^{-\frac{1}{2}} M_{i,j}^t M_{j',k}^k b\\
&=&\delta_{j,j'}{n-2k\choose j-k}^{-\frac{1}{2}}{n-2k\choose i-k}^{-1}\beta_{i,j,k}^t M_{i,k}^k b\nonumber\\
&=&\delta_{j,j'}{n-2k\choose j-k}^{-\frac{1}{2}}{n-2k\choose i-k}^{-\frac{1}{2}}\beta_{i,j,k}^t u_{k,b,i}.\nonumber
\end{eqnarray}
Since 
\begin{equation}
(\widetilde{M}_{i,j}^t)_{(l,c,i'),(k,b,j')}=u_{l,c,i'}^{\transp}M_{i,j}^t u_{k,b,j'}
\end{equation}
the proposition follows.
\end{proof}

\begin{proposition}
The matrix $U$ gives a block diagonalisation of $\mathcal{A}_n$.
\end{proposition} 
\begin{proof}
Proposition \ref{imageinblok} implies that each matrix $\widetilde{M}_{i,j}^t$ has a block diagonal form, where for each $k=0,\ldots \rounddown{\frac{n}{2}}$ there are ${n\choose k}-{n\choose k-1}$ copies of an $(n+1-2k)\times (n+1-2k)$ block on the diagonal. For each $k$ the copies are indexed by the elements of $B_k$, and in each copy the rows and columns are indexed by the integers $i\in\{k,k+1,\ldots,n-k\}$. Hence we need to show that all matrices of this block diagonal form belong to $U^{\transp}\mathcal{A}_nU$. It suffices to show that the dimension $\sum_{k=0}^{\rounddown{\frac{n}{2}}}(n+1-2k)^2$ of the algebra consisting of the matrices in the given block diagonal form equals the dimension of $\mathcal{A}_n$, which is ${n+3\choose 3}$. This follows by induction on $n$ from
\begin{equation}
{n+3\choose 3}-{(n-2)+3\choose 3}={n+1\choose 1}+2{n+1\choose 2}=(n+1)^2.
\end{equation}
\end{proof}

\begin{remark}
Since 
\begin{eqnarray}
(\widetilde{M}_{i,j}^t)^{\transp}&=&U^{\transp}(M_{i,j}^t)^{\transp}U\\
&=&U^{\transp}M_{j,i}^t U\nonumber\\
&=&\widetilde{M}_{j,i}^t,\nonumber
\end{eqnarray}
if follows from Proposition \ref{imageinblok} that $\beta_{i,j,k}^t=\beta_{j,i,k}^t$, which is not obvious from the definition of $\beta_{i,j,k}^t$. In \cite{Lexcodes}, Proposition \ref{imageinblok} is derived in a slightly different manner, resulting in a different expression for $\beta_{i,j,k}^t$ which displays the symmetry between $i$ and $j$:
\begin{equation}
\beta_{i,j,k}^t=\sum_{u=0}^n (-1)^{u-t}{u\choose t}{n-2k\choose u-k}{n-k-u\choose i-u}{n-k-u\choose j-u}.
\end{equation}  
\end{remark}

\section{Block-diagonalisation of $\mathcal{A}_{q,n}$}
In this section we give an explicit block diagonalisation of the algebra $\mathcal{A}_{q,n}$. The block diagonalisation can be seen as an extension of the block diagonalisation in the binary case as given in the previous section. There the binary Hamming space was taken to be the collection of subsets $\P$ of $\{1,\ldots,n\}$. Now it will be convenient to replace this by the collection of subsets of a given finite set $V$. Let $V$ be a finite set of cardinality $|V|=m$. By $\mathcal{P}(V)$ we denote the collection of subsets of $V$. For integers $i,j$, define the $\mathcal{P}(V)\times \mathcal{P}(V)$ matrix $C_{i,j}^V$ by
\begin{equation}
(C_{i,j}^V)_{I,J}:=\begin{cases}
1&\text{if $|I|=i$, $|J|=j$, $I\subseteq J$ or $J\subseteq I$,}\\
0&\text{otherwise.}
\end{cases}
\end{equation}
The matrices $C_{i,k}^V$ correspond to the matrices $M_{i,k}^k$ from the binary Terwilliger algebra. We have renamed them in order to avoid confusion with the matrices $M_{i,j}^{t,p}\in\mathcal{A}_{q,n}$.
For $k=0,\ldots,\lfloor\frac{m}{2}\rfloor$ define the linear space $L_k^V$ by
\begin{equation}
L_k^V:=\{x\in \Complex^{\mathcal{P(V)}}\mid C_{k-1,k}^V x=0,\ x_I=0 \text{\ if $|I|\not=k$}\},
\end{equation}
and let $B_k^V$ be an orthonormal base of $L_k^V$.
For $i,j,k,t\in\{0,\ldots,m\}$, define the number
\begin{equation}
\beta_{i,j,k}^{m,t}:={m-2k\choose i-k}\sum_{p=0}^m (-1)^{k-p}{k\choose p}{i-p\choose t-p}{m+p-i-k\choose m+t-i-j}.
\end{equation}

We recall the following facts.
\begin{proposition}\label{binarycase}
Let $V$ be a finite set of cardinality $m$. Then
\begin{enumerate}
\item[(i)] For $k\in\{0,\ldots,\rounddown{\frac{m}{2}}\}$ we have
\begin{equation} \dim L_k^V={m\choose k}-{m\choose k-1}.
\end{equation}
\item[(ii)] For $i,k,l\in\{0,\ldots,n\}$ with $k,l\leq \rounddown{\frac{m}{2}}$, and for $b\in L_k^V,c\in L_l^V$
\begin{equation}
(C_{i,l}^V c)^{\transp}C_{i,k}^V b =\begin{cases}
{m-2k \choose i-k} c^{\transp}b&\text{if $k=l$},\\
0&\text{otherwise.}
\end{cases}
\end{equation}
\item[(iii)] For $i,j,k,t\in\{0,\ldots,n\}$ with $k\leq \rounddown{\frac{m}{2}}$, $b\in L_k^V$ and $Y\subseteq V$ with $|Y|=j$
\begin{equation}
\sum_{\substack{U\subseteq V\\|U|=i\\|U\cap Y|=t}} (C_{i,k}^V b)_U=\beta_{i,j,k}^{m,t} {m-2k \choose j-k}^{-1} (C_{j,k}^V b)_Y.
\end{equation}
\end{enumerate}
\end{proposition}
\begin{proof}
Items (i),(ii) and (iii) follow directly from Propositions \ref{dimlk}, \ref{innerprod} and \ref{actionmijt}.  
\end{proof}

We will now describe a block diagonalisation of $\mathcal{A}_{q,n}$. Let $\phi\in\mathbb{C}$ be a primitive $(q-1)$-th root of unity. Let
\begin{eqnarray}
\mathcal{V}&:=&\{(a,k,i,\mathbf{a},b)\mid\\
&&\quad a,k,i\text{\ are integers satisfying\ } 0\leq a\leq k\leq i\leq n+a-k,\nonumber\\
&&\quad \mathbf{a}\in\q^n\text{\ satisfies\ }|S(\mathbf{a})|=a, \mathbf{a}_h\not=q-1\text{\ for\ } h=1,\ldots, n,\nonumber\\
&&\quad b\in B_{k-a}^{\overline{S(\mathbf{a})}}\},\nonumber
\end{eqnarray}
where $\overline{U}:=\{1,2,\ldots,n\}\setminus U$ for any set $U\subseteq \{1,2,\ldots,n\}$. For each tuple $(a,k,i,\mathbf{a},b)$ in $\mathcal{V}$, define the vector $\Psi_{\mathbf{a},b}^{a,k,i}\in \mathbb{C}^{\q^n}$ by
\begin{equation}
\begin{split}
&(\Psi_{\mathbf{a},b}^{a,k,i})_{\x}:=\\
&\quad\begin{cases}
(q-1)^{-\frac{1}{2}i}{n+a-2k \choose i-k}^{-\frac{1}{2}}\phi^{\left<\mathbf{a},\x\right>}(C_{i-a,k-a}^{\overline{S(\mathbf{a})}} b)_{S(\x)\setminus S(\mathbf{a})}&\text{if $S(\mathbf{a})\subseteq S(\x)$},\\
0&\text{otherwise,}
\end{cases}
\end{split}
\end{equation}
for any $\x\in\q^n$. Here the nonnegative integer $\left<\x,\y\right>$ is given by 
\begin{equation}
\left<\x,\y\right>:=\sum_{h=0}^{n} \x_h\y_h 
\end{equation}
for any $\x,\y\in \q^n$. We stress that $\left<\x,\y\right>$ is \emph{not} taken modulo $q$. Observe 
that $(\Psi_{\mathbf{a},b}^{a,k,i})_{\x}=0$ if $|S(\x)|\not=i$.
We have:
\begin{proposition}
The vectors $\Psi_{\mathbf{a},b}^{a,k,i},\ (a,k,i,\mathbf{a},b)\in \mathcal{V}$ form an orthonormal base of $\mathbb{C}^{\q^n}$.
\end{proposition}
\begin{proof} First, the number $|\mathcal{V}|$ of vectors $\Psi_{\mathbf{a},b}^{a,k,i}$ equals $q^n$ since:
\begin{equation}
\begin{split}
\sum_{\substack{a,k,i\\0\leq a\leq k\leq i\leq n+a-k}} &{n \choose a} (q-2)^a\left[{n-a \choose k-a}-{n-a \choose k-a-1}\right]\\
=&\sum_{i=0}^n \sum_{a=0}^i  {n \choose a} (q-2)^a\sum_{k=a}^{\min (i,n+a-i)}\left[{n-a \choose k-a}-{n-a \choose k-a-1}\right]\\
=&\sum_{i=0}^n\sum_{a=0}^i {n \choose a} (q-2)^a {n-a \choose \min\{i-a,n-i\}}\\
=&\sum_{i=0}^n{n \choose i}\sum_{a=0}^i (q-2)^a {i \choose a}\\
=&\sum_{i=0}^n{n\choose i}(q-1)^i=q^n.
\end{split}
\end{equation}
Secondly, we calculate the inner product of $\Psi_{\mathbf{a},b}^{a,k,i}$ and $\Psi_{\mathbf{a}',b'}^{a',k',i'}$. If $i\not=i'$ then the inner product is zero since the two vectors have disjoint support. So we may assume that $i'=i$. We obtain:
\begin{equation}\label{innerproduct}
\begin{split}
\left<\Psi_{\mathbf{a},b}^{a,k,i},\right.&\left. \Psi_{\mathbf{a}',b'}^{a',k',i}\right> =(q-1)^{-i}{n+a-2k \choose i-k}^{-\frac{1}{2}}{n+a'-2k'\choose i-k'}^{-\frac{1}{2}}\\
&\quad\cdot\sum_{\x} \phi^{\left<\mathbf{a},\x\right>-\left<\mathbf{a}',\x\right>} (C_{i-a,k-a}^{\overline{S(\mathbf{a})}}b)_{S(\x)\setminus S(\mathbf{a})}\cdot (C_{i-a',k'-a'}^{\overline{S(\mathbf{a}')}}b')_{S(\x)\setminus S(\mathbf{a}')},
\end{split}
\end{equation}
where the sum ranges over all $\x\in\q^n$ with $|S(\x)|=i$ and $S(\x)\supseteq S(\mathbf{a})\cup S(\mathbf{a}')$.
If $\mathbf{a}_j\not= \mathbf{a}'_j$ for some $j$, then the inner product equals zero, since we can factor out $\sum_{\x_j=1}^{q-1} \phi^{\x_j(\mathbf{a}_j-\mathbf{a}'_j)}=0$. So we may assume that $\mathbf{a}=\mathbf{a}'$ (and hence $a=a'$), which simplifies the right-hand side of (\ref{innerproduct}) to
\begin{equation}\label{simpleform}
{n+a-2k \choose i-k}^{-\frac{1}{2}}{n+a-2k'\choose i-k'}^{-\frac{1}{2}} (C_{i-a,k-a}^{\overline{S(\mathbf{a})}}b)^{\transp}C_{i-a,k'-a}^{\overline{S(\mathbf{a})}}b'.
\end{equation}
Indeed, since $\mathbf{a}'=\mathbf{a}$, we observe that
\begin{equation}
\phi^{\left<\mathbf{a},\x\right>-\left<\mathbf{a},\x\right>}=1,
\end{equation}
and hence the summand only depends on the support of $\x$. We obtain 
\begin{equation}
\begin{split}
\sum_{\substack{\x\\|S(\x)|=i, S(x)\supseteq S(\mathbf{a})}} &(C_{i-a,k-a}^{\overline{S(\mathbf{a})}}b)_{S(\x)\setminus S(\mathbf{a})}\cdot (C_{i-a,k'-a}^{\overline{S(\mathbf{a})}}b')_{S(\x)\setminus S(\mathbf{a})}\\
&=\sum_{\substack{X\\|X|=i, X\supseteq S(\mathbf{a})}} (q-1)^i(C_{i-a,k-a}^{\overline{S(\mathbf{a})}}b)_{X\setminus S(\mathbf{a})}\cdot (C_{i-a,k'-a}^{\overline{S(\mathbf{a})}}b')_{X\setminus S(\mathbf{a})}\\
&=(q-1)^i\sum_{\substack{Y\subseteq \overline{S(\mathbf{a})}\\|Y|=i-a}}(C_{i-a,k-a}^{\overline{S(\mathbf{a})}}b)_{Y}\cdot (C_{i-a,k'-a}^{\overline{S(\mathbf{a})}}b')_{Y}\\
&=(q-1)^i(C_{i-a,k-a}^{\overline{S(\mathbf{a})}}b)^{\transp}C_{i-a,k'-a}^{\overline{S(\mathbf{a})}}b'.
\end{split}
\end{equation}
From equation (\ref{simpleform}) and Proposition \ref{binarycase} we conclude that $\left<\Psi_{\mathbf{a},b}^{a,k,i},\Psi_{\mathbf{a},b'}^{a,k',i}\right>$ is nonzero only if $k=k'$ and $b=b'$, in which case the inner product equals $1$.
\end{proof}

The block diagonalisation will follow by writing the matrices $M_{i,j}^{t,p}$ with respect to the new orthonormal basis of $\mathbb{C}^{\q^n}$ formed by the vectors $\Psi_{\mathbf{a},b}^{a,k,i}$. To this end we define for $i,j,t,p,a,k\in\{0,\ldots,n\}$ with $a\leq k\leq i,j$ the number
\begin{equation}
\begin{split}
\alpha(i,j,t,p,a,k):=&\beta_{i-a,j-a,k-a}^{n-a,t-a} (q-1)^{\frac{1}{2}(i+j)-t}\\
&\quad\cdot\sum_{g=0}^p (-1)^{a-g}{a \choose g}{t-a \choose p-g}(q-2)^{t-a-p+g}.
\end{split}
\end{equation}
We obtain the following.
\begin{proposition}\label{blockdiag}
For $(i,j,t,p)\in\mathcal{I}(q,n)$ and $(a,k,i',\mathbf{a},b)\in \mathcal{V}$ we have:
\begin{equation}\label{matrixaction}
\begin{split}
&M_{j,i}^{t,p}\Psi_{\mathbf{a},b}^{a,k,i'}=\\
&\quad\delta_{i,i'}{n+a-2k \choose i-k}^{-\frac{1}{2}} {n+a-2k \choose j-k}^{-\frac{1}{2}} \alpha(i,j,t,p,a,k) \Psi_{\mathbf{a},b}^{a,k,j}.
\end{split}
\end{equation}
\end{proposition}
\begin{proof}
Clearly, both sides of (\ref{matrixaction}) are zero if $i\not=i'$, hence we may assume that $i=i'$. We calculate $(M_{j,i}^{t,p}\Psi_{\mathbf{a},b}^{a,k,i})_{\y}$. We may assume that $|S(\y)|=j$, since otherwise both sides of (\ref{matrixaction}) have a zero in position $\y$. We have:

\begin{eqnarray}
(M_{j,i}^{t,p}\Psi_{\mathbf{a},b}^{a,k,i})_{\y}&=&\sum_{\x\in\q^n} (M_{j,i}^{t,p})_{\y,\x} (\Psi_{\mathbf{a},b}^{a,k,i})_{\x}\\
&=&(q-1)^{-\frac{1}{2}i}{n+a-2k \choose i-k}^{-\frac{1}{2}}\sum_{\x}\phi^{\left<\x,\mathbf{a} \right>} (C_{i-a,k-a}^{\overline{S(\mathbf{a})}} b)_{S(\x)\setminus S(\mathbf{a})},\nonumber
\end{eqnarray}
where the last sum ranges over all $\x\in\q^n$ with $|S(\x)|=i$, $S(\x)\supseteq S(\mathbf{a})$, $|S(\x)\cap S(\y)|=t$ and $|\{h\mid \x_h=\y_h\not=0\}|=p$.

We will work out the sum:
\begin{equation}
\sum_{\substack{\x\\|S(\x)|=i, S(\x)\supseteq S(\mathbf{a})\\|S(\x)\cap S(\y)|=t\\|\{h\mid \x_h=\y_h\not=0\}|}}\phi^{\left<\x,\mathbf{a} \right>} (C_{i-a,k-a}^{\overline{S(\mathbf{a})}} b)_{S(\x)\setminus S(\mathbf{a})}.
\end{equation}
 
If there exists an $h\in S(\mathbf{a})\setminus S(\y)$, we can factor out $\sum_{l=1}^{q-1} \phi^{l\cdot\mathbf{a}_h}=0$, implying that both sides of (\ref{matrixaction}) have a zero at position $\y$. Hence we may assume that $S(\y)\supseteq S(\mathbf{a})$. Now the support of each word $\x$ in this sum can be split into five parts $U,U',V,V',W$, where
\begin{eqnarray}
U&=&\{h\in S(\mathbf{a})\mid \x_h=\y_h\}\\
U'&=&S(\mathbf{a})\setminus U,\nonumber\\
V&=&\{h\in S(\y)\setminus S(\mathbf{a})\mid \x_h=\y_h\},\nonumber\\
V'&=&((S(\y)\setminus S(\mathbf{a}))\cap S(\x))\setminus V,\nonumber\\
W&=&S(\x)\setminus S(\y).\nonumber 
\end{eqnarray}
Setting $g:=|U|$, gives $|U'|=a-g$, $|V|=p-g$, $|V'|=t-a-p+g$ and $|W|=i-t$. Hence splitting the sum over $g$, we obtain:
\begin{multline} \sum_{g=0}^p \:\:\sum_{U,U',V,V',W} (C_{i-a,k-a}^{\overline{S(\mathbf{a})}} b)_{V\cup V'\cup W}\\
\prod_{h\in U} \phi^{\mathbf{a}_h \y_h} \prod_{h\in U'} (-\phi^{\mathbf{a}_h \y_h}) \prod_{h\in V} 1 \prod_{h\in V'} (q-2) \prod_{h\in W} (q-1),
\end{multline}
where $U,U',V,V',W$ are as indicated. Substituting $T=V\cup V'\cup W$, we can rewrite this as
\begin{multline}
\sum_{g=0}^p {a \choose g} {t-a \choose p-g}(-1)^{a-g} (q-2)^{t-a-p+g}\cdot\\
(q-1)^{i-t} \phi^{\left<\mathbf{a},\y\right>} \sum_{T} (C_{i-a,k-a}^{\overline{S(\mathbf{a})}} b)_T,
\end{multline}
where the sum ranges over all $T\subseteq \overline{S(\mathbf{a})}$ with $|T|=i-a$ and $|T\cap S(\y)|=t-a$. 
Now by Proposition \ref{binarycase}(iii), this is equal to
\begin{multline}
(q-1)^{i-t} \sum_{g=0}^p {a \choose g} {t-a \choose p-g}(-1)^{a-g} (q-2)^{t-a-p+g}\cdot\\
\phi^{\left<\mathbf{a},\y\right>} {n+a-2k \choose j-k}^{-1}\beta_{i-a,j-a,k-a}^{n-a,t-a} (C_{j-a,k-a}^{\overline{S(\mathbf{a})}} b)_{S(\y)\setminus S(\mathbf{a})},
\end{multline}
which equals
\begin{multline}
(\Psi_{\mathbf{a},b}^{a,k,j})_{\y} \cdot \beta_{i-a,j-a,k-a}^{n-a,t-a} {n+a-2k \choose j-k}^{-\frac{1}{2}} (q-1)^{i-t+\frac{1}{2}j}\cdot\\
\sum_{g=0}^p (-1)^{a-g} {a \choose g} {t-a \choose p-g} (q-2)^{t-a-p+g}.
\end{multline}
This completes the proof. 
\end{proof}

If we define $U$ to be the $\q^n\times \mathcal{V}$ matrix with $\Psi_{\mathbf{a},b}^{a,k,i}$ as the $(a,k,i,\mathbf{a},b)$-th column, then Proposition \ref{blockdiag} shows that for each $(i,j,t,p)\in \mathcal{I}(q,n)$ the matrix $\widetilde{M}_{i,j}^{t,p}:=U^{\ast} M_{i,j}^{t,p} U$ has entries
\begin{multline}\label{blockform}
(\widetilde{M}_{i,j}^{t,p})_{(a,k,l,\mathbf{a},b),(a',k',l',\mathbf{a}',b')}=\\
\begin{cases}
{n+a-2k \choose i-k}^{-\frac{1}{2}} {n+a-2k \choose j-k}^{-\frac{1}{2}} \alpha(i,j,t,p,a,k)&\text{if $a=a'$, $k=k'$, $\mathbf{a}=\mathbf{a}'$, $b=b'$ and}\\
&\text{$l=i$, $l'=j$},\\
0&\text{otherwise.}
\end{cases}
\end{multline}
This implies
\begin{proposition}
The matrix $U$ gives a block diagonalisation of $\mathcal{A}_{q,n}$.
\end{proposition}
\begin{proof}
Equation (\ref{blockform}) implies that each matrix $\widetilde{M}_{i,j}^{t,p}$ has a block diagonal form, where for each pair $(a,k)$ there are ${n \choose a}(q-2)^a\left[{n-a\choose k-a }-{n-a\choose n-a-1}\right]$ copies of an $(n+a+1-2k)\times (n+a+1-2k)$ block on the diagonal. For fixed $a,k$ the copies are indexed by the pairs $(\mathbf{a},b)$ such that $\mathbf{a}\in\q^n$ satisfies $|S(\mathbf{a})|=a$, $\mathbf{a}_h\not=q-1$ for $h=1,\ldots,n$, and $b\in B_{k-a}^{\overline{S(\mathbf{a})}}$. In each copy the rows and columns in the block are indexed by the integers $i$ with $k\leq i\leq n+a-k$. Hence we need to show that all matrices of this block diagonal form belong to $U^{\ast}\mathcal{A}_{q,n}U$. It suffices to show that the dimension $\sum_{0\leq a\leq k\leq n+a-k} (n+a+1-2k)^2$ of the algebra consisting of the matrices in the given block diagonal form equals the dimension of $\mathcal{A}_{q,n}$, which is $n+4\choose 4$. This follows from
\begin{equation}
\begin{split}
\sum_{0\leq a\leq k\leq n+a-k} &(n+a+1-2k)^2\\
&=\sum_{a=0}^n\sum_{k=0}^{\lfloor\frac{n-a}{2}\rfloor} (n-a+1-2k)^2\\
&=\sum_{a=0}^n {n-a+3\choose 3}\\
&={n+4\choose 4}.
\end{split}
\end{equation}
\end{proof}

This implies the following result.  
\begin{theorem}\label{blokform1}
The matrix 
\begin{equation}
M=\sum_{(i,j,t,p)} x_{i,j}^{t,p} M_{i,j}^{t,p}
\end{equation}
is positive semidefinite if and only if for all $a,k$ with $0\leq a\leq k\leq n+a-k$ the matrices
\begin{equation}\label{blok}
\left(\sum_{t,p}\alpha(i,j,t,p,a,k)x_{i,j}^{t,p}\right)_{i,j=k}^{n+a-k}
\end{equation}
are positive semidefinite.
\end{theorem}
\begin{proof}
The matrix $M$ is positive semidefinite if and only if $U^{*}MU$ is positive semidefinite. Since $U^{*}MU$ is in block diagonal form, where the blocks are exactly the matrices in (\ref{blok}), each with multiplicity at least one, the theorem follows.
\end{proof}

\begin{theorem}\label{blokform2}
The matrix 
\begin{equation}
R\left( \sum_{(i,j,t,p)} x_{i,j}^{t,p} M_{i,j}^{t,p}\right)
\end{equation}
is positive semidefinite if and only if for all $a,k$ with $0\leq a\leq k\leq n+a-k$ and $k\not=0$ the matrix
\begin{equation}
\left(\sum_{t,p}\alpha(i,j,t,p,a,k)x_{i,j}^{t,p}\right)_{i,j=k}^{n+a-k}
\end{equation}
is positive semidefinite, and also the matrix 
\begin{equation}
\begin{pmatrix}
1&x^{\transp}\\
x&L
\end{pmatrix}
\end{equation}
is positive semidefinite, where
\begin{equation}
L:=\left(\sum_{t,p}\alpha(i,j,t,p,0,0)x_{i,j}^{t,p}\right)_{i,j=0}^{n}
\end{equation}
and
\begin{equation}
x_i={n\choose i}(q-1)^i\cdot x_{i,i}^{i,i}\quad\text{for $i=0,\ldots,n$.}
\end{equation}
\end{theorem}
\begin{proof}
Let 
\begin{equation}
M:=\sum_{(i,j,t,p)} x_{i,j}^{t,p} M_{i,j}^{t,p}.
\end{equation}
Observe that
\begin{equation}
\begin{pmatrix}1&0\\0&U\end{pmatrix}^{*}R(M)\begin{pmatrix}1&0\\0&U\end{pmatrix}=\begin{pmatrix}1&(\diag (M))^{\transp}U\\U^{*}\diag (M)&U^{*}MU\end{pmatrix}.
\end{equation}
Since 
\begin{equation}
\chi^{S_i(\zero)}={n\choose i}(q-1)^i\Psi_{\zero,1}^{0,0,i}={n\choose i}(q-1)^iU_{(0,0,i,\zero,1)}
\end{equation}
and $U^{*}U=I$, we see that 
\begin{eqnarray}
U^{*}\diag (M)&=&U^{*}\sum_{i=0}^n x_{i,i}^{i,i}\chi^{S_i(\zero)}\\
&=&\sum_{i=0}^n x_{i,i}^{i,i}{n\choose i}(q-1)^i U^{*}U_{(0,0,i,\zero,1)}\nonumber\\
&=&\sum_{i=0}^n x_{i,i}^{i,i}{n\choose i}(q-1)^i \chi^{(0,0,i,\zero,1)}\nonumber
\end{eqnarray}
has nonzero entries only in the block corresponding to $a=k=0$. The theorem follows.
\end{proof}

\section{The Terwilliger algebra of the Johnson scheme}
The Hamming scheme is a natural and powerful tool in studying subsets of the binary Hamming space with prescribed distance relations. In particular, the Delsarte bound gives good upper bounds on the size of a code. In the case of constant weight codes, one considers subsets of the Johnson space, consisting of the subsets of some fixed size $w$. Now the appropiate tool to use is the \notion{Johnson scheme}. 

Let $w\leq n$ be positive integers and let $\P_n^w$ be the collection of subsets of $\{1,\ldots,n\}$ of cardinality $w$. So $\mathcal{P}_n$ is the disjoint union of $\mathcal{P}_n^0,\mathcal{P}_n^1,\ldots,\mathcal{P}_n^n$. We will assume that $w\leq \rounddown{\frac{n}{2}}$. This is not a severe restriction since $\P_n^w$ and $P_n^{n-w}$, with the Hamming distance, are isomorphic. This is because the Hamming distance is preserved under taking complements: $d(U,V)=d(\overline{U},\overline{V})$ for sets $U,V\in\{1,\ldots,n\}$. It is convenient to define the \notion{Johnson distance} $d_J$ by
\begin{equation}
d_J(U,V):=w-|U\cap V|=\frac{1}{2}d(U,V).
\end{equation}
We denote by $\autj$ the set of automorphisms of the Johnson space. It is easy to see that the automorphisms are just the permutations of $\P_n^w$ induced by permuting the ground set $\{1,\ldots,n\}$. The distance relations $R_0,\ldots, R_w$ given by 
\begin{equation}
R_d:=\{(U,V)\in \P_n^w\times \P_n^w\mid d_J(U,V)=d\}
\end{equation}
are precisely the orbits under the action of $\autj$ on $\P_n^w\times \P_n^w$:
\begin{equation}
R_d=\{(\sigma U,\sigma V)\mid \sigma\in \autj\},\quad\text{when $d_J(U,V)=d$}.
\end{equation}
Hence $R_0,\ldots,R_w$ form an association scheme called the \emph{Johnson scheme} $J(n,w)$. The Bose--Mesner algebra of the Johnson scheme is spanned by the matrices $A_d\in\mathbb{C}^{\P_n^w\times\P_n^w}$, $d=0,\ldots,w$ given by
\begin{equation}
(A_d)_{U,V}:=\begin{cases}1&\text{if $d_J(U,V)=d$}\\0&\text{otherwise}\end{cases}.
\end{equation}

Like in the case of the Hamming scheme, it useful to consider the refinement of the Johnson scheme obtained by replacing the full symmetry group $\autj$ by the stabilizer subgroup $\stabj$ of some arbitrary element $W\in \P_n^w$. Therefore we fix some $W\in\P_n^w$. Consider the complex algebra $\mathcal{T}$ spanned by the $0$--$1$ matrices $M_{i,j}^{s,t}$ where $0\leq s\leq i,j\leq w$ and $t\leq w-i,w-j$, given by
\begin{equation}
(M_{i,j}^{s,t})_{U,V}:=\begin{cases}1&\text{if $|U\cap W|=i,|V\cap W|=j,$}\\
&\text{$|U\cap V\cap W|=s, |U\cap V\setminus W|=t$}\\
0&\text{otherwise}\end{cases}.
\end{equation}
It is not hard to verify that supports of the matrices $M_{i,j}^{s,t}$ correspond to the orbits of $\P_n^w\times\P_n^w$ under the action of $\stabj$. The algebra $\mathcal{T}$ is in fact the Terwilliger algebra of the Johnson scheme $J(n,w)$ with respect to $W$. 

We will give a block diagonalisation of the Terwilliger algebra of the Johnson scheme. This was implicit in the work of Schrijver (\cite{Lexcodes}).

Let $\mathcal{A}_{w,n-w}:=\mathcal{A}_w\otimes\mathcal{A}_{n-w}$ be the tensor product of the algebras $\mathcal{A}_w$ and $\mathcal{A}_{n-w}$. The algebra $\mathcal{A}_{w,n-w}$ consists of the matrices
\begin{equation}\label{generalelt}
\sum_{i,j,t,i',j',t'} x_{i,j,i',j'}^{t,t'} M_{w;i,j}^t\otimes M_{n-w;i',j'}^{t'},
\end{equation}
where $x_{i,j,i',j'}^{t,t'}\in \mathbb{C}$. From Section \ref{Sec:blockdiag} we obtain matrices $U_w$ and $U_{n-w}$ such that $U_w^{\transp}\mathcal{A}_wU_w$ and $U_{n-w}^{\transp}\mathcal{A}_{n-w}U_{n-w}$ are in block diagonal form. It follows that $U:=U_w\otimes U_{n-w}$ block diagonalises $\mathcal{A}_{w,n-w}$ since
\begin{equation}
U^{\transp}\mathcal{A}_{w,n-w}U=U_w^{\transp}\mathcal{A}_wU_w\otimes U_{n-w}^{\transp}\mathcal{A}_{n-w}U_{n-w}.
\end{equation}
It follows from Proposition \ref{imageinblok} that the blocks are indexed by the pairs
\begin{equation}
(k,k')\in \{0,1,\ldots,\rounddown{\frac{w}{2}}\}\times \{0,1,\ldots,\rounddown{\frac{n-w}{2}}\}.
\end{equation}
For each such pair $(k,k')$ we have a block $B_{k,k'}$, consisting of all $(V_k\times V'_{k'})\times (V_k\times V'_{k'})$ matrices, where $V_k:=\{k,\ldots,w-k\}$ and $V'_{k'}:=\{k',\ldots,n-w-k'\}$. The image of (\ref{generalelt}) in block $(k,k')$ is given by
\begin{equation}
\left(\sum_{t,t'}x_{i,j,i',j'}^{t,t'}\beta_{i,j,k}^{w,t}\cdot\beta_{i',j',k'}^{n-w,t'}\textstyle{\left[{w-2k\choose i-k}{w-2k\choose j-k}{n-w-2k'\choose i'-k'}{n-w-2k'\choose j'-k'}\right]^{-\frac{1}{2}}}\right)_{\substack{i,j\in V_k\\i',j'\in V'_{k'}}}
\end{equation}
 
By extending each matrix in $\mathcal{T}$ by zeros to a $\P_n\times \P_n$ matrix, and identifying $\P_n$ with $\P_w\times \mathcal{P}_{n-w}$ (by identifying $U$ and $(U\cap W,U\setminus W)$ for any $U\in \P_n$), the Terwilliger algebra $\mathcal{T}$ can be seen as a subalgebra of $\mathcal{A}_{w,n-w}$, where $M_{i,j}^{s,t}$ is identified with $M_{i,j}^s\otimes M_{w-i,w-j}^t$. It follows that in the block diagonalisation given above, $\mathcal{T}$ is mapped in block $(k,k')$ to those matrices that have nonzeros only in positions with rows and columns indexed by $\{(i,w-i)\mid i\in V_k, w-i\in V'_{k'}\}$. Hence restricting each block to those indices, we obtain a block diagonalisation of $\mathcal{T}$ with blocks of size 
\begin{equation}
|\{k,\ldots,n-k\}\cap \{2w-n+k',\ldots,w-k'\}|
\end{equation}
for each pair $(k,k')$ with $k+k'\leq w$. This was used in \cite{Lexcodes} to obtain bounds on constant weight codes.

In the nonbinary case, let $\E$ be the set of $q$-ary word of length $n$ and weight $w$, equipped with the Hamming distance. The $q$-ary Johnson scheme $J_q(n,w)$ has adjacency matrices $M_{t,p}$ for $0\leq p\leq t\leq w$ given by the orbits of $\E\times \E$ under the action of the automorphism group of $\E$: 
\begin{equation}
(M_{t,p})_{\x,\y}:=\begin{cases}1&\text{if $|S(\x)\cap S(\y)|=t, |\{i\mid \x_i=\y_i\not=0\}|=p$,}\\
0&\text{otherwise}.
\end{cases}
\end{equation}
Replacing the full automorphism group by the stabilizer of some word $\w\in\E$ we obtain an algebra containing the Bose--Mesner algebra of the nonbinary Johnson scheme which may serve to improve bounds for constant weight codes in the nonbinary case. We do not know if this is the Terwilliger algebra (with respect to $\w$) of the nonbinary Johnson scheme $J_q(n,w)$. The algebra is a subalgebra of a tensor product of $\mathcal{A}_{q,n-w}$ and an algebra of dimension $w+9\choose 9$ (or $w+8\choose 8$ if $q=3$). It would be interesting to find a block diagonalisation of this algebra.
  
\chapter{Error correcting codes}\label{CH:errorcodes}
Given a code $C\subseteq \E:=q^n$, the \notion{minimum distance} of $C$ is defined to be the minimum of $\{d(\mathbf{u},\mathbf{v})\mid \mathbf{u}\not=\mathbf{v}, \mathbf{u},\mathbf{v}\in C\}$. The maximum cardinality of a code with minimum distance at least $d$ is denoted by $A_q(n,d)$. In this chapter we give new upper bounds on $A_q(n,d)$ based on a semidefinite programming approach, strengthening Delsarte's linear programming bound.
For more information on coding theory, the reader is referred to \cite{Lint,Sloane}.
\section{Delsarte's linear programming bound}
Given a code $C\subseteq \E$, the $(n+1)$-tuple $(x_0,x_1,\ldots,x_n)$ defined by
\begin{equation}
x_i:=|C|^{-1}\cdot |\{(\mathbf{u},\mathbf{v})\in C\times C\mid d(\mathbf{u},\mathbf{v})=i\}|
\end{equation}
is called the \notion{distance distribution} of the code $C$. For each $i$ the number $x_i$ equals the average number of code words at distance $i$ from a given code word. Observe that $x_0=1$ and $x_0+x_1+\cdots+x_n=|C|$. The key observation that leads to the linear programming bound is that the following inequalities hold:

\begin{equation}\label{krawtchouk}
\sum_{i=0}^n x_i K_j(i)\geq 0 \quad \text{for all $j=0,\ldots,n$},
\end{equation}
where 
\begin{equation}
K_j(x):=\sum_{k=0}^j (-1)^k{x\choose k}{n-x\choose j-k}(q-1)^{j-k},\quad j=0,\ldots,n
\end{equation}
are the \notion{Krawtchouk polynomials}. These inequalities give rise to the following linear programming bound on the size of a code with minimum distance at least $d$:
\begin{eqnarray}
A_q(n,d)\leq \max \{\sum_{i=0}^n x_i \mid&&x_0=1, x_1,\ldots,x_n\geq 0,\\
&&x_1=\cdots=x_{d-1}=0,\nonumber\\
&&\text{the $x_i$ satisfy (\ref{krawtchouk})}\}.\nonumber
\end{eqnarray}
This approach turned out to be very powerful. Many of the best known upper bounds on $A_q(n,d)$ are obtained using this method.   

A proof of the validity of (\ref{krawtchouk}) can be found for example in \cite{Delsarte,Lint}. To illustrate the semidefinite programming approach in this chapter, we sketch a proof here. For any code $C\subseteq \E$, we denote by $M_C$ the $0$--$1$ matrix defined by 
\begin{equation}
(M_C)_{\mathbf{u},\mathbf{v}}:=\begin{cases}1&\text{if $\mathbf{u},\mathbf{v}\in C$}\\0&\text{otherwise}\end{cases}.
\end{equation}
We prove (\ref{krawtchouk}).
\begin{proof}
Consider the matrix
\begin{equation}
M:=\frac{1}{|\aut|\cdot|C|}\sum_{\sigma\in\aut}M_{\sigma C}.
\end{equation}
The matrix $M$ is an element of the Bose--Mesner algebra of the Hamming scheme and the coefficients with respect to the adjacency matrices $A_i$ of the Hamming scheme reflect the distance distribution: 
\begin{equation}
M=\sum_{i=0}^n x_i\gamma_i^{-1}A_i,
\end{equation} 
where 
\begin{equation}
\gamma_i:=q^n(q-1)^i{n\choose i}
\end{equation}
is the number of nonzero entries of $A_i$. Indeed, we have $\left<A_i,M_{\sigma C}\right>=|C|x_i$, and hence $\left<A_i,M\right>=x_i$ for every $i=0,\ldots,n$ and every $\sigma\in \aut$. 

The matrix $M$ is a nonnegative combination of the positive semidefinite matrices $M_{\sigma C}$ and is therefore positive semidefinite itself. The inequalities (\ref{krawtchouk}) will follow from this semidefiniteness by diagonalising the Bose--Mesner algebra. Let the unitary matrix $U\in\mathbb{C}^{\E\times\E}$ be given by 
\begin{equation}
(U)_{\mathbf{u},\mathbf{v}}:=q^{-n/2}\phi^{\left<\mathbf{u},\mathbf{v}\right>}
\end{equation}  
for $\mathbf{u},\mathbf{v}\in\E$, where $\phi$ is a primitive $q$-th root of unity. It is a straightforward calculation to show that for each $i=0,\ldots,n$ the matrix
$\widetilde{A}_i:=U^{\ast}A_iU$ is a diagonal matrix with
\begin{equation}
(\widetilde{A}_i)_{\mathbf{u},\mathbf{u}}=K_i(j)=\gamma_i\gamma_j^{-1}K_j(i)\quad\text{when $d(\zero,\mathbf{u})=j$}.
\end{equation}
Since $M$ is positive semidefinite, also the diagonal matrix $U^{\ast}MU$ is positive semidefinite, which means that all diagonal elements $\sum_{i=0}^n x_iK_j(i)\gamma_j^{-1}$, $j=0,\ldots n$ are nonnegative. This implies (\ref{krawtchouk}).
\end{proof}
In fact, the equivalence of $U^{\ast}MU\psd$ and $M\psd$ shows the following, which we mention for future reference.
\begin{proposition}\label{Krawtchouk}
For $x_0,x_1,\ldots, x_n\in \Real$, we have
\begin{eqnarray}
x_0A_0+xc_1A_1+\cdots+x_nA_n\psd\quad\text{if and only if}\\
x_0K_j(0)+x_1K_j(1)+\cdots+x_nK_j(n)\geq 0\quad\text{for $j=0,\ldots,n$}.\nonumber
\end{eqnarray}
\end{proposition}

\section{Semidefinite programming bound}
In this section we describe a way to obtain upper bounds on $A_q(n,d)$ by semidefinite programming. The method strengthens Delsarte's linear programming bound and was introduced by Schrijver in \cite{Lexcodes} in the case of binary codes. There it was used to find a large number of improved bounds for binary codes. While this thesis was being written, the same method was used by de Klerk and Pasechnik in \cite{orthogonalitygraph} to bound the stability number of \notion{orthogonality graphs}, (or equivalently) the maximum size of a binary code of length $n$ in which no two words have Hamming distance $\frac{1}{2}n$, where $n$ is divisible by four.

In this section we will describe this method, but restrict ourselves to the nonbinary case. In Section \ref{Sec:comp1} we give a list of improved upper bounds that we found with this method for $q=3,4,5$. 

Let $C$ be any code. We define the matrices $M'$ and $M''$ by:
\begin{eqnarray}
M'&:=&|\aut |^{-1}\sum_{\substack{\sigma\in\aut\\ \zero\in \sigma C}}M_{\sigma C}\\
M''&:=&|\aut |^{-1}\sum_{\substack{\sigma\in\aut\\ \zero\not\in \sigma C}}M_{\sigma C}.\nonumber
\end{eqnarray}

By construction, the matrices $M'$ and $M''$ are invariant under permutations $\sigma\in\stab$ of the rows and columns. Hence $M'$ and $M''$ are elements of the algebra $\mathcal{A}_{q,n}$. We write
\begin{equation}\label{Mcoefficients}
M'=\sum_{(i,j,t,p)} x_{i,j}^{t,p} M_{i,j}^{t,p}.
\end{equation}
Here the matrices $M_{i,j}^{t,p}$ are the standard basis matrices of the algebra $\mathcal{A}_{q,n}$.

The matrix $M''$ can be expressed in terms of the coefficients $x_{i,j}^{t,p}$ as follows.

\begin{proposition}
The matrix $M''$ satisfies
\begin{equation}\label{defineM}
M''=\sum_{(i,j,t,p)}(x_{i+j-t-p,0}^{0,0}-x_{i,j}^{t,p})M_{i,j}^{t,p}.
\end{equation}
\end{proposition} 
\begin{proof}
The matrix
\begin{equation}
M:=M'+M''=|\aut|^{-1}\sum_{\sigma\in\aut} M_{\sigma C}
\end{equation}
is invariant under permutation of the rows and columns by any permutation $\sigma\in\aut$, and hence is an element of the Bose--Mesner algebra, say
\begin{equation}
M=\sum_{k} y_k A_k.
\end{equation}
Observe that for any $\mathbf{u}\in\E$ with $d(\mathbf{u},\zero)=k$, we have 
\begin{equation}
y_k=(M)_{\mathbf{u},\zero}=(M')_{\mathbf{u},\zero}=x_{k,0}^{0,0},
\end{equation}
since $(M'')_{\mathbf{u},\zero}=0$. Hence we have
\begin{eqnarray}
M''&=&M-M'\\
&=&\sum_{k} x_{k,0}^{0,0} A_k-\sum_{(i,j,t,p)} x_{i,j}^{t,p} M_{i,j}^{t,p}\nonumber\\ &=&\sum_{k}\sum_{i+j-t-p=k} x_{k,0}^{0,0} M_{i,j}^{t,p} -\sum_{(i,j,t,p)} x_{i,j}^{t,p} M_{i,j}^{t,p}\nonumber\\
&=&\sum_{(i,j,t,p)} (x_{i+j-t-p,0}^{0,0}-x_{i,j}^{t,p}) M_{i,j}^{t,p},\nonumber
\end{eqnarray}
which proves the proposition.
\end{proof}
The coefficients $x_{i,j}^{t,p}$ carry important information about the code $C$, comparable to the distance distribution in Delsarte's linear programming approach. Where the distance distribution records for each distance $d$ the number of pairs in $C$ at distance $d$, the coefficients $x_{i,j}^{t,p}$ count the number of \emph{triples} $(\mathbf{u},\mathbf{v},\mathbf{w})\in C^3$ for each equivalence class of $\E^3$ under the action of $\aut$. We express this formally as follows. Recall that 
\begin{equation}
X_{i,j,t,p}:=\{(\mathbf{u},\mathbf{v},\mathbf{w})\in\E\times\E\times\E\mid d(\mathbf{u},\mathbf{v},\mathbf{w})=(i,j,t,p)\},
\end{equation}
for $(i,j,t,p)\in \mathcal{I}(q,n)$. Now denote for each $(i,j,t,p)\in \mathcal{I}(q,n)$ the numbers
\begin{equation}
\lambda_{i,j}^{t,p}:=|(C\times C\times C)\cap X_{i,j,t,p}|,
\end{equation}
and let
\begin{equation}
\gamma_{i,j}^{t,p}:=|(\{\zero\}\times \E\times\E)\cap X_{i,j,t,p}|
\end{equation}
be the number of nonzero entries of $M_{i,j}^{t,p}$. A simple calculation yields:
\begin{equation}
\gamma_{i,j}^{t,p}=(q-1)^{i+j-t} (q-2)^{t-p} {n \choose p,t-p,i-t,j-t}.
\end{equation}
The numbers $x_{i,j}^{t,p}$ are related to the numbers $\lambda_{i,j}^{t,p}$ by
\begin{proposition}\label{expressxijtp}
$x_{i,j}^{t,p}=q^{-n}(\gamma_{i,j}^{t,p})^{-1} \lambda_{i,j}^{t,p}$.
\end{proposition}
\begin{proof}
Observe that the matrices $M_{i,j}^{t,p}$ are pairwise orthogonal and that $\left<M_{i,j}^{t,p},M_{i,j}^{t,p}\right>=\gamma_{i,j}^{t,p}$  for $(i,j,t,p)\in\mathcal{I}(q,n)$. Hence
\begin{eqnarray}
\left<M',M_{i,j}^{t,p}\right>&=&|\aut|^{-1}\sum_{u\in C}\sum_{\substack{\sigma\in\aut\\ \sigma \mathbf{u}=\zero}}\left< M_{\sigma C},M_{i,j}^{t,p}\right>\\
&=&|\aut|^{-1}\cdot |\stab|\sum_{\mathbf{u}\in C}\cdot |(\{\mathbf{u}\}\times C\times C)\cap X_{i,j,t,p}|\nonumber\\
&=&q^{-n} |(C\times C\times C) \cap X_{i,j,t,p}|=q^{-n}\lambda_{i,j}^{t,p}\nonumber
\end{eqnarray}
implies that
\begin{equation} M'=q^{-n}\sum_{(i,j,t,p)\in\mathcal{I}(q,n)}\lambda_{i,j}^{t,p}(\gamma_{i,j}^{t,p})^{-1} M_{i,j}^{t,p}.
\end{equation}
Comparing the coefficients of the $M_{i,j}^{t,p}$ with those in (\ref{Mcoefficients}) proves the proposition.
\end{proof}

\begin{proposition}
The $x_{i,j}^{t,p}$ satisfy the following linear constraints, where (iii) holds if $C$ has minimum distance at least $d$:
\begin{eqnarray}\label{linconstraint}
\text{(i)}&&0\leq x_{i,j}^{t,p}\leq x_{i,0}^{0,0}\\
\text{(ii)}&&x_{i,j}^{t,p}=x_{i',j'}^{t',p'}\text{\ if\ } t-p=t'-p'\text{\ and}\nonumber\\
&&(i,j,i+j-t-p)\text{\ is a permutation of\ } (i',j',i'+j'-t'-p')\nonumber\\
\text{(iii)}&&x_{i,j}^{t,p}=0 \text{\ if\ } \{i,j,i+j-t-p\}\cap\{1,2,\ldots,d-1\}\not=\emptyset\nonumber.
\end{eqnarray}
\end{proposition}
\begin{proof}
Conditions (ii) and (iii) follow directly from Proposition \ref{expressxijtp}. Condition (i) follows from the fact that if $M=M_{\sigma C}$ for some $\sigma\in\aut$ with $\zero\in \sigma C$, then $0\leq M_{\mathbf{u},\mathbf{v}}\leq M_{\zero,\mathbf{u}}$ for any $\mathbf{u},\mathbf{v}\in\E$.
\end{proof}

An important feature of the matrices $M'$ and $M''$ is, that they are positive semidefinite. This follows since $M'$ and $M''$ are nonnegative combinations of the matrices $M_{\sigma C}=\chi^{\sigma C}(\chi^{\sigma C})^{\transp}$ which are clearly positive semidefinite. Using the block diagonalisation of $\mathcal{A}_{q,n}$, the positive semidefiniteness of $M'$ and $M''$ is equivalent to:
\begin{gather}\label{semidefinite}
\text{for all $a,k$ with $0\leq a\leq k\leq n+a-k$, the matrices}\\ \left(\sum_{t,p}\alpha(i,j,t,p,a,k)x_{i,j}^{t,p}\right)_{i,j=k}^{n+a-k}\notag\\
\text{and}\notag\\ \left(\sum_{t,p}\alpha(i,j,t,p,a,k)(x_{i+j-t-p,0}^{0,0}-x_{i,j}^{t,p})\right)_{i,j=k}^{n+a-k}\notag\\
\text{are positive semidefinite.}\notag
\end{gather}

If we view the $x_{i,j}^{t,p}$ as variables, we obtain an upper bound on the size of a code of minimum distance $d$ as follows.
\begin{theorem}
The semidefinite programming problem
\begin{eqnarray}\label{primesdp}
\text{maximize } \sum_{i=0}^n {n\choose i}(q-1)^i x_{i,0}^{0,0}\quad \text{subject to}\\
x_{0,0}^{0,0}=1, \text{and conditions (\ref{linconstraint}) and (\ref{semidefinite})}\nonumber
\end{eqnarray}
is an upper bound on $A_q(n,d)$.
\end{theorem}
\begin{proof}
We first remark that conditions (\ref{linconstraint}) and (\ref{semidefinite}) are invariant under scaling the numbers $x_{i,j}^{t,p}$ with a common positive factor. The constraint $x_{0,0}^{0,0}=1$ serves as a normalisation. 
If $C\subseteq \E$ is a code of minimum distance $d$. Setting
\begin{equation}
x_{i,j}^{t,p}:=q^n\cdot \lambda_{i,j}^{t,p}\gamma_{i,j}^{t,p}
\end{equation}
gives a feasible solution with objective value $|C|$. 
\end{proof}

This is a semidefinite programming problem with $O(n^4)$ variables, and can be solved in time polynomial in $n$. This semidefinite programming bound is at least as strong as the Delsarte bound. Indeed, the Delsarte bound is equal to the maximum of $\sum_{i=0}^n x_{i,0}^{0,0}{n\choose i}(q-1)^i$ subject to the conditions $x_{0,0}^{0,0}=1$, $x_{1,0}^{0,0}=\cdots=x_{d-1,0}^{0,0}=0$, $x_{i,0}^{0,0}\geq 0$ for all $i=d,\ldots,n$ and \begin{equation}\label{psddelsarte}
\sum_{i=0}^n x_{i,0}^{0,0} A_i\quad\text{is positive semidefinite},
\end{equation} 
as was shown in the previous section. This last constraint is equivalent to
\begin{equation}
\sum_{i,j,t,p} x_{i+j-t-p,0}^{0,0}M_{i,j}^{t,p}\quad\text{is positive semidefinite},
\end{equation} 
since $A_k=\sum_{i+j-t-p=k} M_{i,j}^{t,p}$. It follows that (\ref{psddelsarte}) is implied by the condition that $M'$ and $M''$ be positive semidefinite, that is condition (\ref{semidefinite}).

\subsection{Variations}
There are a number of obvious variations to the semidefinite program (\ref{primesdp}), altering the objective function and the constraint $x_{0,0}^{0,0}=1$. For convenience we will optimize over matrices 
\begin{equation}
M:=\sum_{i,j,t,p} x_{i,j}^{t,p}M_{i,j}^{t,p}
\end{equation}
in the Terwilliger algebra. Observe that the numbers $x_{i,j}^{t,p}$ are uniquely determined by $M$ and vice versa. The semidefinite program (\ref{primesdp}) can be rewritten as
\begin{equation}\label{original}
\text{maximize } \trace{M}\quad \text{subject to (\ref{linconstraint}), (\ref{semidefinite}) and $x_{0,0}^{0,0}=1$.}
\end{equation}
Consider the following two variations 
\begin{eqnarray}\label{variation1}
\text{maximize}&&x_{0,0}^{0,0}\\
\text{subject to}&&\text{(\ref{linconstraint}), (\ref{semidefinite}) and }
\begin{pmatrix}1&x_{0,0}^{0,0}\\x_{0,0}^{0,0}&\trace M\end{pmatrix}\psd,\nonumber
\end{eqnarray}
and
\begin{eqnarray}\label{variation2}
\text{maximize}&&\one^{\transp}M\one\\
\text{subject to}&&\text{(\ref{linconstraint}), (\ref{semidefinite}) and $\trace{M}=1.$}\nonumber
\end{eqnarray}

The idea behind variation (\ref{variation1}) is that for a code $C$, setting 
\begin{equation}
x_{i,j}^{t,p}:=\lambda_{i,j}^{t,p}\cdot q^n(\gamma_{i,j}^{t,p})^{-1},
\end{equation}
we obtain a feasible solution with $x_{0,0}^{0,0}=|C|$ and $\trace{M}=|C|^{2}$. If $M'=\sum_{i,j,t,p}y_{i,j}^{t,p}M_{i,j}^{t,p}$ is a feasible solution to (\ref{variation1}), then $M:=(y_{0,0}^{0,0})^{-1}M'$ is a feasible solution to (\ref{original}) with $\trace{M}=\trace M'\cdot (y_{0,0}^{0,0})^{-1}\geq y_{0,0}^{0,0}$. Conversely, if $M'=\sum_{i,j,t,p}y_{i,j}^{t,p}M_{i,j}^{t,p}$ is a feasible solution to (\ref{original}), then setting $x_{i,j}^{t,p}:=y_{i,j}^{t,p}\cdot \trace{M'}$ gives a feasible solution to (\ref{variation1}) with $x_{0,0}^{0,0}=1\cdot\trace{M'}$. Hence both semidefinite programs yield the same value.

The validity of variation (\ref{variation2}) can be seen by setting $x_{i,j}^{t,p}:= \lambda_{i,j}^{t,p}\cdot q^n(\gamma_{i,j}^{t,p})^{-1}|C|^{-2}$ for a given code $C$. Then $\trace{M}=1$ and $\one^{\transp}M\one=|C|$. For any feasible solution $M'$ to (\ref{original}), we have $\one^{\transp}M'\one\geq (\trace{M'})^2$, hence $M:=(\trace{M'})^{-1}M'$ is a feasible solution to (\ref{variation2}) with $\one^{\transp}M\one\geq \trace{M'}$. It follows that the optimum value in (\ref{variation2}) is at least the optimum value in (\ref{original}). We do not know if the reverse inequality holds.

\subsection{A strengthening}
It was observed by Laurent (see \cite{Monique}) that not only is the matrix $M''$ defined in (\ref{defineM}) positive semidefinite, also the following stronger property holds: 
\begin{equation}\label{RM} 
\begin{pmatrix}
1-x_{0,0}^{0,0}&(\diag (M''))^{\transp}\\
\diag (M'')& M''
\end{pmatrix}\quad\text{is positive semidefinite}.
\end{equation} 
This follows from the fact that for a code $C$ and $\sigma\in\aut$ the matrix
\begin{equation}
\begin{pmatrix}
1&(\chi^{\sigma C})^{\transp}\\
\chi^{\sigma C}&\chi^{\sigma C}(\chi^{\sigma C})^{\transp}
\end{pmatrix}={1\choose \chi^{\sigma C}}{1\choose \chi^{\sigma C}}^{\transp}\quad\text{is positive semidefinite}
\end{equation}
and the fact that semidefiniteness is preserved under taking nonnegative linear combinations.
This yields the stronger semidefinite programming bound 
\begin{equation}\label{moniquesbound}
\text{maximize } q^n\cdot x_{0,0}^{0,0}\quad \text{subject to (\ref{linconstraint}), (\ref{semidefinite}) and (\ref{RM}),}
\end{equation} 
where 
\begin{equation}
M'':=\sum_{i,j,t,p}(x_{i+j-t-p,0}^{0,0}-x_{i,j}^{t,p})M_{i,j}^{t,p}.
\end{equation}
Observe that condition (\ref{RM}) can be checked in time polynomial in $n$ by Theorem \ref{blokform2}. 
The bound obtained is as least as good as the one obtained from (\ref{original}). Indeed, given a feasible solution to (\ref{moniquesbound}), the matrix $M:=M'+M''$ satisfies: $R(M)$ is positive semidefinite. Hence 
\begin{equation}
q^n\trace{M'}=\one^{\transp}M\one\geq \trace{M}^2=(q^n\cdot x_{0,0}^{0,0})^2.
\end{equation}
This implies that $N':=\frac{1}{x_{0,0}^{0,0}}M'$ is a feasible solution to (\ref{original}) with $\trace{N'}\geq q^n\cdot x_{0,0}^{0,0}$. Hence the optimum in(\ref{moniquesbound}) is at most the optimum in (\ref{original}). In the binary case, this yields an improved bound when $n=25$ and $d=6$. We did not find new improvements using this strengthening in the range $q=3, n\leq 16$, $q=4, n\leq 12$ or $q=5, n\leq 11$.
 
\section{Computational results}\label{Sec:comp1}
The semidefinite programming method was successfully applied to binary codes in \cite{Lexcodes} where a large number of upper bounds were improved. In this section we describe the computational results obtained in the nonbinary case. Apart from the binary case, tables of bounds on $A_{q}(n,d)$ are maintained for $q=3,4,5$. We have limited the computations to these three cases and computed the semidefinite programming bound for the range $n\leq 16$, $n\leq 12$ and $n\leq 11$, respectively. The instances in which we found an improvement over the best upper bound that was known, are summarized in Tables \ref{table1}, \ref{table2} and \ref{table3} below. As a reference we have used the tables given by Brouwer, H\"am\"al\"ainen, \"Osterg\aa rd and Sloane \cite{threecodes} and by Bogdanova, Brouwer, Kapralov and \"Osterg\aa rd \cite{fourcodes} for the cases $q=3$ and $q=4$, along with subsequent improvements recorded on the website of Brouwer \cite{brouwersite} and the table by Bogdanova and \"Osterg\aa rd \cite{fivecodes} for the case $q=5$. 

\begin{table}[ht]\caption{New upper bounds on $A_3(n,d)$\label{table1}}
\begin{center}
\begin{tabular}{|r|r|r|r|r|r|}
\hline &&best&&best upper&\\
&&lower&new&bound&\\
&&bound&upper&previously&Delsarte\\
$n$&$d$&known&bound&known&bound\\
\hline
12&4&4374&6839&7029&7029\\
13&4&8019&19270&19682&19683\\
14&4&24057&54774&59046&59049\\
15&4&72171&149585&153527&153527\\
16&4&216513&424001&434815&434815\\
\hline
12&5&729&1557&1562&1562\\
13&5&2187&4078&4163&4163\\
14&5&6561&10624&10736&10736\\
15&5&6561&29213&29524&29524\\
\hline
13&6&729&1449&1562&1562\\
14&6&2187&3660&3885&4163\\
15&6&2187&9904&10736&10736\\
16&6&6561&27356&29524&29524\\
\hline
14&7&243&805&836&836\\
15&7&729&2204&2268&2268\\
16&7&729&6235&6643&6643\\
\hline
13&8&42&95&103&103\\
15&8&243&685&711&712\\
16&8&297&1923&2079&2079\\
\hline
14&9&31&62&66&81\\
15&9&81&165&166&166\\
\hline
16&10&54&114&117&127\\
\hline
\end{tabular}
\end{center} 
\end{table}

\begin{table}[ht]\caption{New upper bounds on $A_4(n,d)$\label{table2}}
\begin{center}
\begin{tabular}{|r|r|r|r|r|r|}
\hline &&best&&best upper&\\
&&lower&new&bound&\\
&&bound&upper&previously&Delsarte\\
$n$&$d$&known&bound&known&bound\\
\hline 7&4&128&169&179&179\\
8&4&320&611&614&614\\
9&4&1024&2314&2340&2340\\
10&4&4096&8951&9360&9362\\
\hline 10&5&1024&2045&2048&2145\\
\hline 10&6&256&496&512&512\\
11&6&1024&1780&2048&2048\\
12&6&4096&5864&6241&6241\\
\hline 12&7&256&1167&1280&1280\\
\hline
\end{tabular}
\end{center}
\end{table}

\begin{table}[ht]\caption{New upper bounds on $A_5(n,d)$\label{table3}}
\begin{center}
\begin{tabular}{|r|r|r|r|r|r|}
\hline &&best&&best upper&\\
&&lower&new&bound&\\
&&bound&upper&previously&Delsarte\\
$n$&$d$&known&bound&known&bound\\
\hline
7&4&250&545&554&625\\
\hline
7&5&53&108&125&125\\
8&5&160&485&554&625\\
9&5&625&2152&2291&2291\\
10&5&3125&9559&9672&9672\\
11&5&15625&44379&44642&44642\\
\hline
10&6&625&1855&1875&1875\\
11&6&3125&8840&9375&9375\\
\hline
\end{tabular}
\end{center}
\end{table}

\chapter{Covering codes}\label{CH:coveringcodes}
Consider the following combinatorial problem. Given integers $q$, $n$ and $r$, what is the smallest number of Hamming spheres of radius $r$ that cover the Hamming space consisting of all $q$-ary words of length $n$? This covering problem is the dual of the packing problem from the previous chapter. Apart from being an aesthetically appealing combinatorial problem, it has several technical applications, for example to write-once memories and data compression. Another, down to earth, application is to betting systems. In many countries a popular game is played that involves forecasting the outcomes of a set of $n$ (football)matches. Each match can end in three ways: a loss, a tie or a win for the hosting club. The goal is to find an efficient set of bets that is guaranteed to have a forecast with at most one wrong outcome. For this reason the covering problem in the case $q=3$ and $r=1$ is widely known as the \notion{football pool problem}, see \cite{Footballpool}.

In this chapter we show how the method of matrix cuts from Chapter \ref{CH:matrixcuts} can be applied to obtain new lower bounds on the minimum size of covering codes. For a survey of results on covering codes as well as many applications, the reader is referred to \cite{Coveringcodes}.

\section{Definitions and notation}
Let $q\geq 2$ and $n\geq 1$ be integers. Let $\E:=\q^n$ be the Hamming space consisting of all words of length $n$ over the alphabet $\q:=\{0,1,\ldots,q-1\}$. Recall that the Hamming distance $d(\mathbf{u},\mathbf{v})$ of two words $\mathbf{u}, \mathbf{v}\in \E$ is defined as the number of positions in which $\mathbf{u}$ and $\mathbf{v}$ differ. We define $\dd(\mathbf{u},\mathbf{v}):=(i,j,t)$ where $i=d(\mathbf{u},\zero)$, $j=d(\mathbf{v},\zero)$ and $2t=i+j-d(\mathbf{u},\mathbf{v})$. For a word $\mathbf{u}\in \E$, we denote the \emph{support} of $\mathbf{u}$ by $S(\mathbf{u}):=\{i\mid \mathbf{u}_i\not=0\}$. Note that $|S(\mathbf{u})|=d(\mathbf{u},\zero)$, where $\zero$ is the all-zero word. Denote by
\begin{eqnarray}
B_r(\mathbf{u})&:=&\{\mathbf{v}\in\E\mid d(\mathbf{u},\mathbf{v})\leq r\}\quad\text{and}\\
S_r(\mathbf{u})&:=&\{v\in\E\mid d(\mathbf{u},\mathbf{v})=r\}\nonumber
\end{eqnarray}
the ball and the sphere respectively, with center $\mathbf{u}\in\E$ and radius $r$. They are generally referred to as the \notion{Hamming sphere} and the \notion{Hamming ring} with center $\mathbf{u}$ and radius $r$ in the literature. The \notion{covering radius} of a code $C\subseteq E$ is the smallest integer $r$ for which
\begin{equation}
\bigcup_{\mathbf{u}\in C} B_r(\mathbf{u})=\E.
\end{equation} 
A code $C\subseteq \E$ is called an $(n,K,q)r$ code if $|C|=K$ and the covering radius of $C$ is $r$. We denote
\begin{equation}
K_q(n,r):=\min\{K\mid \text{there exists an $(n,K,q)r$ code}\}.
\end{equation}
In this chapter we will be interested in lower bounds on $K_q(n,r)$.

\section{Method of linear inequalities}
An important tool used in deriving lower bounds on $K_q(n,r)$ is the method of linear inequalities. Let $C\subseteq \E$ be a code and denote
\begin{equation}
A_i(\mathbf{u}):=|C\cap S_i(\mathbf{u})|
\end{equation} 
for $\mathbf{u}\in\E$ and $i=0,\ldots,n$. 
We consider linear inequalities of a code. That is, valid inequalities of the form
\begin{equation}\label{lambdabeta}
\sum_{i=0}^n \lambda_i A_i(\mathbf{u})\geq \beta\quad\text{for all $\mathbf{u}\in\E$},
\end{equation}
where $\lambda_0,\ldots,\lambda_n\geq 0$ and $\beta>0$.
Such a set of inequalities is denoted by $(\lambda_0,\ldots,\lambda_n)\beta$ and leads to a lower bound on $K_q(n,r)$ by the following proposition.
\begin{proposition}\label{linineqbound}
If any $(n,K,q)r$ code satisfies $(\lambda_0,\ldots,\lambda_n)\beta$ then
\begin{equation}
K\geq \frac{\beta q^n}{\sum_{i=0}^n \lambda_i{n\choose i}(q-1)^i}.
\end{equation}
\end{proposition}

\begin{proof}
Summing \ref{lambdabeta} over all $\mathbf{u}\in\E$ we obtain
\begin{eqnarray}
\beta q^n\leq \sum_{\mathbf{u}\in\E}\sum_{i=0}^n \lambda_iA_i(\mathbf{u})&=&\sum_{i=0}^n \lambda_i \sum_{\mathbf{u}\in\E} A_i(\mathbf{u})\\
&=&\sum_{i=0}^n \lambda_i \sum_{\mathbf{v}\in C} |S_i(\mathbf{v})|\nonumber\\
&=&|C|\sum_{i=0}^n \lambda_i {n\choose i}(q-1)^i.\nonumber
\end{eqnarray}
\end{proof}

The basic \notion{sphere covering inequalities}
\begin{equation}
\sum_{i=0}^r A_i(\mathbf{u})\geq 1\quad\text{for all $\mathbf{u}\in\E$}
\end{equation}
give the \notion{sphere covering bound}
\begin{equation}
K_q(n,r)\geq \frac{q^n}{\sum_{i=0}^r {n \choose i}(q-1)^i}.
\end{equation}

Many other valid inequalities have been obtained, in particular in the binary case ($q=2$), by studying the way the elements in $B_s(\mathbf{u})$ can be covered for $s=1,2,3$. In the case $s=1$ this gives the van Wee inequalities \cite{VanWee,VanWeeThesis}:
\begin{equation}
\sum_{i=0}^{r-1} \roundup{\frac{n+1}{r+1}}A_i(\mathbf{u})+A_r(\mathbf{u})+A_{r+1}(\mathbf{u})\geq \roundup{\frac{n+1}{r+1}}
\end{equation}
which improve upon the sphere covering bound whenever $r+1$ does not divide $n+1$.

The case $s=2$ leads to the \emph{pair covering inequalities} found by Johnson \cite{Johnson} and Zhang \cite{Zhang}:
\begin{equation}
\sum_{i=0}^{r-2} m_0A_i(\mathbf{u})+m_1(A_{r-1}(\mathbf{u})+A_r(\mathbf{u}))+A_{r+1}(\mathbf{u})+A_{r+2}(\mathbf{u})\geq m_0,
\end{equation}
where 
\begin{eqnarray}
m_1&=\max_{i\geq 2} \frac{F(n-r+1,r+2)-F(n-iR+1,R+2)}{i-1},\\
m_0&=m_1+F(n-r+1,r+2),\nonumber
\end{eqnarray} 
and $F(m,k)$ is the minimum number of $k$-sets needed to cover all pairs of an $m$-set.
Other inequalities can be found in \cite{ZhangLo}.

Starting from a set of inequalities for a code, new inequalities can be obtained by taking nonnegative linear combinations. Also by summing the inequality $(\lambda_0,\ldots,\lambda_n)\beta$ over $S_i(\mathbf{u})$, we obtain the \emph{induced} inequality
$(\lambda'_0,\ldots,\lambda'_n)\beta'$, where 
\begin{eqnarray}
\lambda'_k&:=&\sum_{j=0}^n \lambda_j \alpha_{i,j}^k\\
\beta'&:=&{n\choose i}(q-1)^i\beta\nonumber,
\end{eqnarray}
and 
\begin{equation}
\alpha_{i,j}^k:=|\{\mathbf{v}\mid d(\zero,\mathbf{v})=i, d(\mathbf{v},\mathbf{u})=j\}|
\end{equation}
when $d(\zero,\mathbf{u})=k$.
The numbers $\alpha_{i,j}^k$ can be expressed as
\begin{equation}
\alpha_{i,j}^k=\begin{cases}
\sum_{\substack{p,t\\t+p=k+i-j}} {k\choose t-p,p}{n-k\choose i-t}(q-1)^{i-t}(q-2)^{t-p}&\text{if $q\geq 3$}\\ \\
\sum_{\substack{t\\2t=k+i-j}} {k\choose t}{n-k\choose i-t}& \text{if $q=2$.}
\end{cases}
\end{equation} 
Note that the bound obtained from an induced inequality is equal to the bound obtained from the original one. Using the fact that the $A_i(\mathbf{u})$ are integers, the inequality $(\lambda_0,\ldots,\lambda_n)\beta$ implies the inequality $(\roundup{\lambda_0},\ldots,\roundup{\lambda_n})\roundup{\beta}$. This way the van Wee inequalities, for example, can be derived from the sphere covering inequalities as follows. Starting from the sphere covering inequalities, we obtain
\begin{equation}
\sum_{i=0}^{r-1}(n+1)A_i(\mathbf{u})+(r+1)(A_{r}(\mathbf{u})+A_{r+1}(\mathbf{u}))\geq n+1\quad \text{for every $\mathbf{u}\in\E$}
\end{equation}
by summing the sphere covering inequalities over $B_1(\mathbf{u})$. Then dividing by $r+1$ and rounding up the coefficients, the van Wee inequalities are obtained.

Using this method, Habsieger and Plagne obtained many new lower bounds in the binary and ternary case, by computer search see \cite{HabsiegerPlagne}.
 
\section{Semidefinite programming bounds}
The bound from Proposition \ref{linineqbound} may be viewed as a linear programming bound as follows. Given $\lambda\in\Real^{n+1}$ and $\beta\in\Real$, define the polyhedron
\begin{equation}
P_{\lambda,\beta}:=\{x\in\Real^{\E}\mid \sum_{i=0}^n \lambda_i x(S_i(\mathbf{u}))\geq \beta\quad \text{for all $\mathbf{u}\in\E$ }\}.
\end{equation}
We have the following proposition.
\begin{proposition}
\begin{equation}
\min \{\one^{\transp}x\mid x\in P_{\lambda,\beta}\}=\frac{\beta q^n}{\sum_{i=0}^n \lambda_i {n\choose i}(q-1)^i}.
\end{equation}
\end{proposition}
\begin{proof}
Observe that for any $x\in P_{\lambda,\beta}$ also 
\begin{equation}
\overline{x}:=\frac{1}{|\aut|}\sum_{\sigma\in\aut} \sigma(x)\in P_{\lambda,\beta}
\end{equation}
and $\overline{x}=c\one$ where $\one^{\transp}c\one=\one^{\transp}x$. Hence
\begin{eqnarray}
\min \{\one^{\transp}x\mid x\in P_{\lambda,\beta}\}&=&\min \{\one^{\transp}c\one\mid c\one\in P_{\lambda,\beta}\}\\
&=& \min \{q^nc\mid \sum_{i=0}^n \lambda_i c|S_i(\zero)|\geq \beta\}\nonumber\\
&=& \min \{q^nc\mid c\geq \frac{\beta}{\sum_{i=0}^n \lambda_i {n\choose i}(q-1)^i}\}\nonumber\\
&=& \frac{\beta q^n}{\sum_{i=0}^n \lambda_i{n\choose i}(q-1)^i}.\nonumber
\end{eqnarray}
\end{proof}

Clearly, replacing $P_{\lambda,\beta}$ by $P_{\lambda,\beta}\cap \{0,1\}^{\E}$ and considering the $0$--$1$ optimization problem, can be expected to give a better lower bound. In fact, when $(\lambda_0,\ldots,\lambda_n)\beta$ corresponds to the sphere covering inequalities, this $0$--$1$ program gives the exact value $K_q(n,r)$\footnote{In general there may be solutions that do not have covering radius $\leq r$, for example when $n=3,r=1$, the code $\{100,010,001\}$ has covering radius $2$ but satisfies the van Wee inequalities.}. This motivates to replace the linear relaxation $P_{\lambda,\beta}$ by a tighter (semidefinite) relaxation using the method of matrix cuts from Chapter \ref{CH:matrixcuts}. We will pursue this idea in the following.

\subsection{The first SDP bound}
In this section we derive a semidefinite programming lower bound on $K_q(n,r)$ with $O(n)$ variables and $O(n)$ constraints. This bound is equal to the value obtained by minimizing $\one^{\transp}x$ over $N_+(P_{\lambda,\beta})$, see Chapter \ref{CH:matrixcuts}. 

To any code $C\subseteq \E$, we associate the symmetric $0$--$1$ matrix $M_C$ defined by:
\begin{equation}
(M_C)_{\mathbf{u},\mathbf{v}}:=\begin{cases}
1&\text{if $\mathbf{u},\mathbf{v}\in C$,}\\
0&\text{otherwise.}
\end{cases}
\end{equation}
Let $C\subseteq E$ be a code. Define the matrix
\begin{equation}
M:=|\aut|^{-1}\sum_{\sigma\in\aut} M_{\sigma C}. 
\end{equation}
By construction, the matrix $M$ is invariant under permutations of the rows and columns by any $\sigma\in\aut$. Hence $M$ is an element of the Bose--Mesner algebra of the Hamming scheme and we write 
\begin{equation}
M=\sum_{i=0}^n x_i A_i,
\end{equation}
where $A_i$ is the $i$-th basis matrix of the Bose--Mesner algebra and $x_0,\ldots,x_n\in \Real$.
\begin{proposition}\label{absconstraints}
The matrix $M$ satisfies the following.
\begin{eqnarray}
\mathrm{(i)}&&\trace M=|C|,\\
\mathrm{(ii)}&&M\geq 0\text{\ and\ }R(M)\psd,\nonumber\\
\mathrm{(iii)}&&\text{If $C$ satisfies $(\lambda_0,\ldots,\lambda_n)\beta$, then}\nonumber\\
&& M_\mathbf{u}\in M_{\mathbf{u},\mathbf{u}} P_{\lambda,\beta}\quad\text{and}\quad \diag(M)-M_\mathbf{u}\in (1-M_{\mathbf{u},\mathbf{u}})P_{\lambda,\beta}\nonumber\\
&&\text{for every $\mathbf{u}\in\E$.}\nonumber
\end{eqnarray}
\end{proposition} 
\begin{proof}
Since $M$ is a convex combination of the $M_{\sigma C}$, $\sigma\in\aut$, it suffices to observe that the contraints hold for each $M_{\sigma C}$. Clearly, $\trace M_{\sigma C}=|C|$ and $M_{\sigma C}\geq 0$. As $R(M_{\sigma C})={1\choose \chi^{\sigma C}} {1\choose \chi^{\sigma C}}^{\transp}$, $R(M_{\sigma C})$ is positive semidefinite. Finally, for any $\mathbf{u}\in\E$ 
\begin{equation}
(M_{\sigma C})_\mathbf{u}=(M_{\sigma C})_{\mathbf{u},\mathbf{u}}\chi^{\sigma C}
\end{equation}
and
\begin{equation}
\diag(M_{\sigma C})-(M_{\sigma C})_\mathbf{u}=(1-(M_{\sigma C})_{\mathbf{u},\mathbf{u}})\chi^{\sigma C}
\end{equation}
and hence (iii) follows from the fact that $\sigma C$ satisfies $(\lambda_0,\ldots,\lambda_n)\beta$ for every $\sigma\in\aut$.
\end{proof}
Below, we will make these constraints more explicit by expressing them in terms of the variables $x_i$.

\begin{proposition}
$R(M)\psd$ is equivalent to 
\begin{eqnarray}
&&\sum_{i=0}^n x_iP_j(i)\geq 0\quad\text{for every }j=0,\ldots,n\\
\text{and}\nonumber\\
&&\begin{pmatrix}
q^n&q^nx_0\\q^nx_0&\sum_{i=0}^n x_i{n\choose i}(q-1)^i
\end{pmatrix}\psd.\nonumber
\end{eqnarray}
\end{proposition}
\begin{proof}
Since $\trace M=q^nx_0$ and $\sumentries{M}=q^n\sum_{i=0}^n x_i{n\choose i}(q-1)^i$, it follows from Proposition \ref{R(A)PSD} that $R(M)\psd$ if and only if $M\psd$ and
\begin{equation*}
\begin{pmatrix}q^n&q^nx_0\\q^nx_0& \sum_{i=0}^n x_i{n\choose i}(q-1)^i\end{pmatrix}
\end{equation*} is positive semidefinite. By Proposition \ref{Krawtchouk} it follows that $M=\sum_{i=0}^n x_iA_i$ is positive semidefinite if and only if $\sum_{i=0}^n x_iP_j(i)\geq 0$ for every $j=0,\ldots,n$.
\end{proof}

\begin{proposition}\label{rowinequality1}
Let $x=\sum_{i=0}^n x_i\chi^{S_i(\zero)}\in\Real^{\E}$. Then the following are equivalent:
\begin{eqnarray}
\mathrm{(i)}&&\sum_{i=0}^n \lambda_i x(S_i(\mathbf{u}))\geq \beta\quad\text{for every $\mathbf{u}\in\E$},\\
\mathrm{(ii)}&&\sum_{j=0}^n x_j\cdot \sum_{i=0}^n \lambda_i\alpha_{i,j}^k\geq \beta\quad\text{for every $k=0,\ldots,n$}.\nonumber
\end{eqnarray} 
\end{proposition}
\begin{proof}
If $d(\mathbf{u},\zero)=k$ then 
\begin{eqnarray}
\sum_{i=0}^n \lambda_i x(S_i(\mathbf{u}))&=&
\sum_{i=0}^n \lambda_i \sum_{j=0}^n\sum_{\substack{\mathbf{v}\in\E\\d(\zero,\mathbf{v})=j\\d(\mathbf{u},\mathbf{v})=i}} x_j\\
&=&\sum_{i=0}^n \lambda_i \sum_{j=0}^n \alpha_{i,j}^k x_j\nonumber\\
&=&\sum_{j=0}^n x_j\sum_{i=0}^n \lambda_i\alpha_{i,j}^k.\nonumber
\end{eqnarray}
\end{proof}

\begin{proposition}

The following are equivalent
\begin{eqnarray}
\mathrm{(i)}&&M_\mathbf{u}\in M_{\mathbf{u},\mathbf{u}} P_{\lambda,\beta}\quad\text{and}\\
&&\diag(M)-M_\mathbf{u}\in (1-M_{\mathbf{u},\mathbf{u}})P_{\lambda,\beta}\nonumber\\
&&\text{for every $\mathbf{u}\in\E$},\nonumber\\
\mathrm{(ii)}&&\sum_{j=0}^n x_j\cdot \sum_{i=0}^n \lambda_i\alpha_{i,j}^k\geq x_0\beta\\
&&\sum_{j=0}^n (x_0-x_j)\cdot \sum_{i=0}^n \lambda_i\alpha_{i,j}^k\geq (1-x_0)\beta\nonumber\\
&&\text{for every $k=0,\ldots,n$.}\nonumber
\end{eqnarray}
\end{proposition}
\begin{proof}
Directly from Proposition \ref{rowinequality1}
\end{proof}

Collecting all the propositions, we obtain the following theorem.
\begin{theorem}
If every code $C\subseteq \E$ with covering radius $r$ satisfies $(\lambda_0,\ldots,\lambda_n)\beta$, we have
\begin{equation}
K_q(n,r)\geq \min_x q^n x_0,
\end{equation}
where the minimum ranges over all $x=(x_0,x_1,\ldots,x_n)^{\transp}\in\Real^{n+1}$ satisfying 
\begin{eqnarray}
\mathrm{(i)} &&x_k\geq 0,\\
\mathrm{(ii)}&&\sum_{i=0}^n x_iP_k(i)\geq 0,\nonumber\\
\mathrm{(iii)}&&\sum_{i=0}^n x_i\cdot \sum_{j=0}^n \lambda_j\alpha_{i,j}^k\geq \beta x_0,\nonumber\\
\mathrm{(iv)}&&\sum_{i=0}^n (x_0-x_i)\cdot\sum_{j=0}^n \lambda_j\alpha_{i,j}^k\geq \beta (1-x_0),\nonumber\\
\mathrm{(v)}&&\begin{pmatrix}q^n&q^nx_0\\q^nx_0& \sum_{i=0}^n x_i{n\choose i}(q-1)^i\end{pmatrix}\psd\nonumber
\end{eqnarray}
for all $k=0,\ldots,n$.
\end{theorem}
\begin{proof}
\end{proof}
Observe that if we relax the semidefinite program by only requiring $M$ to be positive semidefinite instead of $R(M)$ (that is: delete condition (v)), we obtain for a linear program in $O(n)$ variables and inequalities that is a lower bound on $K_q(n,r)$.
 
\subsection{The second SDP bound}
In this section we describe a stronger semidefinite programming relaxation that uses more of the symmetry of the Hamming space, but requires $O(n^3)$ variables in the binary case and $O(n^4)$ variables in the nonbinary case. In this section we will focus on the binary case. The nonbinary case is very similar, although more complicated and it will be adressed in the next section.   

Restricting ourselves to the binary case, we have $\E=\{0,1\}^n$, the $n$-dimensional Hamming cube. 
Let $C\subseteq \E$ be any code and define the matrices $M'$ and $M''$ by:
\begin{eqnarray}
M'&:=&|\autbin|^{-1}\sum_{\substack{\sigma\in\autbin\\ \zero\in \sigma C}} M_{\sigma C}\\
M''&:=&|\autbin|^{-1}\sum_{\substack{\sigma\in\autbin\\ \zero\not\in \sigma C}} M_{\sigma C}\nonumber.
\end{eqnarray}
By construction, the matrices $M'$ and $M''$ are invariant under permutations $\sigma\in\stabbin$ of the rows and columns, that fix the element $\zero$. Hence $M'$ and $M''$ are elements of the algebra $\mathcal{A}_{2,n}$. Write 
\begin{equation}
M'=\sum_{(i,j,t)} x_{i,j}^{t} M_{i,j}^{t},
\end{equation}
where the matrices $M_{i,j}^t$ are the zero--one basis matrices of $\mathcal{A}_{2,n}$. The matrix $M''$ can be expressed in terms of the coefficients $x_{i,j}^{t}$ as follows.
\begin{proposition}
The matrix $M''$ satisfies
\begin{equation}
M''=\sum_{(i,j,t)}(x_{i+j-2t,0}^{0,0}-x_{i,j}^{t}) M_{i,j}^{t}.
\end{equation}
\end{proposition}
\begin{proof}
The matrix
\begin{equation}
M:=M'+M''=|\autbin|^{-1}\sum_{\sigma\in\autbin} M_{\sigma C}
\end{equation}
is invariant under permutation of the rows and columns by any permutation $\sigma\in\autbin$, and hence is an element of the Bose--Mesner algebra, say
\begin{equation}
M=\sum_{k} y_k A_k.
\end{equation}
Observe that for any $\mathbf{u}\in\E$ with $d(\mathbf{u},\zero)=k$, we have 
\begin{equation}
y_k=(M)_{\mathbf{u},\zero}=(M')_{\mathbf{u},\zero}=x_{k,0}^{0},
\end{equation}
since $(M'')_{\mathbf{u},\zero}=0$. Hence we have
\begin{eqnarray}
M''&=&M-M'\\
&=&\sum_{k} x_{k,0}^{0} A_k-\sum_{(i,j,t)} x_{i,j}^{t} M_{i,j}^{t}\nonumber\\ &=&\sum_{k}\sum_{i+j-2t=k} x_{k,0}^{0} M_{i,j}^t -\sum_{(i,j,t)} x_{i,j}^{t} M_{i,j}^t\nonumber\\
&=&\sum_{(i,j,t)} (x_{i+j-2t,0}^{0}-x_{i,j}^{t}) M_{i,j}^{t},\nonumber
\end{eqnarray}
which proves the proposition.
\end{proof}

\begin{proposition}\label{psd}
The matrices 
\begin{equation}
M'\quad\text{and}\quad \begin{pmatrix}1-x_{0,0}^0&(\diag (M''))^{\transp}\\ \diag(M'')& M''\end{pmatrix}
\end{equation}
are positive semidefinite.
\end{proposition}
\begin{proof}
Clearly, $R(M_{\sigma C})={1\choose \chi^{\sigma C}}{1\choose \chi^{\sigma C}}^{\transp}$ is positive semidefinite for each $\sigma\in\autbin$. Hence $R((x_{0,0}^0)^{-1}M')$ and $R((1-x_{0,0}^0)^{-1}M'')$ are positive semidefinite as they are convex combinations of the $R(M_{\sigma C})$. This implies the statement in the proposition. 
\end{proof}

Using the block diagonalisation of $\mathcal{A}_{2,n}$, Proposition \ref{psd} is equivalent to
the following matrices being positive semidefinite
\begin{eqnarray}\label{constraint3}
\left(\sum_{t=0}^n \beta_{i,j,k}^t x_{i,j}^t\right)_{i,j=k}^{n-k},\ \left(\sum_{t=0}^n \beta_{i,j,k}^t (x_{i+j-2t,0}^0-x_{i,j}^t)\right)_{i,j=k}^{n-k}\nonumber
\end{eqnarray}
for each $k=1,\ldots,\rounddown{\frac{n}{2}}$,
\begin{equation}
\left(\sum_{t=0}^n \beta_{i,j,0}^t x_{i,j}^t\right)_{i,j=0}^n,\ \begin{pmatrix}
1-x_{0,0}^0 & x^{\transp}\\x&L\end{pmatrix}\psd\nonumber
\end{equation}
where
\begin{eqnarray}
L:=\left(\sum_{t=0}^n \beta_{i,j,0}^t (x_{i+j-2t,0}^0-x_{i,j}^t)\right)_{i,j=0}^n,\nonumber\\
x_i:=(x_{0,0}^0-x_{i,0}^0){n\choose i},\text{ for $i=0,\ldots,n$}.\nonumber
\end{eqnarray}
 
\begin{proposition}
The coefficients $x_{i,j}^t$ satisfy the following:
\begin{equation}
2^nx_{0,0}^0=|C|,
\end{equation}
and for any $i,j,t$
\begin{eqnarray}\label{constraint1}
\mathrm{(i)}&&0\leq x_{i,j}^t\leq x_{i,i}^i,\\
\mathrm{(ii)}&&x_{i,0}^0+x_{i+j-2t,0}^0-x_{0,0}^0\leq x_{i,j}^t\leq x_{i+j-2t,0}^0,\nonumber\\
\mathrm{(iii)}&&x_{i,j}^t=x_{i',j'}^{t'}\quad\text{if $(i,j,i+j-2t)$ is a permutation}\nonumber\\
&&\text{of $(i',j',i'+j'-2t')$.}\nonumber
\end{eqnarray}
\end{proposition}
\begin{proof}
Since for any $\mathbf{u}\in\E$
\begin{equation}
|\{\sigma\in\autbin\mid \sigma \mathbf{u}=\zero\}|=|\stabbin|,
\end{equation}
we obtain 
\begin{equation}
x_{0,0}^0=\frac{|\{\sigma\in\autbin\mid \zero\in\sigma C\}|}{|\autbin|}=|C|\frac{|\stabbin|}{|\autbin|}=2^{-n}|C|.
\end{equation}
Inequalities (i) and (ii) follow from the fact that $(M')_{\mathbf{u},\mathbf{u}}\geq (M')_{\mathbf{u},\mathbf{v}}$ and $(M'')_{\mathbf{u},\mathbf{u}}\geq (M'')_{\mathbf{u},\mathbf{v}}$ for any $\mathbf{u},\mathbf{v}\in\E$ respectively. The truth of (iii) can be seen as follows. Let $\mathbf{u},\mathbf{v}\in\E$ be such that $\dd(\mathbf{u},\mathbf{v})=(i,j,t)$ and let $(i',j',t')$ be such that $(i,j,i+j-2t)$ is a permutation of $(i',j',i'+j'-2t')$. It can be seen that in that case there is a $\sigma\in\autbin$ such that $\sigma \{\zero,\mathbf{u},\mathbf{v}\}=\{\zero, \mathbf{u}',\mathbf{v}'\}$ with $\dd(\mathbf{u}',\mathbf{v}')=(i',j',t')$. Hence 
\begin{equation}
x_{i,j}^t=(M')_{\mathbf{u},\mathbf{v}}=(M')_{\mathbf{u}',\mathbf{v}'}=x_{i',j'}^{t'}.
\end{equation}
\end{proof}

Given two words $\mathbf{u},\mathbf{v}\in\E$ with $\dd(\mathbf{u},\mathbf{v})=(i,j,t)$, we denote by $\alpha_{(i,j',t'),d}^{(i,j,t)}$ the number of words $\mathbf{w}\in\E$ with $\dd(\mathbf{u},\mathbf{w})=(i,j',t')$ and $d(\mathbf{v},\mathbf{w})=d$. This number is well-defined, and indeed we have the following proposition.
\begin{proposition}
The numbers $\alpha_{(i,j',t'),d}^{(i,j,t)}$ are given by
\begin{equation}
\alpha_{(i,j',t'),d}^{(i,j,t)}=\sum_{a_{00},a_{01},a_{10},a_{11}} {i-t\choose a_{10}}{j-t\choose a_{01}}{t\choose a_{11}}{n+t-i-j\choose a_{00}},
\end{equation}
where the indices $a_{00},a_{01},a_{10}$ and $a_{11}$ range over the nonnegative integers that satisfy
\begin{eqnarray}
j'&=&a_{00}+a_{01}+a_{10}+a_{11}\\
t'&=&a_{10}+a_{11}\nonumber\\
d-j&=&a_{00}+a_{10}-a_{01}-a_{11}\nonumber.
\end{eqnarray}
\end{proposition} 
\begin{proof}
Partition the support of the words $\mathbf{w}$ into four sets $A_{00},A_{01},A_{10}$ and $A_{11}$ as follows:
\begin{eqnarray}
A_{00}&:=&\{k\in S(\mathbf{w})\mid \mathbf{u}_k=0,\mathbf{v}_k=0\}\\
A_{01}&:=&\{k\in S(\mathbf{w})\mid \mathbf{u}_k=0,\mathbf{v}_k\not=0\}\nonumber\\
A_{10}&:=&\{k\in S(\mathbf{w})\mid \mathbf{u}_k\not=0, \mathbf{v}_k=0\}\nonumber\\
A_{11}&:=&\{k\in S(\mathbf{w})\mid \mathbf{u}_k\not=0,\mathbf{v}_k\not=0\}\nonumber.
\end{eqnarray}
If we denote the sizes of these four sets by $a_{00},a_{01},a_{10}$ and $a_{11}$ respectively, we obtain the claimed result by summing over all possible sets $A_{00},A_{01},A_{10}$ and $A_{11}$.
\end{proof}

\begin{proposition}\label{rowinequalitybin}
Let $\mathbf{u}\in\E$ be a word with $d(\mathbf{u},\zero)=i$ and let $x\in\Real^{\E}$ be such that $x_\mathbf{v}$ only depends on $\dd(\mathbf{u},\mathbf{v})$, say $x_\mathbf{v}=x_{i,j}^{t}$, when $\dd(\mathbf{u},\mathbf{v})=(i,j,t)$. Then 
\begin{equation}
\sum_{d=0}^n\lambda_d x(S_d(\mathbf{v}))\geq \beta\quad\text{for all $\mathbf{v}\in\E$}
\end{equation}
is equivalent to
\begin{equation}
\sum_{j',t'} x_{i,j'}^{t'}\cdot \sum_{d=0}^n \lambda_d \alpha_{(i,j',t'),d}^{(i,j,t)}\geq \beta\quad\text{for all $j,t$}.
\end{equation} 
\end{proposition}
\begin{proof}
Let $\mathbf{v}\in\E$ and let $\dd(\mathbf{u},\mathbf{v})=(i,j,t)$. Then we have the following equalities.
\begin{eqnarray}
\sum_{d=0}^n \lambda_d x(S_d(\mathbf{v}))&=&\sum_{d=0}^n\lambda_d\sum_{\substack{\mathbf{w}\in\E\\d(\mathbf{v},\mathbf{w})=d}} x_{\mathbf{w}}\\
&=&\sum_{d=0}^n\lambda_d\sum_{j',t'} \sum_{\substack{\mathbf{w}\in\E\\d(\mathbf{v},\mathbf{w})=d)\\ \dd(\mathbf{u},\mathbf{w})=(i,j',t')}}x_{\mathbf{w}}\nonumber\\
&=&\sum_{d=0}^n \lambda_d\sum_{j',t'} \alpha_{(i,j',t'),d}^{(i,j,t)} x_{i,j'}^{t'}\nonumber\\
&=&\sum_{j',t'} x_{i,j'}^{t'} \cdot \sum_{d=0}^n \lambda_d \alpha_{(i,j',t'),d}^{(i,j,t)}.\nonumber
\end{eqnarray}
\end{proof}

\begin{proposition}
If the code $C$ satisfies the set of inequalities $(\lambda_0,\ldots,\lambda_n)\beta$, then the variables $x_{i,j}^{t}$ satisfy the following set of inequalities.
For every tuple $(i,j,t)$
\begin{eqnarray}\label{constraint2}
\sum_{j',t'} x_{i,j'}^{t'}\cdot \lambda_{j',t'}^{i,j,t}&\geq& x_{i,0}^{0}\beta\\
\sum_{j',t'} (x_{j',0}^{0}-x_{i,j'}^{t'})\cdot \lambda_{j',t'}^{i,j,t}&\geq&(x_{0,0}^{0}-x_{i,0}^{0})\beta\nonumber\\
\sum_{j',t'} (x_{i+j-2t,0}^{0}-x_{i,j}^{t})\cdot \lambda_{j',t'}^{i,j,t}&\geq&(x_{0,0}^{0}-x_{i,0}^{0})\beta\nonumber\\
\sum_{j',t'} (x_{0,0}^{0}-x_{j',0}^{0}-x_{i+j'-2t',0}^{0}+x_{i,j'}^{t'})\cdot \lambda_{j',t'}^{i,j,t}&\geq&(1-2x_{0,0}^{0}+x_{i,0}^{0})\beta,\nonumber
\end{eqnarray}
where we use the shorthand notation
\begin{equation}
\lambda_{j',t'}^{i,j,t}:=\sum_{d=0}^n \lambda_d \alpha_{(i,j',t'),d}^{(i,j,t)}.
\end{equation}
\end{proposition}
\begin{proof}
For any $\sigma\in\autbin$, the matrix $M:=M_{\sigma C}$ satisfies
\begin{eqnarray}\label{mcbin}
M_\mathbf{u}&\in & M_{\mathbf{u},\mathbf{u}}P_{\lambda,\beta},\\
\diag(M)-M_\mathbf{u}&\in& (1-M_{\mathbf{u},\mathbf{u}})P_{\lambda,\beta}\nonumber\\
\text{for every $\mathbf{u}\in\E$.}\nonumber
\end{eqnarray}
This implies that also the matrices $\frac{1}{x_{0,0}^{0}}M'$ and $\frac{1}{1-x_{0,0}^{0}}M''$ satisfy (\ref{mcbin}) as they are convex combinations of the matrices $M_{\sigma C}$. Now using Proposition \ref{rowinequalitybin} gives a proof of the claim.
\end{proof}

This leads to the following semidefinite programming bound on $K_2(n,r)$.
\begin{theorem}
If any code $C\subseteq \E$ with covering radius $r$ satisfies $(\lambda_0,\ldots,\lambda_n)\beta$, we have
\begin{equation}
K_2(n,r)\geq \min_x 2^nx_{0,0}^0,
\end{equation}
where the minimum ranges over all $x=(x_{i,j}^t)$ satisfying (\ref{constraint3}), (\ref{constraint1}) and (\ref{constraint2}).
\end{theorem}
\begin{proof}
\end{proof}

\section{Nonbinary case}
In this section we consider the nonbinary case, that is $q\geq 3$. The nonbinary case is very similar to the binary case described in the previous section and we will skip some of the details in the proofs. 

Again define the matrices $M'$ and $M''$ by 
\begin{eqnarray}
M'&:=&|\autbin|^{-1}\sum_{\substack{\sigma\in\autbin\\ \zero\in \sigma C}} M_{\sigma C}\\
M''&:=&|\autbin|^{-1}\sum_{\substack{\sigma\in\autbin\\ \zero\not\in \sigma C}} M_{\sigma C}\nonumber.
\end{eqnarray}
The matrices $M'$ and $M''$ are invariant under permutations of the rows and columns by permutations $\sigma\in\stab$. Hence $M'$ and $M''$ are elements of the algebra $\mathcal{A}_{q,n}$. We write
\begin{equation}
M'=\sum_{(i,j,t,p)}x_{i,j}^{t,p}M_{i,j}^{t,p} 
\end{equation}
where the $M_{i,j}^{t,p}$ are the $0$--$1$ basis matrices of the algebra $\mathcal{A}_{q,n}$.
The matrix $M''$ can be expressed in terms of the coefficients $x_{i,j}^{t,p}$ as follows.
\begin{proposition}
The matrix $M''$ is given by
\begin{equation}
M''=\sum_{(i,j,t,p)}(x_{i+j-t-p,0}^{0,0}-x_{i,j}^{t,p}) M_{i,j}^{t,p}.
\end{equation}
\end{proposition}
\begin{proof}
The matrix
\begin{equation}
M:=M'+M''=|\aut|^{-1}\sum_{\sigma\in\aut} M_{\sigma C}
\end{equation}
is invariant under permutation of the rows and columns by permutations $\sigma\in\aut$, and hence is an element of the Bose--Mesner algebra, say
\begin{equation}
M=\sum_{k} y_k A_k.
\end{equation}
Note that for any $\mathbf{u}\in\E$ with $|S(\mathbf{u})|=k$, we have 
\begin{equation}
y_k=(M)_{\mathbf{u},\zero}=(M')_{\mathbf{u},\zero}=x_{k,0}^{0,0},
\end{equation}
since $(M'')_{\mathbf{u},\zero}=0$. Hence we have
\begin{eqnarray}
M''&=&M-M'\\
&=&\sum_{k} x_{k,0}^{0,0} A_k-\sum_{(i,j,t,p)} x_{i,j}^{t,p} M_{i,j}^{t,p}\nonumber\\ &=&\sum_{k}\sum_{i+j-t-p=k} (x_{k,0}^{0,0}-x_{i,j}^{t,p}) M_{i,j}^{t,p}\nonumber\\
&=&\sum_{(i,j,t,p)} (x_{i+j-t-p,0}^{0,0}-x_{i,j}^{t,p}) M_{i,j}^{t,p},\nonumber
\end{eqnarray}
which proves the proposition.
\end{proof}
\begin{proposition}
The matrices 
\begin{equation}\label{psd2}
M'\quad\text{and}\quad \begin{pmatrix}1-x_{0,0}^{0,0}&(\diag(M''))^{\transp}\\ \diag(M'')& M''\end{pmatrix}
\end{equation}
are positive semidefinite.
\end{proposition}

Using the block diagonalisation of $\mathcal{A}_{q,n}$, the positive semidefiniteness of $R$ and $R'$ is equivalent to:
\begin{gather}\label{constr3}
\text{for all $a,k$ with $0\leq a\leq k\leq n+a-k$, $k\not=0$ the matrices}\\ 
\nonumber\\
\left(\sum_{t,p}\alpha(i,j,t,p,a,k)x_{i,j}^{t,p}\right)_{i,j=k}^{n+a-k}\notag\\
\text{and}\notag\\ \left(\sum_{t,p}\alpha(i,j,t,p,a,k)(x_{i+j-t-p,0}^{0,0}-x_{i,j}^{t,p})\right)_{i,j=k}^{n+a-k}\notag\\
\text{are positive semidefinite, and}\notag\\
\left(\sum_{t,p}\alpha(i,j,t,p,0,0)x_{i,j}^{t,p}\right)_{i,j=0}^{n}\notag\\
\text{and}\notag\\
\begin{pmatrix}1-x_{0,0}^{0,0}&x^{\transp}\\x&L\end{pmatrix}\notag\\
\text{are positive semidefinite, where}\notag\\
L:=\left(\sum_{t,p}\alpha(i,j,t,p,0,0)(x_{i+j-t-p,0}^{0,0}-x_{i,j}^{t,p})\right)_{i,j=0}^{n},\notag\\
x_i:=(x_{0,0}^{0,0}-x_{i,i}^{i,i}){n\choose i}(q-1)^i\text{ for $i=0,\ldots,n$}.
\end{gather}

\begin{proposition}
The coefficients $x_{i,j}^{t,p}$ satisfy the following.
\begin{equation}
q^nx_{0,0}^0=|C|,
\end{equation}

and for any $i,j,t,p$
\begin{eqnarray}\label{constr1}
\text{(i)}&&0\leq x_{i,j}^{t,p}\leq x_{i,i}^{i,i},\\
\text{(ii)}&&x_{i,0}^{0,0}+x_{i+j-t-p,0}^{0,0}-x_{0,0}^{0,0}\leq x_{i,j}^{t,p}\leq x_{i+j-t-p,0}^{0,0},\nonumber\\
\text{(iii)}&&x_{i,j}^{t,p}=x_{i',j'}^{t',p'}\quad\text{if $(i,j,i+j-t-p)$ is a permutation of}\nonumber\\
&&\text{$(i',j',i'+j'-t'-p')$ and $t-p$=$t'-p'$.}\nonumber
\end{eqnarray}
\end{proposition}

\begin{proposition}
Let $\mathbf{u},\mathbf{v}\in\E$ be words with $d(\mathbf{u},\mathbf{v})=(i,j,t,p)$ and let $(i,j',t',p')$ and $d$ be given. Then the number $\alpha_{(i,j',t',p'),d}^{(i,j,t,p)}$ of words $\mathbf{w}\in\E$ with $d(\mathbf{u},\mathbf{w})=(i,j',t,'p')$ and $d(\mathbf{v},\mathbf{w})=d$ is given by
\begin{multline}
\alpha_{(i,j',t',p'),d}^{(i,j,t,p)}=\sum_{\substack{a_1,a_2\\b_1,b_2\\c_1,c_2\\d_1,d_2,d_3\\e}} {i-t\choose a_1,a_2}{j-t\choose b_1,b_2}{p\choose c_1,c_2}{t-p\choose d_1,d_2,d_3}\\
\cdot {n+t-i-j\choose e}(q-1)^e(q-2)^{a_2+b_2+c_2}(q-3)^{d_3}, 
\end{multline}
where the indices $a_1,a_2,b_1,b_2,c_1,c_2,d_1,d_2,d_3$ and $e$ range over the nonnegative integers that satisfy
\begin{eqnarray}
j'&=&a_1+a_2+b_1+b_2+c_1+c_2+d_1+d_2+d_3+e\\
t'&=&a_1+a_2+c_1+c_2+d_1+d_2+d_3\nonumber\\
p'&=&a_1+c_1+d_1\nonumber\\
d &=&a_1+a_2+e+j-b_1-c_1-d_2.\nonumber
\end{eqnarray}
\end{proposition}
Note that in the case $q=3$ we adopt the convention that $0^0=1$. 

\begin{proof}
Partition the support of the word $w$ into sets $A_1$, $A_2$, $B_1$, $B_2$, $C_1$, $C_2$, $D_1$, $D_2$, $D_3$ and $E$ as follows
\begin{eqnarray}
A_1&:=&\{k\in S(\mathbf{w})\mid \mathbf{u}_k\not=0, \mathbf{v}_k=0, \mathbf{w}_k=\mathbf{u}_k\}\\
A_2&:=&\{k\in S(\mathbf{w})\mid \mathbf{u}_k\not=0, \mathbf{v}_k=0, \mathbf{w}_k\not=\mathbf{u}_k\}\nonumber\\
B_1&:=&\{k\in S(\mathbf{w})\mid \mathbf{u}_k=0, \mathbf{v}_k\not=0, \mathbf{w}_k=\mathbf{v}_k\}\nonumber\\
B_2&:=&\{k\in S(\mathbf{w})\mid \mathbf{u}_k=0, \mathbf{v}_k\not=0, w_k\not=\mathbf{v}_k\}\nonumber\\
C_1&:=&\{k\in S(\mathbf{w})\mid \mathbf{u}_k\not=0, \mathbf{v}_k=\mathbf{u}_k, w_k=\mathbf{u}_k\}\nonumber\\
C_2&:=&\{k\in S(\mathbf{w})\mid \mathbf{u}_k\not=0, \mathbf{v}_k=\mathbf{u}_k, w_k\not=\mathbf{u}_k\}\nonumber\\
D_1&:=&\{k\in S(\mathbf{w})\mid \mathbf{u}_k\not=0, \mathbf{v}_k\not=0,\mathbf{u}_k, w_k=\mathbf{u}_k\}\nonumber\\
D_2&:=&\{k\in S(\mathbf{w})\mid \mathbf{u}_k\not=0, \mathbf{v}_k\not=0,\mathbf{u}_k, w_k=\mathbf{v}_k\}\nonumber\\
D_3&:=&\{k\in S(\mathbf{w})\mid \mathbf{u}_k\not=0, \mathbf{v}_k\not=0,\mathbf{u}_k, w_k\not=\mathbf{u}_k,\mathbf{v}_k\}\nonumber\\
E&:=&\{k\in S(\mathbf{w})\mid \mathbf{u}_k=0, \mathbf{v}_k=0\}.\nonumber
\end{eqnarray}
If we denote the sizes by $a_1,a_2,b_1,b_2,c_1,c_2,d_1,d_2,d_3$ and $e$ respectively, we obtain the proposition by summing over all possible sets $A_1$, $A_2$, $B_1$, $B_2$, $C_1$, $C_2$, $D_1$, $D_2$, $D_3$ and $E$.
\end{proof}

\begin{proposition}\label{rowinequality}
Let $\mathbf{u}\in\E$ be a word with $|S(\mathbf{u})|=i$ and let $x\in\Real^{\E}$ be such that $x_{\mathbf{v}}$ only depends on $d(\mathbf{u},\mathbf{v})$, say $x_{\mathbf{v}}=x_{i,j}^{t,p}$, when $d(\mathbf{u},\mathbf{v})=(i,j,t,p)$. Then 
\begin{equation}
\sum_{d=0}^n\lambda_d x(S_d(\mathbf{v}))\geq \beta\quad\text{for all $\mathbf{v}\in\E$}
\end{equation}
is equivalent to
\begin{equation}
\sum_{j',t',p'} x_{i,j'}^{t',p'}\cdot \sum_{d=0}^n \lambda_d \alpha_{(i,j',t',p'),d}^{(i,j,t,p)}\geq \beta\quad\text{for every $j,t,p$}.
\end{equation} 
\end{proposition}
\begin{proof}
Let $\mathbf{v}\in\E$ and let $d(\mathbf{u},\mathbf{v})=(i,j,t,p)$. Then we have the following equality.
\begin{eqnarray}
\sum_{d=0}^n \lambda_d x(S_d(\mathbf{v}))&=&\sum_{d=0}^n\sum_{\substack{\mathbf{w}\in\E\\d(\mathbf{v},\mathbf{w})=d}} x_{\mathbf{w}}\\
&=&\sum_{d=0}^n\lambda_d\sum_{j',t',p'} x_{i,j'}^{t',p'} \alpha_{(i,j',t',p'),d}^{(i,j,t,p)}\\
&=&\sum_{j',t',p'} x_{i,j'}^{t',p'} \cdot \sum_{d=0}^n \lambda_d \alpha_{(i,j',t',p'),d}^{(i,j,t,p)}.
\end{eqnarray}
\end{proof}

\begin{proposition}
If the code $C$ satisfies the set of inequalities $(\lambda,\beta)$, then the variables $x_{i,j}^{t,p}$ satisfy the following set of inequalities.
For every tuple $(i,j,t,p)$
\begin{eqnarray}\label{constr2}
\sum_{j',t',p'} x_{i,j'}^{t',p'}\cdot \lambda_{(i,j',t',p')}^{(i,j,t,p)}&\geq& x_{i,0}^{0,0}\beta\\
\sum_{j',t',p'} (x_{j',0}^{0,0}-x_{i,j'}^{t',p'})\cdot \lambda_{(i,j',t',p')}^{(i,j,t,p)}&\geq&(x_{0,0}^{0,0}-x_{i,0}^{0,0})\beta\nonumber\\
\sum_{j',t',p'} (x_{i+j-t-p,0}^{0,0}-x_{i,j}^{t,p})\cdot \lambda_{(i,j',t',p')}^{(i,j,t,p)}&\geq&(x_{0,0}^{0,0}-x_{i,0}^{0,0})\beta\nonumber\\
\sum_{j',t',p'} (x_{0,0}^{0,0}-x_{j',0}^{0,0}-x_{i+j'-t'-p',0}^{0,0}+x_{i,j'}^{t',p'})\cdot \lambda_{(i,j',t',p')}^{(i,j,t,p)}&\geq&(1-2x_{0,0}^{0,0}+x_{i,0}^{0,0})\beta,\nonumber
\end{eqnarray}
where we have used the shorthand notation
\begin{equation}
\lambda_{(i,j',t',p')}^{(i,j,t,p)}:=\sum_{d=0}^n \lambda_d \alpha_{(i,j',t',p'),d}^{(i,j,t,p)}.
\end{equation}
\end{proposition}
\begin{proof}
For any $\sigma\in\aut$, the matrix $M:=M_{\sigma C}$ satisfies
\begin{equation}\label{matcut}
M_\mathbf{u}\in M_{\mathbf{u},\mathbf{u}}P_{\lambda,\beta}\quad\text{and}\quad\diag(M)-M_\mathbf{u}\in (1-M_{\mathbf{u},\mathbf{u}})P_{\lambda,\beta}\quad \text{for every $\mathbf{u}\in\E$}.
\end{equation}
This implies that also the matrices $\frac{1}{x_{0,0}^{0,0}M'}$ and $\frac{1}{1-x_{0,0}^{0,0}}$ satisfy (\ref{matcut}) as they are convex combinations of the matrices $M_{\sigma C}$. Now using Proposition \ref{rowinequality} gives a proof of the claim.
\end{proof}

\begin{theorem}
If any code $C\subseteq \E$ with covering radius $r$ satisfies $(\lambda_0,\ldots,\lambda_n)\beta$, we have
\begin{equation}
K_q(n,r)\geq \min_x q^nx_{0,0}^{0,0},
\end{equation}
where the minimum ranges over all $x=(x_{i,j}^{t,p})$ satisfying (\ref{constr3}), (\ref{constr1}) and (\ref{constr2}).
\end{theorem}
\begin{proof}
\end{proof}

\section{Computational results}
Using the sphere covering inequalities, we obtained a number of explicit new upper bounds in the case $q=4$ and $q=5$. The results\footnote{In the instance $R=1$, $n=11$ we were unable to solve the second SDP. The given number is the bound obtained from the first SDP.} are shown in table \ref{covtable1} and \ref{covtable2} below. The upper bounds and previous lower bounds are taken from the website of G. K\'eri (\cite{keri}), who maintains an updated table of upper and lower bounds on covering codes. In the binary and ternary case, no new lower bounds were found.

\begin{table}[ht]\caption{New lower bounds on $K_4(n,R)$\label{covtable1}}
\begin{center}
\begin{tabular}{|r|r|r|r|r|r|}
\hline
&&best&&best lower&\\
&&upper&new&bound&Sphere\\
&&bound &lower&previously&covering\\
$n$ & $R$ &known &bound &known & bound\\ 
\hline
7&1&1008&762&752&745\\
11&1&131072&123846&123362&123362\\
\hline
9&2&1024&748&747&745\\
10&2&4096&2412&2408&2405\\
11&2&16128&7942&7929&7929\\
\hline
11&3&2048&843&842&842\\
\hline
9&4&64&22&21&21\\
11&4&512&134&133&133\\
\hline
11&5&128&31&30&30\\
\hline
11&6&32&10&9&9\\
\hline
\end{tabular}
\end{center}
\end{table}

\begin{table}[ht]\caption{New lower bounds on $K_5(n,R)$\label{covtable2}}
\begin{center}
\begin{tabular}{|r|r|r|r|r|r|}
\hline
&&best&&best lower&\\
&&upper&new&bound&Sphere\\
&&bound &lower&previously&covering\\
$n$ & $R$ &known &bound &known &bound\\ 
\hline
7&1&3125&2722&2702&2694\\
8&1&15625&11945&11887&11838\\
9&1&78125&53138&52800&52788\\
10&1&390625&238993&238200&238186\\
\hline
11&2&115000&52842&52788&52788\\
\hline
11&3&21875&4253&4252&4252\\
\hline
11&4&3125&510&509&509\\
\hline
11&5&625&87&86&86\\
\hline
11&6&125&21&20&20\\
\hline

\end{tabular}
\end{center}
\end{table}

\chapter{Matrix cuts}\label{CH:matrixcuts}
In Chapter \ref{CH:errorcodes} we discussed the problem of finding good upper bounds on the maximum size of a code with certain distance constraints. This is a special case of the general problem to find bounds for the stability number of a graph. There exist general methods for bounding the stability number. In this chapter we explore the relationship between these general methods, when applied to codes, and the method from Chapter \ref{CH:errorcodes}.
 
Recall that for any symmetric matrix $A\in\Real^{n\times n}$, the matrix $R(A)$ is defined by:
\begin{equation}
R(A):=\begin{pmatrix}1&a^{\transp}\\a&A\end{pmatrix},
\end{equation}
where $a:=\diag(A)$ is the vector of diagonal elements of $A$. We will index the extra row and column of $R(A)$ by $0$. Denote the convex set of symmetric matrices
\begin{equation}
\mathcal{R}_n:=\{A\in \Real^{n\times n}\mid R(A)\psd\}.
\end{equation}
Observe that for $A\in\mathcal{R}_n$, the entries of $A$ belong to $[-1,1]$. Indeed, let $i,j\in \{0,\ldots,n\}$. The principal submatrix 
\begin{equation}
\begin{pmatrix}
1& A_{i,i}\\
A_{i,i}& A_{i,i}
\end{pmatrix}
\end{equation}
of $R(A)$ indexed by $0$ and $i$ is positive semidefinite. This is equivalent to $A_{i,i}^2\leq 1\cdot A_{i,i}$, which implies that $A_{i,i}\in [0,1]$. For $i\leq j$ the semidefiniteness of the principal submatrix of $R(A)$ indexed by $i$ and $j$ 
\begin{equation}
\begin{pmatrix}
A_{i,i}& A_{i,j}\\
A_{i,j}& A_{j,j}
\end{pmatrix}
\end{equation}
implies that $A_{i,j}^2\leq A_{i,i}A_{j,j}\leq 1$ and hence $A_{i,j}\in [-1,1]$.
We define the projection $p(\mathcal{M})$ of a set $\mathcal{M}\subseteq \mathcal{R}_n$ and the lift $l(K)$ of a set $K\subseteq [0,1]^n$ by
\begin{eqnarray}
p(\mathcal{M})&:=&\{\diag(A)\mid A\in \mathcal{M}\}\\
l(K)&:=&\{A\in\mathcal{R}_n\mid \diag(A)\in K\}.\nonumber
\end{eqnarray}    
By the previous remarks we see that $p(\mathcal{M})\subseteq [0,1]^n$ for $\mathcal{M}\subseteq \mathcal{R}_n$. Observe that for any $K\subseteq [0,1]^n$ we have $p(l(K))=K$. Indeed, if $x\in [0,1]^n$, the matrix
\begin{equation}
{1\choose x}{1\choose x}^{\transp}+\Diag (0,x_1-x_1^2,\ldots,x_n-x_n^2)
\end{equation}
is a positive semidefinite matrix of the form $R(A)$ with diagonal $x$.
Conversely, we only have $l(p(\mathcal{M}))\supseteq \mathcal{M}$ for $\mathcal{M}\subseteq \mathcal{R}_n$.

In the following, the idea will be for a given convex set $K$, to find approximations of the convex hull of the $0$--$1$ points in $K$. The method will be to describe these approximations as the projection of set in the larger space $\mathcal{R}_n$. The most prominent example is the so-called theta body of a graph, and the associated \notion{Lov\'asz theta number}.
  
\section{The theta body $\mathrm{TH}(G)$}
Let $G=(V,E)$ be a graph. We will assume that the vertex set is given by $V=\{1,\ldots,n\}$. Define the set $\mathcal{M}(G)$ by  
\begin{equation}\label{thetabody1}
\mathcal{M}(G):=\{A\in\mathcal{R}_n\mid A_{i,j}=0\text{ if $\{i,j\}\in E$}\}.
\end{equation}
The projection 
\begin{equation}\label{thetabody2}
\mathrm{TH}(G):=p(\mathcal{M}(G))=\{\diag(A)\mid A\in \mathcal{M}(G)\}
\end{equation}
was defined in \cite{GLS} and is referred to as the \notion{theta body} of $G$. The number
\begin{equation}
\vartheta (G):=\max\{\one^{\transp}x\mid x\in\mathrm{TH}(G)\}
\end{equation}
was introduced by Lov\'asz in \cite{lovasztheta} as an upper bound on the Shannon capacity of the graph $G$. Although we will not be concerned with Shannon capacities, the following two properties of $\vartheta(G)$ are relevant to our discussion: the number $\vartheta (G)$ can be approximated in polynomial time, and gives an (often close) upper bound on the stability number $\alpha(G)$. This last fact follows since for every stable set $S\subseteq V$ in the graph $G$, the matrix $\chi^{S}(\chi^{S})^{\transp}$ belongs to $\mathcal{M}(G)$. The theta body gives a good approximation of the stable set polytope. In particular, for perfect graphs $G$, equality holds, implying that the stability number can be calculated in polynomial time for perfect graphs.

The following strengthening of the theta body was given by Schrijver in \cite{modifiedtheta}. Define 
\begin{equation}
\mathcal{M}'(G):=\{A\in\mathcal{R}_n\mid A\geq 0, A_{i,j}=0\text{ if $\{i,j\}\in E$}\},
\end{equation}
and 
\begin{equation}
\mathrm{TH}'(G):=p(\mathcal{M}'(G)).
\end{equation}
Again the number
\begin{equation}
\vartheta '(G):= \max\{\one^{\transp}x\mid x\in\mathrm{TH}'(G)\}
\end{equation} 
gives an upper bound on $\alpha (G)$ and clearly $\vartheta'(G)\leq \vartheta (G)$. We note that $\vartheta '(G)$ can be alternatively defined by
\begin{eqnarray}\label{alttheta}
\vartheta '(G)=\max\{&\one^{\transp}A\one\mid A\in \Real_{\geq 0}^{n\times n}, \trace A=1,\\
&A_{i,j}=0\text{ when $\{i,j\}\in E$}\},\nonumber
\end{eqnarray}
and similarly for $\vartheta (G)$. The equivalence of the two definitions follows from Propositions \ref{symoptimum} and \ref{twooptima} in Chapter \ref{CH:Preliminaries} (see also \cite{CombOptB}).

It was shown in \cite{modifiedtheta} that for association schemes, the number $\vartheta'(G)$ corresponds to the Delsarte bound. Given a scheme $(X,R)$ with adjacency matrices $I=A_0,A_1,\ldots, A_n$ and $M\subseteq \{1,\ldots,n\}$ we are interested in the maximum size of an $M$-clique, that is a subset $S\subseteq X$ with the property that $(A_i)_{x,y}=0$ for all $x,y\in X$ and $i\not\in M$. Consider the graph $G=(X,E)$, where $E=\{\{x,y\}\mid (A_i)_{x,y}=1\text{ for some $i\not\in M$}\}$. Then the stable sets of $G$ are precisely the $M$-cliques of the scheme $(X,R)$. By (\ref{alttheta}), the upper bound $\vartheta '(G)$ on the maximum size of a stable set in $G$ is given by
\begin{equation}
\max\{\one^{\transp}A\one\mid A\in \Real_{\geq 0}^{X\times X}, \trace A=1, A_{i,j}=0\text{ when $\{i,j\}\in E$}\}.
\end{equation} 
We will sketch a proof that this maximum equals the Delsarte bound. The proof consists of two ideas. 
\begin{proof}
First, we may restrict the range of $A$ in the program to the matrices in the Bose--Mesner algebra, without decreasing the maximum. Indeed, let $\pi$ denote the orthogonal projection onto the Bose--Mesner algebra (as a subspace of $\Real^{X\times X}$) given by
\begin{equation}
\pi(A):=\sum_{i=0}^n \frac{\left<A,E_i\right>}{\left<E_i,E_i\right>}\cdot E_i,
\end{equation}
where the matrices $E_0,\ldots, E_n$ are the orthogonal idempotents of the scheme. Since the $E_i$ have eigenvalues $0$ and $1$, they are positive semidefinite. Hence for positive semidefinite $A$ the projection $\pi(A)$ is a nonnegative combination of positive semidefinite matrices, and hence again positive semidefinite. Furthermore, $\pi$ preserves the inner product with matrices in the Bose-Mesner algebra. In particular 
\begin{eqnarray}
\trace \pi(A)=\left<I,\pi(A)\right>&=&\left<I,A\right>=\trace A\\
\one^{\transp}\pi(A)\one=\left<J,\pi(A)\right>&=&\left<J,A\right>=\one^{\transp}A\one\nonumber\\
\left<A_i,\pi(A)\right>&=&\left<A_i,A\right>=0\quad\text{for $i\not\in M$}\nonumber\\
\left<A_i,\pi(A)\right>&=&\left<A_i,A\right>\geq 0\quad\text{for $i=0,\ldots,n$}.\nonumber
\end{eqnarray}
It follows that $\pi(A)$ is a feasible point with the same objective value as $A$.

Secondly, writing
\begin{equation}
A=\sum_{i=0}^n x_i \widetilde{A}_i,
\end{equation} 
where $\widetilde{A}_i:=\left<A_i,A_i\right>^{-1} A_i$,
the program becomes
\begin{equation}
\max\{\sum_{i\in M} x_i\mid x_0=1, x_i\geq 0 \text{ for $i\in M$}, \sum_{i\in M} x_i \widetilde{A}_i\psd\}.
\end{equation} 
Since
$\widetilde{A}_i=\sum_{j=0}^n Q_{j,i} \left<E_j,E_j\right>^{-1} E_j$, where $Q$ is the second eigenmatrix of the scheme, the positive semidefinite constraint reduces to linear constraints
\begin{equation}
\sum_{i\in M} x_i Q_{j,i} \geq 0\quad \text{for $j=0,\ldots, n$.}
\end{equation}
\end{proof}

We remark that when the Bose--Mesner algebra is the centralizer algebra of its automorphism group, for example in the case of the Hamming schemes and the Johnson schemes, the orthogonal projection $\pi$ satisfies
\begin{equation}
\pi (A)=|\Gamma |^{-1}\sum_{\sigma\in \Gamma} \sigma A,
\end{equation} 
where $\Gamma$ denotes the automorphism group of the scheme.

\section{Matrix cuts}
In \cite{matrixcuts}, Lov\'asz and Schrijver introduced a general \notion{lift and project method} for strengthening approximations of $0$--$1$ polytopes. Given a convex body $K$ contained in the unit cube $[0,1]^n$, a convex body $N_+(K)$ is constructed such that 
\begin{equation}
K\supseteq N_+(K)\supseteq N_+(N_+(K))\supseteq\cdots\supseteq N_+^{(n)}(K)=K\cap\{0,1\}^n.
\end{equation}
An important property of the operator $N_+$ is that for a family $\mathcal{K}$ of convex bodies, if one can optimize in polynomial time over $K$ for each $K\in\mathcal{K}$, then also the optimization problem over $N_+(K)$ is polynomial time solvable for $K\in\mathcal{K}$. An important instance is when $G=(V,E)$ is a perfect graph and $K=\mathrm{FRAC}(G)$ is the fractional stable set polytope of $G$. In that case one iteration of the $N_+$ operator suffices to obtain the stable set polytope $\mathrm{STAB}(G):=\mathrm{FRAC}(G)\cap\{0,1\}^V$. 

We start by describing the lift-and-project-method of Lov\'asz and Schrijver and prove some of the basic properties of the operator $N_+$. The idea is to lift a convex set $K\subseteq [0,1]^n$ to a convex set in the space of symmetric positive semidefinite $n\times n$ matrices and then to project it back into $[0,1]^n$.

For $\mathcal{M}\subseteq \mathcal{R}_n$, define the set $N(\mathcal{M})$ by
\begin{eqnarray}
N(\mathcal{M})&:=\{A\in \mathcal{R}_m\mid &\text{for $i=1,\ldots,n$ there are $U,V\in \mathcal{M}$}\\
&&\text{ such that $A_i=A_{i,i}\cdot\diag(U)$,}\nonumber\\
&&\diag(A)-A_i=(1-A_{i,i})\diag(V)\}.\nonumber
\end{eqnarray}
The operator $N_+$ is now defined as
\begin{equation}
N_+(K):=p(N(l(K))).
\end{equation}
Clearly 
\begin{equation}
N_+(K)\subseteq [0,1]^n,
\end{equation}
since if $R(A)$ is positive semidefinite $\diag(A)\in [0,1]^n$ as we have seen before. Furthermore, we have:
\begin{equation}
N_+(K)\subseteq K.
\end{equation}
Indeed, if $A\in N(l(K))$, then for any $i=1,\ldots,n$ we have:
\begin{equation}
\diag(A)=A_i+(\diag(A)-A_i)\in A_{i,i}\cdot K+(1-A_{i,i})\cdot K
\end{equation}
and hence $\diag(A)\in K$, since $A_{i,i}\in [0,1]$. Note that this argument shows that in fact
\begin{equation}
N_+(K)\subseteq \mathrm{conv.hull}\{x\in K\mid x_i\in\{0,1\}\}
\end{equation}
for $i=1,\ldots n$ since for each $i$ we have 
\begin{equation}
A_i\in A_{i,i}\cdot \{x\in K\mid x_i=1\}, \quad \diag(A)-A_i\in (1-A_{i,i})\cdot \{x\in K\mid x_i=0\}. 
\end{equation}
By induction it then follows that
\begin{equation} 
N_+^n(K)\subseteq \{0,1\}^n\cap K.
\end{equation}
On the other hand, $N_+$ does not cut off any integer points:  
\begin{equation}
\{0,1\}^n\cap K\subseteq N_+(K),
\end{equation}
since for any $x\in \{0,1\}^n\cap K$ the matrix $xx^{\transp}$ belongs to $N(l(K))$.
The operator $N_+$ was introduced in \cite{matrixcuts}, see also \cite{CombOptB}. 

Let $G=(V,E)$ be a graph and let $\mathrm{FRAC}(G)$ denote the fractional stable set polytope of $G$, that is:
\begin{equation}
\mathrm{FRAC}(G):=\{x\in \Real^V\mid x\geq 0, x_i+x_j\leq 1 \text{ for any edge $\{i,j\}\in E$}\}.
\end{equation} 
We observe that $N_+(\mathrm{FRAC}(G))$ is contained in the modified theta body $\mathrm{TH}'(G)$. Indeed, if $A\in N(l(\mathrm{FRAC}(G)))$, then $A_{i,j}=0$ for any edge $\{i,j\}$ since $A_i\in A_{i,i}\cdot\mathrm{FRAC}(G)$ implies that
\begin{equation}
A_{i,i}+A_{i,j}\leq A_{i,i}\cdot 1.
\end{equation}

We will also consider the operator $\widetilde{N}$ given by
\begin{eqnarray}
\widetilde{N}(\mathcal{M})&:=\{A\in \mathcal{R}_n\mid &\text{for $i=1,\ldots,n$ there are}\\
&& U\in A_{i,i}\mathcal{M}, V\in (1-A_{i,i})\mathcal{M}\nonumber\\
&&\text{ such that $U_{i,i}=A_{i,i}$ and $A=U+V$}\}.\nonumber
\end{eqnarray}
Clearly $\widetilde{N}(\mathcal{M})\subseteq N(\mathcal{M})$ for any $\mathcal{M}\subseteq \mathcal{R}_n$. We show that any $x\in p(\mathcal{M})\cap \{0,1\}^n$ belongs to $p(N(\mathcal{M}))$. Let $A\in\mathcal{M}$ have diagonal $x\in\{0,1\}^n$. Then for $i=1,\ldots,n$ we have $A=A_{i,i}U+(1-A_{i,i})V$, where we take $U=A,V=0$ if $A_{i,i}=1$ and $U=0, V=A$ if $A_{i,i}=0$. 
The set $\widetilde{N}(\mathcal{M})$ can alternatively be described
by:
\begin{equation}
\widetilde{N}(\mathcal{M})=\{A\mid A\in\text{conv.hull} \{M\in \mathcal{M}\mid M_{i,i}\in\{0,1\}\}\text{ for each $i=1,\ldots,n$}\}.
\end{equation}

\begin{proof}
Let $A\in \widetilde{N}(\mathcal{M})$. Let $i\in\{1,\ldots,n\}$ and let $U,V$ be as in the definition. Observe that $R(A)$ is positive semidefinite, since $A=U+V\in A_{i,i}\mathcal{M}+(1-A_{i,i})\mathcal{M}=\mathcal{M}$. We prove that 
\begin{equation}\label{claim}
A_i=\diag(U) \text{ and }\diag(A)-A_i=\diag(V).
\end{equation}
If $A_{i,i}=0$ we have $A_i=0$ since $A$ is positive semidefinite, and (\ref{claim}) follows. Hence we may assume that $A_{i,i}>0$. Notice that $V_{i,i}=0$ and hence $V_i=0$. Since $A_{i,i}^{-1}U_{i,i}=1$ it follows from the positive semidefiniteness of $R(A_{i,i}^{-1}U)$ that $U_i=\diag(U)$. 
Hence $A_i=U_i+V_i=\diag(U)$ and $\diag(A) -A_i=\diag(U)\ U_i+\diag(V)+V_i=\diag(V)$.
\end{proof}

\section{Bounds for codes using matrix cuts}
Fix integers $1\leq d\leq n$ and $q\geq 2$, and fix an alphabet $\q=\{0,1,\ldots,q-1\}$. The Hamming distance $d(x,y)$ of two words $x$ and $y$ is defined as the number of positions in which $x$ and $y$ differ. Let $G=(V,E)$ be the graph with $V=\q^n$, where two different \emph{words} $x,y\in V$ are joined by an edge if $x$ and $y$ differ in at most $d-1$ position. The stable sets in $G$ are precisely the $q$-ary codes of length $n$ and minimum distance at most $d$. The stability number of $G$ equals $A_q(n,d)$. 
Define
\begin{equation}
\mathcal{M}':=\{A\in\mathcal{R}_V\mid A\geq 0, A_{x,y}=0 \text{ if } \{x,y\}\in E\},
\end{equation}
and let
\begin{equation}
\mathrm{TH'}(G):=p(\mathcal{M}')=\{\diag(A)\mid A\in\mathcal{M}'\}
\end{equation}
denote the modified theta body of $G$. Maximizing the all-one vector over $\mathrm{TH'}(G)$ gives an upper bound on $A_q(n,d)$, which we have seen, equals the Delsarte bound. A tighter upper bound can be found by maximizing the all-one vector over the smaller convex set $N_+(\mathrm{TH'}(G))$:
\begin{equation}\label{mc}
\max \{\trace A\mid A\in N(\mathcal{M})\}.
\end{equation} 
Using the symmetries of the graph $G$, this can be made more explicit as follows.

Denote by $\aut$ the set of permutations of $\q^n$ that preserve the Hamming distance. It is not hard to see that $\aut$ consists of the permutations of $\q^n$ obtained by permuting the $n$ coordinates followed by independently permuting the alphabet $\q$ at each of the $n$ coordinates. The group $\aut$ acts on the set of $V\times V$ matrices in the following way. For $\sigma\in\aut$ and $A\in \Real^{V\times V}$ define $\sigma(A)$ by
\begin{equation}
(\sigma (A))_{\sigma x,\sigma y}=A_{x,y}.
\end{equation}
The matrices in $\Real^{V\times V}$ that are invariant under this action of $\aut$ are precisely the adjacency matrices $A_0,A_1,\ldots,A_n$ of the Hamming scheme $H(n,q)$ defined by
\begin{equation}
(A_i)_{x,y}:\begin{cases}
1&\text{if $d(x,y)=i$,}\\
0&\text{otherwise,}
\end{cases}
\end{equation}
for $i=0,1,\ldots,n$ and the Bose--Mesner algebra of the Hamming scheme.

In the following calculations, it will be convenient do define for a square matrix $A$ and a positive real number $c$ the matrix $R(c;A)$ by
\begin{equation}
R(c;A):=\begin{pmatrix}c & (\diag(A))^{\transp}\\ \diag(A)&A\end{pmatrix}.
\end{equation}
Observe that $R(1;A)=R(A)$ and $R(c;A)$ is positive semidefinite if and only if $R(c^{-1}A)$ is positive semidefinite.
Since $G$ is invariant under the permutations $\sigma\in\aut$, also $\mathcal{M}'$, $ N(\mathcal{M}')$ and $N_+(\mathrm{TH'}(G))$ are invariant under the action of $\aut$. Hence if $A\in N(\mathcal{M}')$ maximizes $\trace M$ over all $M\in N(\mathcal{M}')$, also 
\begin{equation}
\frac{1}{|\aut|}\sum_{\sigma\in\aut} \sigma (A)\in  N(\mathcal{M}')
\end{equation}
is a maximizer. Hence the maximum in (\ref{mc}) is equal to
\begin{equation}
\max \{\trace A\mid A\in  N(\mathcal{M}')\text { is in the Bose--Mesner algebra}\}.
\end{equation}   
If $A$ is a matrix in the Bose-Mesner algebra, all rows of $A$ are equal up to permuting by elements of $\aut$. Hence since $\mathcal{M}'$ is invariant under these permutations, the maximum is equal to
\begin{equation}\label{mc2}
\begin{split}
\max \{q^n\cdot x_0\mid &R(\sum_{i=0}^{n} x_iA_i)\text { is positive semidefinite}\\
&\text{and there exist $U\in x_0\cdot\mathcal{M}'$ and $V\in (1-x_0)\cdot \mathcal{M}'$ such that}\\
&U_{u,u}=x_i\text{ if $d(u,0)=i$ and}\\
&V_{u,u}=x_0-x_i\text{ if $d(u,0)=i$}\}.
\end{split}
\end{equation}   
Note that if $U$ and $V$ are as in (\ref{mc2}), and $\sigma\in\aut$ fixes the zero word, then $\sigma(U)$ and $\sigma (V)$ satisfy the same constraints. Hence $U$ may be replaced by 
\begin{equation}
\frac{1}{|\stab|}\sum_{\sigma\in\stab} \sigma(U)
\end{equation}
and similarly for $V$. Here $\stab$ denotes the set $\{\sigma\in\aut\mid \sigma(0)=0\}$. Hence we may impose that $U$ and $V$ are elements of the Terwilliger algebra without changing the maximum. We obtain
\begin{equation}
\begin{split}
\max\{q^n\cdot x_0\mid &R(\sum_{i=0}^{n} x_iA_i)\text { is positive semidefinite},\\
&R(\sum_{i,j,t,p} y_{i,j}^{t,p} M_{i,j}^{t,p}), R(\sum_{i,j,t,p} z_{i,j}^{t,p}M_{i,j}^{t,p})\text{ are positive semidefinite},\\
&x_i=x_0y_{i,i}^{i,i},\quad x_0-x_i=(1-x_0)z_{i,i}^{i,i}\quad (i=0,\ldots,n),\\
&y_{i,j}^{t,p}, z_{i,j}^{t,p}\geq 0,\\
&y_{i,j}^{t,p}=z_{i,j}^{t,p}=0 \text{ if $i+j-t-p\in\{1,\ldots,d-1\}$}\}.
\end{split}
\end{equation}  
Since $x_i=x_0\cdot y_{i,i}^{i,i}$ we can eliminate the variables $x_i$ from this program by substituting 
\begin{eqnarray}
\widetilde{y}_{i,j}^{t,p}&:=&x_0\cdot y_{i,j}^{t,p}\\
\widetilde{z}_{i,j}^{t,p}&:=&(1-x_0)z_{i,j}^{t,p}.\nonumber
\end{eqnarray}
We obtain the following semidefinite program (where we have dropped all the tilde's from the variables):
\begin{equation}\label{matrixcut}
\begin{split}
\max\{q^n\cdot y_{0,0}^{0,0}\mid & R(\sum_{i=0}^{n} y_{i,i}^{i,i} A_i)\text { is positive semidefinite},\\
&R(y_{0,0}^{0,0};\sum_{i,j,t,p} y_{i,j}^{t,p}M_{i,j}^{t,p}), R(1-y_{0,0}^{0,0};\sum_{i,j}^{t,p} z_{i,j}^{t,p} M_{i,j}^{t,p}) \psd\\
&y_{0,0}^{0,0}-y_{i,i}^{i,i}=z_{i,i}^{i,i}\quad (i=0,\ldots,n),\\
&y_{i,j}^{t,p}, z_{i,j}^{t,p}\geq 0,\\
&y_{i,j}^{t,p}=z_{i,j}^{t,p}=0 \text{ if $i+j-t-p\in\{1,\ldots,d-1\}$}\}.
\end{split}
\end{equation}  

If we use the stronger operator $\widetilde{N}$ in stead of $N$ we arrive in a similar fashion at the program
\begin{equation}
\begin{split}
\max\{q^n\cdot x_0\mid &R(x_0;\sum_{i,j,t,p} y_{i,j}^{t,p} M_{i,j}^{t,p}), R(1-x_0;\sum_{i,j,t,p} z_{i,j}^{t,p}M_{i,j}^{t,p})\psd,\\
&y_{0,0}^{0,0}=x_0, y_{i,j}^{t,p}+z_{i,j}^{t,p}=x_{i+j-t-p}\\
&y_{i,j}^{t,p},z_{i,j}^{t,p}\geq 0, y_{i,j}^{t,p}=z_{i,j}^{t,p}=0\text{ if $i+j-t-p\in\{1,\ldots,d-1\}$}\}
\end{split}
\end{equation}
it follows from $y_{0,0}^{0,0}+z_{0,0}^{0,0}=x_0$ and $y_{0,0}^{0,0}=x_0$ that $z_{0,0}^{0,0}=0$. Since for each feasible solution the matrix
\begin{equation}
M:=\sum_{i,j}^{t,p} z_{i,j}^{t,p} M_{i,j}^{t,p}
\end{equation}

is positive semidefinite, it follows from $M_{\zero,\zero}=z_{0,0}^{0,0}=0$ that $0=M_{u,\zero}=z_{i,0}^{0,0}$ when $\mathbf{u}$ has weight $i$. Hence
\begin{equation}
x_{i}=y_{i,0}^{0,0}+z_{i,0}^{0,0}=y_{i,0}^{0,0}.
\end{equation}
This implies that we can eliminate the variables $x_i$ and $z_{i,j}^{t,p}$ by using
\begin{equation}
x_i=y_{i,0}^{0,0},\quad z_{i,j}^{t,p}=y_{i+j-t-p,0}^{0,0}-y_{i,j}^{t,p}
\end{equation}
for all $i,j,t,p$. We obtain the following semidefinite program:

\begin{equation}
\begin{split}
\max\{q^n\cdot y_{0,0}^{0,0}\mid &R(y_{0,0}^{0,0};\sum_{i,j,t,p} y_{i,j}^{t,p} M_{i,j}^{t,p}), R(1-y_{0,0}^{0,0};\sum_{i,j,t,p} (y_{i+j-t-p,0}^{0,0}-y_{i,j}^{t,p})M_{i,j}^{t,p})\psd,\\
&y_{i,j}^{t,p},y_{i+j-t-p,0}^{0,0}-y_{i,j}^{t,p}\geq 0,\\
&y_{i,j}^{t,p}=y_{i+j-t-p,0}^{0,0}=0\text{ if $i+j-t-p\in\{1,\ldots,d-1\}$}\}
\end{split}
\end{equation}
This program may be further simplified by observing that for a matrix $A$ with $A_{1,1}=1$ 
\begin{equation}
R(A)\psd \quad\text{if and only if}\quad A\psd, A_1=\diag(A).
\end{equation}
We finally obtain
\begin{equation}\label{newcut}
\begin{split}
\max\{q^n\cdot y_{0,0}^{0,0}\mid &\sum_{i,j,t,p} y_{i,j}^{t,p} M_{i,j}^{t,p}, R(1-y_{0,0}^{0,0};\sum_{i,j,t,p} (y_{i+j-t-p,0}^{0,0}-y_{i,j}^{t,p})M_{i,j}^{t,p})\psd,\\
&y_{i,j}^{t,p},y_{i+j-t-p,0}^{0,0}-y_{i,j}^{t,p}\geq 0,\\
&y_{i,0}^{0,0}=y_{i,i}^{i,i},\\
&y_{i,j}^{t,p}=y_{i+j-t-p,0}^{0,0}=0\text{ if $i+j-t-p\in\{1,\ldots,d-1\}$}\}
\end{split}
\end{equation}
This bound is very similar to the bound (\ref{moniquesbound}) derived in Chapter \ref{CH:errorcodes}, that is to say the impoved version of Schrijver's bound given by Laurent. The main difference is that in (\ref{newcut}) the symmetry conditions 
\begin{equation}
\begin{split}
y_{i,j}^{t,p}=y_{i',j'}^{t',p'}\quad \text{when }& \text{$t-p=t'-p'$ and $(i,j,i+j-t-p)$ is a}\\
&\text{\quad permutation of $(i',j',i'+j'-t'-p')$}
\end{split}
\end{equation}
are lacking. It can be seen from the computational results in given in the next section, that these conditions make a huge difference in the resulting bound. 

\section{Computational results}
In this section we give some computational results on the different bounds we obtain and compare them to the bound proposed by Schrijver (see \cite{Lexcodes}, \cite{nonbincodes}) with the improvement of Laurent. That is, the bound obtained from (\ref{moniquesbound}).  Each of the bounds can be computed in polynomial time in $n$ by block diagonalising the Terwilliger algebra of the Hamming scheme for each $q$ and $n$.

From the tables below it follows that we can have the strict inequality
\begin{equation}
\widetilde{N}(\mathcal{M})\subset  N(\mathcal{M}).
\end{equation}

\begin{table}[ht]\caption{Bounds on $A_3(n,d)$\label{tab3}}
\begin{center}
\begin{tabular}{|r|r|r|r||r|r|r|r|}
\hline
&&best&best upper&&&&\\
&&lower&bound&&bound&bound&bound\\
&&bound &previously&Delsarte&from&from&from\\
$n$ & $d$ &known &known & bound&(\ref{matrixcut})&(\ref{newcut})&(\ref{moniquesbound})\\ 
\hline
6&3&38&38&48&48&48&46\\
7&3&99&111&145&145&144&136\\
8&3&243&333&340&340&340&340\\
9&3&729&937&937&937&937&937\\
\hline
7&4&33&33&48&48&48&44\\
8&4&99&99&139&139&139&121\\
9&4&243&297&340&340&339&324\\
10&4&729&891&937&937&937&914\\
11&4&1458&2561&2811&2811&2805&2583\\
12&4&4374&7029&7029&7029&7029&6839\\
\hline
6&5&4&4&5&5&4&4\\
7&5&10&10&15&15&14&13\\
8&5&27&27&41&41&41&33\\
9&5&81&81&90&90&90&86\\
10&5&243&243&243&243&243&243\\
11&5&729&729&729&729&729&729\\
12&5&729&1562&1562&1562&1562&1557\\
\hline
7&6&3&3&4&4&4&4\\
\hline
9&7&6&6&7&7&7&7\\
10&7&14&14&21&21&21&21\\
11&7&36&36&63&63&62&49\\
12&7&54&108&138&138&138&131\\
\hline
\end{tabular}
\end{center}
\end{table}

\begin{table}[ht]\caption{Bounds on $A_4(n,d)$\label{tab2}}
\begin{center}
\begin{tabular}{|r|r|r|r||r|r|r|r|}
\hline
&&best&best upper&&&&\\
&&lower&bound&&bound&bound&bound\\
&&bound &previously&Delsarte&from&from&from\\
$n$ & $d$ &known &known & bound&(\ref{matrixcut})&(\ref{newcut})&(\ref{moniquesbound})\\ 
\hline
7&4&128&179&179&179&179&169\\
8&4&320&614&614&614&614&611\\
9&4&1024&2340&2340&2340&2340&2314\\
10&4&4096&9360&9362&9362&9360&8951\\
\hline
7&5&32&32&40&40&40&39\\
8&5&70&128&160&160&160&147\\
9&5&256&512&614&614&614&579\\
10&5&1024&2048&2145&2145&2145&2045\\
\hline
10&6&256&512&512&512&511&496\\
11&6&1024&2048&2048&2048&2047&1780\\
12&6&4096&6241&6241&6241&6241&5864\\
\hline
10&7&40&80&112&112&111&106\\
12&7&256&1280&1280&1280&1280&1167\\
\hline
\end{tabular}
\end{center}
\end{table}

\begin{table}[ht]\caption{Bounds on $A_5(n,d)$\label{tab1}}
\begin{center}
\begin{tabular}{|r|r|r|r||r|r|r|r|}
\hline
&&best&best upper&&&&\\
&&lower&bound&&bound&bound&bound\\
&&bound &previously&Delsarte&from&from&from\\
$n$ & $d$ &known &known & bound&(\ref{matrixcut})&(\ref{newcut})&(\ref{moniquesbound})\\ 
\hline
6&4&125&125&125&125&125&125\\
7&4&250&554&625&625&623&545\\
8&4&1125&2291&2291&2291&2291&2291\\
9&4&3750&9672&9672&9672&9672&9672\\
10&4&15625&44642&44642&44642&44642&44642\\
11&4&78125&217013&217013&217013&217013&217013\\
\hline
7&5&53&125&125&125&124&108\\
8&5&160&554&625&625&623&485\\
9&5&625&2291&2291&2291&2291&2152\\
10&5&3125&9672&9672&9672&9672&9559\\
11&5&15625&44642&44642&44642&44642&44379\\
\hline
8&6&45&75&75&75&75&75\\
9&6&135&375&375&375&375&375\\
10&6&625&1875&1875&1875&1875&1855\\
11&6&3125&9375&9375&9375&9375&8840\\
\hline
11&9&25&35&45&45&45&43\\
\hline
\end{tabular}
\end{center}
\end{table}
Remark: We calculate the Delsarte bound by maximizing $x_0\cdot q^n$ (the trace of $\sum_{i}x_iM_i$) under the condition that the $x_i$ are nonnegative and $R(\sum_i x_iM_i)$ is positive semidefinite. This turns out to give a more stable semidefinite program than setting $x_0=1$ and maximizing $\sum_i x_i {n\choose i}(q-1)^i$.

\chapter{Further discussion}
In this chapter, we present some further observations, and notes related to the methods from previous chapters.

\section{Bounds for affine caps}
Let $\mathrm{AG}(k,q)$ be the $k$-dimensional affine space over the field $\mathrm{GF}_q$. A subset $A\subseteq \mathrm{AG}(k,q)$ is called an \notion{affine cap} if no three elements of $A$ are on an affine line, that is, any three different vectors in $\{ {1\choose a}\mid a\in A\}$ are linearly independent. We denote by $C_k(q)$ the maximum cardinality of an affine cap in $\mathrm{AG}(k,q)$.

The effectiveness of the semidefinite programming approach for error correcting codes, suggested that we could, more generally, find good bounds for the size of a code where we forbid the occurence of triples of code words in some prescribed configuration. Indeed the variables $x_{i,j}^{t,p}$ in the semidefinite program correspond exactly to the number of tripes in a code, for each equivalence class under automorphisms of the Hamming space. One may be led to wonder if setting those variables that correspond to forbidden configurations to zero, would yield good upper bounds in general. This is an appealing idea. Unfortunately, it turned out to be false in general. 

A prominent structure that might be approached this way are affine caps over the field of three elements. The only known values are $C_1(3)=2, C_2(3)=4, C_3(3)=9, C_4(3)=20$ and $C_5(3)=45$. In \cite{Bierbrauer} the general bound $C_k(3)\leq 3^k\frac{k+1}{k^2}$ was shown. A code $A\subseteq \mathrm{AG}(k,3)$ is an affine cap if and only if for any three elements $\mathbf{u},\mathbf{v},\mathbf{w}\in A$ we have $d(\mathbf{u},\mathbf{v},\mathbf{w})\not=(i,i,i,0)$ for every $i=1,\ldots,k$. Recall that 
\begin{eqnarray}
d(\mathbf{u},\mathbf{v},\mathbf{w})&:=&(i,j,t,p),\text{\ where}\\
&&i:=d(\mathbf{u},\mathbf{v}),\nonumber\\
&&j:=d(\mathbf{u},\mathbf{w}),\nonumber\\
&&t:=|\{i\mid \mathbf{u}_i\not=\mathbf{v}_i\text{\ and\ }\mathbf{u}_i\not=\mathbf{w}_i\}|,\nonumber\\
&&p:=|\{i\mid \mathbf{u}_i\not=\mathbf{v}_i=\mathbf{w}_i\}|\nonumber.
\end{eqnarray}

Consider the following semidefinite program.

\begin{eqnarray}\label{affinecap}
&&\text{maximize } \sum_{i=0}^n {n\choose i}2^i x_{i,0}^{0,0}\quad \text{subject to}\\
\text{(i)}&&x_{0,0}^{0,0}=1\nonumber\\
\text{(ii)}&&0\leq x_{i,j}^{t,p}\leq x_{i,0}^{0,0}\nonumber\\
\text{(iii)}&&x_{i,j}^{t,p}=x_{i',j'}^{t',p'}\text{\ if\ } t-p=t'-p'\text{\ and}\nonumber\\
&&(i,j,i+j-t-p)\text{\ is a permutation of\ } (i',j',i'+j'-t'-p')\nonumber\\
\text{(iv)}&&x_{i,i}^{i,0}=0 \text{\ for $i=1,\ldots,n$.\ }\nonumber\\
\text{(v)}&&\sum_{i,j,t,p}x_{i,j}^{t,p}M_{i,j}^{t,p}, \sum_{i,j,t,p}(x_{i+j-t-p,0}^{0,0}-x_{i,j}^{t,p})M_{i,j}^{t,p}\text{\ are positive semidefinite}.\nonumber
\end{eqnarray}

Clearly, this gives an upper bound on $C_n(3)$. We have the following result.
\begin{proposition}
The maximum in (\ref{affinecap}) equals $1+\frac{3^n-1}{2}$.
\end{proposition}
\begin{proof}
Setting the variables $x_{i,j}^{t,p}$ as follows:
\begin{equation}
x_{i,j}^{t,p}:=\begin{cases}
1&\text{if $i=j=t=p=0$,}\\
\frac{1}{2}&\text{if $i=j=t=p\not=0$ or exactly one of $i,j$ is zero,}\\
0&\text{if $i=j=t\not=0$ and $p=0$,}\\
\frac{1}{4}&\text{otherwise}
\end{cases}
\end{equation}
gives a feasible solution with objective value equal to $1+\frac{3^n-1}{2}$.

On the other hand, let any feasible solution be given. Then the matrix $M':=\sum_{i,j,t,p} x_{i,j}^{t,p} M_{i,j}^{t,p}$ is positive semidefinite. Hence for any nonzero word $\mathbf{u}$ the $3\times 3$ principal submatrix indexed by the words $\zero,\mathbf{u},-\mathbf{u}$ is positive semidefinite and equals
\begin{equation}
\begin{pmatrix}
1&M_{\mathbf{u},\mathbf{u}}&M_{-\mathbf{u},-\mathbf{u}}\\
M_{\mathbf{u},\mathbf{u}}&M_{\mathbf{u},\mathbf{u}}&0\\
M_{-\mathbf{u},-\mathbf{u}}&0&M_{-\mathbf{u},-\mathbf{u}}.
\end{pmatrix}
\end{equation}
This implies that $(M_{\mathbf{u},\mathbf{u}}-M_{\mathbf{u},\mathbf{u}}^2)(M_{-\mathbf{u},-\mathbf{u}}-M_{-\mathbf{u},-\mathbf{u}}^2)\geq (M_{\mathbf{u},\mathbf{u}}M_{-\mathbf{u},-\mathbf{u}})^2$. Hence $M_{\mathbf{u},\mathbf{u}}+M_{-\mathbf{u},-\mathbf{u}}\leq 1$. Since the objective function equals the trace of $M$, the value is at most $1+\frac{3^n-1}{2}$.
\end{proof}

This bound is very poor. The same bound already follows from the fact that if $\zero\in A$, for every nonzero word $\mathbf{u}$ not both $\mathbf{u}$ and $-\mathbf{u}$ can belong to $A$.

\section{Notes on computational results}
The computational results from Chapters \ref{CH:errorcodes} and \ref{CH:coveringcodes} have been obtained by using CSDP version 4.7 (see \cite{csdp}) and SDPT3 (see \cite{sdpt3}). Both are \notion{SDP-solvers} and can be accessed also through the NEOS server, see 
\begin{center}
\texttt{www-neos.mcs.anl.gov/}
\end{center}. The semidefinite programs were generated by perl scripts in the sparse SDPA format, which allows for explicit block structure in the constraint matrices to be exploited by the solver.

In the case of error correcting codes (tables 1,2,3 from Chapter \ref{CH:errorcodes}), all solutions produced by the solvers have been examined by a perl script to ensure that the produced numbers really do give valid upper bounds on error correcting codes. In none of the instances this has made a diference for the final bound obtained. This was done as follows.

The original problem was to maximize $\sum_{i}{n\choose i}(q-1)^i x_{i,0}^{0,0}$, given certain constraints on the variables $x_{i,j}^{t,p}$ (\ref{primesdp}). By changing the sign of the objective vector, we obtain a semidefinite program of the following form:
\begin{eqnarray}
\text{minimize}&&x_1c_1+\cdots+x_mc_m\\
\text{subject to}&&x_1F_1+\cdots+x_mF_m-F_0=:X\psd,\nonumber
\end{eqnarray}
where we minimize over $x^{\transp}=(x_1,\ldots,x_m)$, $c^{\transp}=(c_1,\ldots,c_m)$ is the objective vector and $F_0,\ldots,F_m$ are given symmetric matrices. The SDP solver not only returns a solution to this (primal) problem, but also to its dual:
\begin{eqnarray}
\text{maximize}&&\left<F_0,Y\right>\\
\text{subject to}&&\left<F_i,Y\right>=c_i,\quad i=1,\ldots, m,\nonumber\\
&&Y\psd\nonumber.
\end{eqnarray}
Any genuine feasible matrix $Y$ for the dual problem gives a lower bound on the minimum in the primal problem, and hence an upper bound for our coding problem. However, the produced dual solutions $Y$ usually do not exactly satisfy the linear constraint, but satisfy
\begin{equation}
\left<F_i,Y\right>=c_i+\epsilon_i,
\end{equation}
for small numbers $\epsilon_1,\ldots,\epsilon_m$. In all cases we did find that $Y\psd$ was satisfied. This yields a lower bound on the optimum of the primal program as follows. Let $(x,X)$ be an optimal solution for the primal program with value $O$. Then we obtain:
\begin{eqnarray}
\left<F_0,Y\right>&=&\left<x_1F_1+\cdots+x_mF_m-X,Y\right>\\
&\leq&\left<x_1F_1+\cdots+x_mF_m,Y\right>\nonumber\\
&=&x_1\left<F_1,Y\right>+\cdots+x_m\left<F_m,Y\right>\nonumber\\
&=&x_1(c_1+\epsilon_1)+\cdots+x_m(c_m+\epsilon_m)\nonumber\\
&=&O+(x_1\epsilon_1+\cdots+x_m\epsilon_m).\nonumber
\end{eqnarray}
The numbers $\epsilon_1,\ldots,\epsilon_m$ are easily calculated from the solution $Y$. Although the  numbers $x_1,\ldots,x_m$ are not known, we can say that $x_i\in [0,1]$ in this case, which allows us to bound the error term by
\begin{equation}
(x_1\epsilon_1+\cdots+x_m\epsilon_m)\leq \max \{0,\epsilon_1\}+\cdots+\max \{0,\epsilon_m\}.
\end{equation}
This gives lower bound on $O$, and hence an upper bound on the maximum size of a code.

\cleardoublepage
\bibliographystyle{abbrv}
\addcontentsline{toc}{chapter}{Bibliography}
\bibliography{auxil,errorcodes,coveringcodes}

\begin{thebibliography}{10}

\bibitem{Bailey}
R.~A. Bailey.
\newblock {\em Association schemes}, volume~84 of {\em Cambridge Studies in
  Advanced Mathematics}.
\newblock Cambridge University Press, Cambridge, 2004.
\newblock Designed experiments, algebra and combinatorics.

\bibitem{BannaiIto}
E.~Bannai and T.~Ito.
\newblock {\em Algebraic combinatorics. {I}}.
\newblock The Benjamin/Cummings Publishing Co. Inc., Menlo Park, CA, 1984.
\newblock Association schemes.

\bibitem{bdref}
G.~P. Barker, L.~Q. Eifler, and T.~P. Kezlan.
\newblock A non-commutative spectral theorem.
\newblock {\em Linear Algebra and Appl.}, 20(2):95--100, 1978.

\bibitem{Bierbrauer}
J.~Bierbrauer and Y.~Edel.
\newblock Bounds on affine caps.
\newblock {\em J. Combin. Des.}, 10(2):111--115, 2002.

\bibitem{fourcodes}
G.~T. Bogdanova, A.~E. Brouwer, S.~N. Kapralov, and P.~R.~J. {\"O}sterg{\aa}rd.
\newblock Error-correcting codes over an alphabet of four elements.
\newblock {\em Des. Codes Cryptogr.}, 23(3):333--342, 2001.

\bibitem{fivecodes}
G.~T. Bogdanova and P.~R.~J. {\"O}sterg{\aa}rd.
\newblock Bounds on codes over an alphabet of five elements.
\newblock {\em Discrete Math.}, 240(1-3):13--19, 2001.

\bibitem{csdp}
B.~Borchers.
\newblock C{SDP}, a {C} library for semidefinite programming.
\newblock {\em Optim. Methods Softw.}, 11/12(1-4):613--623, 1999.
\newblock Interior point methods.

\bibitem{Bose}
R.~C. Bose and T.~Shimamoto.
\newblock Classification and analysis of partially balanced incomplete block
  designs with two associate classes.
\newblock {\em J. Amer. Statist. Assoc.}, 47:151--184, 1952.

\bibitem{brouwersite}
A.~E. Brouwer.
\newblock Website: \texttt{http://www.win.tue.nl/$\sim$aeb}.

\bibitem{Distreggraph}
A.~E. Brouwer, A.~M. Cohen, and A.~Neumaier.
\newblock {\em Distance-regular graphs}, volume~18 of {\em Ergebnisse der
  Mathematik und ihrer Grenzgebiete (3) [Results in Mathematics and Related
  Areas (3)]}.
\newblock Springer-Verlag, Berlin, 1989.

\bibitem{threecodes}
A.~E. Brouwer, H.~O. H{\"a}m{\"a}l{\"a}inen, P.~R.~J. {\"O}sterg{\aa}rd, and
  N.~J.~A. Sloane.
\newblock Bounds on mixed binary/ternary codes.
\newblock {\em IEEE Trans. Inform. Theory}, 44(1):140--161, 1998.

\bibitem{coherent}
P.~J. Cameron.
\newblock Coherent configurations, association schemes and permutation groups.
\newblock In {\em Groups, combinatorics \& geometry (Durham, 2001)}, pages
  55--71. World Sci. Publishing, River Edge, NJ, 2003.

\bibitem{ChenHonkala}
W.~D. Chen and I.~S. Honkala.
\newblock Lower bounds for {$q$}-ary covering codes.
\newblock {\em IEEE Trans. Inform. Theory}, 36(3):664--671, 1990.

\bibitem{Coveringcodes}
G.~Cohen, I.~Honkala, S.~Litsyn, and A.~Lobstein.
\newblock {\em Covering codes}, volume~54 of {\em North-Holland Mathematical
  Library}.
\newblock North-Holland Publishing Co., Amsterdam, 1997.

\bibitem{Delsarte}
P.~Delsarte.
\newblock An algebraic approach to the association schemes of coding theory.
\newblock {\em Philips Res. Rep. Suppl.}, (10):vi+97, 1973.

\bibitem{nonbincodes}
D.~Gijswijt, A.~Schrijver, and H.~Tanaka.
\newblock New upper bounds for nonbinary codes based on the terwilliger algebra
  and semidefinite programming.
\newblock Submitted, November 2004.

\bibitem{GLS}
M.~Gr{\"o}tschel, L.~Lov{\'a}sz, and A.~Schrijver.
\newblock {\em Geometric algorithms and combinatorial optimization}, volume~2
  of {\em Algorithms and Combinatorics}.
\newblock Springer-Verlag, Berlin, second edition, 1993.

\bibitem{Habsieger}
L.~Habsieger.
\newblock Some new lower bounds for ternary covering codes.
\newblock {\em Electron. J. Combin.}, 3(2):Research Paper 23, approx.\ 14 pp.\
  (electronic), 1996.
\newblock The Foata Festschrift.

\bibitem{HabsiegerPlagne}
L.~Habsieger and A.~Plagne.
\newblock New lower bounds for covering codes.
\newblock {\em Discrete Math.}, 222(1-3):125--149, 2000.

\bibitem{Footballpool}
H.~H{\"a}m{\"a}l{\"a}inen, I.~Honkala, S.~Litsyn, and P.~{\"O}sterg{\aa}rd.
\newblock Football pools---a game for mathematicians.
\newblock {\em Amer. Math. Monthly}, 102(7):579--588, 1995.

\bibitem{HonkalaIiro}
I.~S. Honkala.
\newblock Lower bounds for binary covering codes.
\newblock {\em IEEE Trans. Inform. Theory}, 34(2):326--329, 1988.

\bibitem{HornJohnson}
R.~A. Horn and C.~R. Johnson.
\newblock {\em Matrix analysis}.
\newblock Cambridge University Press, Cambridge, 1990.
\newblock Corrected reprint of the 1985 original.

\bibitem{Johnson}
S.~M. Johnson.
\newblock A new lower bound for coverings by rook domains.
\newblock {\em Utilitas Math.}, 1:121--140, 1972.

\bibitem{keri}
G.~K\'eri.
\newblock Website: \texttt{http://www.sztaki.hu/$\sim$keri/codes}.

\bibitem{Keri2}
G.~K{\'e}ri and P.~R.~J. {\"O}sterg{\aa}rd.
\newblock Bounds for covering codes over large alphabets.
\newblock {\em Des. Codes Cryptogr.}, 37(1):45--60, 2005.

\bibitem{orthogonalitygraph}
{\noopsort{Klerk de Pasechnik}}{E. de Klerk and D.V. Pasechnik}.
\newblock A note on the stability number of an orthogonality graph.
\newblock ArXiv:math.CO/0505038, May 2005.

\bibitem{Langlinalg}
S.~Lang.
\newblock {\em Linear algebra}.
\newblock Undergraduate Texts in Mathematics. Springer-Verlag, New York, third
  edition, 1989.

\bibitem{Monique}
M.~Laurent.
\newblock Strengthened semidefinite bounds for codes.
\newblock Januari 2005.

\bibitem{LiChenWen}
D.~F. Li and W.~D. Chen.
\newblock New lower bounds for binary covering codes.
\newblock {\em IEEE Trans. Inform. Theory}, 40(4):1122--1129, 1994.

\bibitem{Lint}
{\noopsort{Lint van}}{J. H. van Lint}.
\newblock {\em Introduction to coding theory}, volume~86 of {\em Graduate Texts
  in Mathematics}.
\newblock Springer-Verlag, Berlin, third edition, 1999.

\bibitem{lovasztheta}
L.~Lov{\'a}sz.
\newblock On the {S}hannon capacity of a graph.
\newblock {\em IEEE Trans. Inform. Theory}, 25(1):1--7, 1979.

\bibitem{matrixcuts}
L.~Lov{\'a}sz and A.~Schrijver.
\newblock Cones of matrices and set-functions and {$0$}-{$1$} optimization.
\newblock {\em SIAM J. Optim.}, 1(2):166--190, 1991.

\bibitem{Sloane}
F.~J. MacWilliams and N.~J.~A. Sloane.
\newblock {\em The theory of error-correcting codes. {I, II}}.
\newblock North-Holland Publishing Co., Amsterdam, 1977.
\newblock North-Holland Mathematical Library, Vol. 16.

\bibitem{nesterov}
Y.~Nesterov and A.~Nemirovskii.
\newblock {\em Interior-point polynomial algorithms in convex programming},
  volume~13 of {\em SIAM Studies in Applied Mathematics}.
\newblock Society for Industrial and Applied Mathematics (SIAM), Philadelphia,
  PA, 1994.

\bibitem{Roerdam}
M.~R{\o}rdam, F.~Larsen, and N.~Laustsen.
\newblock {\em An introduction to {$K$}-theory for {$C\sp *$}-algebras},
  volume~49 of {\em London Mathematical Society Student Texts}.
\newblock Cambridge University Press, Cambridge, 2000.

\bibitem{modifiedtheta}
A.~Schrijver.
\newblock A comparison of the {D}elsarte and {L}ov\'asz bounds.
\newblock {\em IEEE Trans. Inform. Theory}, 25(4):425--429, 1979.

\bibitem{CombOptB}
A.~Schrijver.
\newblock {\em Combinatorial optimization. {P}olyhedra and efficiency. {V}ol.
  {B}}, volume~24 of {\em Algorithms and Combinatorics}.
\newblock Springer-Verlag, Berlin, 2003.
\newblock Matroids, trees, stable sets, Chapters 39--69.

\bibitem{Lexcodes}
A.~Schrijver.
\newblock New code upper bounds from the terwilliger algebra.
\newblock {\em IEEE Trans. Inform. Theory}, To appear.

\bibitem{Terwilliger}
P.~Terwilliger.
\newblock The subconstituent algebra of an association scheme. {I}.
\newblock {\em J. Algebraic Combin.}, 1(4):363--388, 1992.

\bibitem{TerwilligerIII}
P.~Terwilliger.
\newblock The subconstituent algebra of an association scheme. {III}.
\newblock {\em J. Algebraic Combin.}, 2(2):177--210, 1993.

\bibitem{Todd}
M.~J. Todd.
\newblock Semidefinite optimization.
\newblock {\em Acta Numer.}, 10:515--560, 2001.

\bibitem{sdpt3}
K.~C. Toh, M.~J. Todd, and R.~H. T{\"u}t{\"u}nc{\"u}.
\newblock S{DPT}3---a {MATLAB} software package for semidefinite programming,
  version 1.3.
\newblock {\em Optim. Methods Softw.}, 11/12(1-4):545--581, 1999.
\newblock Interior point methods.

\bibitem{VanWee}
G.~J.~M. van Wee.
\newblock Improved sphere bounds on the covering radius of codes.
\newblock {\em IEEE Trans. Inform. Theory}, 34(2):237--245, 1988.

\bibitem{VanWeeThesis}
G.~J.~M. van Wee.
\newblock {\em Covering codes, perfect codes, and codes from algebraic curves}.
\newblock Technische Universiteit Eindhoven, Eindhoven, 1991.
\newblock Dissertation, Technische Universiteit Eindhoven, Eindhoven, 1991,
  With a Dutch summary.

\bibitem{wedderburn}
J.~H.~M. Wedderburn.
\newblock {\em Lectures on matrices}.
\newblock Dover Publications Inc., New York, 1964.

\bibitem{Zhang}
Z.~Zhang.
\newblock Linear inequalities for covering codes. {I}. {P}air covering
  inequalities.
\newblock {\em IEEE Trans. Inform. Theory}, 37(3, part 1):573--582, 1991.

\bibitem{ZhangLo}
Z.~Zhang and C.~Lo.
\newblock Linear inequalities for covering codes. {II}. {T}riple covering
  inequalities.
\newblock {\em IEEE Trans. Inform. Theory}, 38(6):1648--1662, 1992.

\end{thebibliography}

\cleardoublepage
\addcontentsline{toc}{chapter}{Index}
\printindex
\chapter*{Samenvatting}
\markboth{\textsl{SAMENVATTING}}{}
\addcontentsline{toc}{chapter}{Samenvatting}
Dit proefschrift gaat over foutcorrigerende codes en overdekkingscodes. Een code is een collectie woorden van dezelfde lengte $n$ met letters uit een alfabet $\mathbf{q}=\{0,1,\ldots, q-1\}$ bestaande uit een $q$-tal symbolen. In het geval dat $q=2$ bestaat elk woord uit een rijtje van $n$ nullen en enen. We spreken in dat geval van een binaire code. Voor $q\geq 3$ spreken we van niet-binaire codes. De Hamming afstand $d(\mathbf{x},\mathbf{y})$ tussen twee woorden $\mathbf{x}$ en $\mathbf{y}$ is gedefinieerd als het aantal posities waarin zij verschillen. Zo krijgt de verzameling $\mathbf{q}^n$ van alle woorden de structuur van een metrische ruimte, de Hamming ruimte. 

Een centrale vraag in de theorie van foutcorrigerende codes is:
\begin{quote}
Gegeven een `minimum afstand' $d$, wat is het maximale aantal woorden in een code als we eisen dat van elk tweetal woorden de onderlinge afstand ten minste $d$ moet zijn? 
\end{quote}
Dit maximum, aangegeven met $A_q(n,d)$ heeft een mooie `meetkundige' interpretatie in het geval dat $d=2e+1$ oneven is. Het getal $A_q(n,d)$ is dan precies het aantal bollen van straal $e$ dat binnen de Hammingruimte kan worden gestapeld. De foutcorrigerende eigenschappen van zo'n code volgen uit het feit dat wanneer een codewoord in hoogstens $e$ posities wordt gewijzigd, het originele woord weer terug wordt gevonden door het dichtsbijzijnde codewoord te nemen.

Een twee vraag, die duaal is aan de vorige, speelt een rol in onder andere datacompressie:
\begin{quote}
Gegeven een `overdekkings straal' $r$, wat is het minimale aantal woorden in een code als we eisen dat ieder woord afstand ten hoogste $r$ tot een woord in de code heeft?
\end{quote}
Dit minimum, aangegeven met $K_q(n,r)$ is het aantal bollen van straal $r$ dat nodig is om de hele Hamming ruimte te bedekken.

In het algemeen zijn de getallen $A_q(n,d)$ en $K_q(n,r)$ erg moeilijk te bepalen en slechts weinig waarden zijn bekend. Daarom is het interessant om goede onder- en bovengrenzen te vinden voor deze getallen. Het meetkundige beeld van het stapelen van en overdekken met bollen geeft al een bovengrens voor $A_q(n,2r+1)$ en een ondergrens voor $K_q(n,r)$ door het volume van de gehele Hamming ruimte te delen door het volume van een bol van straal $r$.

In dit proefschrift geven we nieuwe bovengrenzen voor $A_q(n,d)$ en nieuwe ondergrenzen voor $K_q(n,r)$ met behulp van semidefiniete programmering. Een belangrijke rol wordt gespeeld door een expliciete blokdiagonalisatie van de Terwilliger algebra van het Hamming schema. Deze maakt het mogelijk om de grote symmetriegroep van de Hamming ruimte te benutten, zowel voor het verkrijgen van scherpere grenzen, als voor het efficient kunnen bepalen van deze grenzen. De beschreven methode voor het begrenzen van $A_q(n,d)$ werd door Schrijver geintroduceerd voor het binaire geval in \cite{Lexcodes}. In hetzelfde artikel werd ook een blokdiagonalisatie gegeven voor de Terwilliger algebra van het binaire Hamming schema. Een centraal resultaat uit dit proefschrift is een expliciete blokdiagonalisatie van de Terwilliger algebra van het niet-binaire Hamming schema.

In hoofdstuk 2 brengen we de benodigde theorie in herinnering. In het bijzonder stippen we de krachtige methode van Delsarte \cite{Delsarte} aan, waarmee met behulp van associatieschemas bovengrenzen voor $A_q(n,d)$ te verkrijgen zijn door middel van lineaire programmering. Het idee is om te kijken naar de af\-stands\-ver\-de\-ling $1=x_0,x_1,\ldots,x_n$ van een code, waar $x_i$ het gemiddeld aantal codewoorden op afstand $i$ van een codewoord is. De getallen $x_i$ voldoen aan bepaalde lineaire on\-ge\-lijk\-heden. De eerste soort on\-ge\-lijk\-heden heeft een directe combinatorische betekenis: de getallen $x_i$ zijn niet-negatief en $x_i=0$ als er geen twee woorden zijn op afstand $i$. De andere on\-ge\-lijk\-heden, met coefficienten gegeven door de Krawtchouk polynomen, hebben een diepere betekenis. Zij weerspiegelen het feit dat de corresponderende lineaire combinatie $A:=x_0A_0+\cdots+x_n A_n$ van associatiematrices van het Hamming schema positief semidefiniet is. Dat het positief semidefiniet zijn van $A$ kan worden teruggebracht tot een $n+1$ tal lineaire on\-ge\-lijk\-heden, is het plezierige gevolg van het feit dat de Bose--Mesner algebra behorende bij het Hamming schema commutatief is, en daardoor in diagonaalvorm kan worden gebracht.

Een van de ideeen achter het onderhavige werk, is om naar de verdeling van \emph{drietallen} codewoorden te kijken in plaats van naar paren. Dit leidt tot de bestudering van een verfijning van de Bose--Mesner algebra in Hoofdstuk 3. De algebra bestaat uit alle matrices die invariant zijn onder die automorfismen van het Hamming schema $H(n,q)$, die een gekozen woord vasthouden. Er is een basis van $0$--$1$ matrices die geparametriseerd wordt door de mogelijke configuraties van drietallen woorden modulo automorfismen. We laten zien dat de algebra overeenkomt met de Terwilliger algebra \cite{Terwilliger} van het Hamming schema. Deze Terwilliger algebra is niet langer commutatief en kan daarom niet worden gediagonaliseerd. Het analogon voor niet-commutatieve algebras is een blokdiagonalisatie, waarbij de algebra bestaat uit alle matrices met gegeven blok-diagonaal structuur. Een centraal resultaat van dit proefschrift is een expliciete blokdiagonalisatie van de Terwilliger algebra behorende bij het niet-binaire Hamming schema. Hoewel het positief semidefiniet zijn van een matrix in de Terwilliger algebra niet langer kan worden geformuleerd door een klein aantal lineaire ongelijkheden, geeft de blokdiagonalisatie toch een handzame formulering in termen van het positief semidefinitiet zijn van een klein aantal kleine matrices (het aantal is $O(n^2)$ en de grootte $O(n)$).

In hoofdstuk 4 geven we een verscherping van de Delsarte grens voor codes. Met behulp van de expliciete  blokdiagonalisatie van de Terwilliger algebra uit hoofdstuk 3, kan deze grens efficient worden bepaald middels semidefinite programmering. Voor $q=3,4,5$ levert dit computationeel een reeks verscherpingen op voor bekende bovengrenzen voor $A_q(n,d)$.

In hoofdstuk 5 beschouwen we overdekkingscodes en geven we nieuwe ondergrenzen voor $K_q(n,r)$. Veel bestaande grenzen voor $K_q(n,r)$ zijn gebaseerd op de af\-stands\-ver\-de\-ling $A_0(\mathbf{x}),\ldots, A_n(\mathbf{x})$ van de code $C$ gezien vanuit een woord $\mathbf{x}$. Hier is $A_i(\mathbf{x})$ het aantal woorden in $C$ op afstand $i$ van $\mathbf{x}$. Iedere lineaire ongelijkheid in $A_0,\ldots, A_n$ die voor de af\-stands\-ver\-de\-ling vanuit ieder woord $\mathbf{x}$ geldt, geeft een ondergrens voor $K_q(n,r)$. De voor de hand liggende ongelijkheid
\begin{equation}
A_0+A_1+\cdots+A_r\geq 1
\end{equation}
leidt op deze manier tot dezelfde grens (`sphere covering bound') als het volume argument als boven. Vanuit polyhedraal oogpunt optimaliseren we een lineare functie over een een polytoop $P\subseteq [0,1]^{\mathbf{q}^n}$ binnen de eenheidskubus, met een groot aantal symmetrieen, namelijk de symmetrieen van de Hamming ruimte $\mathbf{q}^n$. Met behulp van de theorie van matrix snedes \cite{matrixcuts} kunnen we $P$ vervangen door een kleinere convexe verzameling, en daarmee scherpere grenzen voor $K_q(n,r)$ vinden. Om deze grenzen efficient te kunnen bepalen met lineaire en semidefiniete programmering, is wederom de blokdiagonalisatie van de Terwilliger algebra van het Hamming schema van groot belang. Computationeel levert deze methode voor $q=3$ en $q=4$ een aantal verscherpingen op ten opzichte van de ondergrenzen voor $K_q(n,r)$ uit de literatuur.

In hoofdstuk 6 brengen we deze theorie van matrix-snedes in herinnering en bestuderen we de relatie tussen de nieuwe grenzen voor $A_q(n,d)$ en deze theorie van matrix snedes. In het bijzonder blijkt dat de grenzen voor $A_q(n,d)$ uit hoofdstuk 4 scherper zijn dan die afkomstig van het toepassen van matrix-snedes op het `theta-body'. Dit is (vooral) te danken aan extra relaties die voortvloeien uit de aanwezige symmetrieen.   

\chapter*{Dankwoord}
\markboth{}{}
Met groot genoegen maak ik hier van de mogelijkheid gebruik om een ieder te bedanken die mij gedurende mijn promotietijd heeft gesteund, met dit proefschrift als resultaat. 

Zonder mijn promotor, Lex Schrijver, was dit proefschrift er niet geweest, en zou ik geen promotieonderzoek hebben gedaan in de combinatorische optimalisatie. Graag wil ik hem bedanken voor zijn continue steun en vertrouwen, ook wanneer ik dat laatste schier verloren had. Ik ben blij dat ik zo veel van hem heb kunnen leren.

Dit boekje heb ik opgedragen aan Violeta, mijn partner en moeder van onze twee kinderen Amber en Mark. Hoewel ik soms te weinig tijd voor hen vrij maakte, heeft Violeta mij altijd gesteund. Mark en Amber hebben mij altijd weer weten te verleiden om mijn werk even opzij te schuiven.

Veel heb ik ook te danken aan mijn ouders, en ook mijn schoonouders. Zij hebben mij op zoveel verschillende manieren beinvloed en ondersteund.

Voor de prettige werksfeer dank ik mijn collegas aan de UvA, in het bijzonder mijn kamergenoot Pia Pfluger. Ook aan het CWI heb ik het enorm getroffen met fijne collegas. Ik bedank Monique Laurent voor de gesprekken over het onderwerp van dit proefschrift. Ook wil ik mijn kamergenoot Gabor Maroti bedanken. Zonder hem was het vast niet gelukt om \emph{Theory of Integer and Linear Programming} zo grondig te bestuderen. Ook ging er geen keer voorbij dat hij geen \emph{nice puzzle} had. Ik vermoed dat zijn Hongaarse achtergrond hier deels verantwoordelijk voor was.

Graag wil ik ook Chris Zaal noemen. Ik herinner mij goed de SET-workshop die we aan het APS gegeven hebben, de opnamen van de Nationale Wetenschaps Quiz en diverse andere creativiteiten op het gebied van wiskunde-promotie. In het bijzonder heeft hij mij geintroduceerd bij Pythagoras, waar ik nog steeds de problemenrubriek mag vullen.

Ik dank ook Marco Zwaan. Via hem heb ik mogen proeven hoe het is om wiskundeleraren te onderwijzen aan de eerstegraads opleiding, een bijzondere ervaring.

Tenslotte dank ik de iedereen die ik niet met name heb genoemd, maar aan wie ik niettemin fijne herinneringen dank. Bedankt!
\addcontentsline{toc}{chapter}{Dankwoord}
 
\chapter*{Curriculum Vitae}
Dion Camilo Gijswijt is geboren op 20 maart 1978 te Bunschoten. Al vroeg had hij grote interesse in natuurkunde en andere exacte wetenschappen. Tijdens de eerste jaren van zijn middelbareschool-tijd groeide zijn interesse voor kunstmatige intelligentie en wiskunde. Het boek \emph{G\"odel, Escher, Bach} van D.R. Hofstadter had daarin een groot aandeel. Uiteindelijk gaf zijn deelname aan de International Wiskunde Olympiade in de laatste jaren van de middelbare school de doorslag om wiskunde te gaan studeren. 

In 1996 slaagde hij voor het eindexamen VWO aan het Goois Lyceum in Bussum en begon aan een dubbele studie wiskunde en natuurkunde aan de Universiteit van Amsterdam. Na in 1997 zijn propedeuses wiskunde en natuurkunde cum laude te hebben behaald, besloot hij zich te concentreren op de studie wiskunde. Tijdens zijn studie besteede hij een groot deel van zijn tijd aan het tijdschrift Pythagoras, dat juist onder de bezielende leiding van Chris Zaal nieuw leven was ingeblazen. Nadat hij in 1996 begon met het vullen van de problemenrubriek, werd hij in 1998 tevens redacteur.

In augustus 2001 behaalde hij zijn doctoraaldiploma (cum laude) na het schrijven van een scriptie getiteld \emph{The Colin de Verdi\`ere Graph Parameter $\mu$} onder begeleiding van Lex Schrijver. In september 2001 begon hij, eveneens onder supervisie van Lex Schrijver, aan zijn promotie-onderzoek aan het KdVI aan de Universiteit van Amsterdam. Dit onderzoek, met als thema \emph{Spectral Methods for Graph Optimization and Embedding}, mondde in juli 2005 uit in het onderhavige proefschrift.

Gedurende deze vier jaar als promovendus was hij ook in de gelegenheid andere activiteiten te ontplooien. Naast het met plezier begeleiden van diverse werkcolleges, heeft hij vooral veel geleerd van het doceren van een college grafentheorie aan eerstejaars wiskunde studenten, en later aan leraren aan de eerstegraads lerarenopleiding. Ook het begeleiden van twee studenten bij het schrijven van hun kandidaatsscriptie was een grote ervaring. 

Daarnaast heeft hij zich beziggehouden met activiteiten ter promotie van de wiskunde. Naast zijn werk voor Pythagoras, was hij actief als lid van de vraagstukkencommissie voor de Nederlandse Wiskunde Olympiade, en was betrokken bij diverse promotionele activiteiten waaronder: de Leve de Wiskunde dag, Nationale Wetenschapsdag, voorlichtingsdagen van de UvA, mastercourse voor wiskundeleraren en workshops voor scholieren en leraren.

\addcontentsline{toc}{chapter}{Curriculum Vitae}

\end{document}